\documentclass[12pt]{article}
\usepackage{amscd, amsmath}
\usepackage{amsthm, amssymb, mathrsfs}
\usepackage{mathtools}
\usepackage{sectsty}
\usepackage{slashed}
\usepackage{scalerel}
\usepackage[bottom,flushmargin]{footmisc}
\usepackage{mathrsfs}
\usepackage{float}
\usepackage{array}
\usepackage{accents}
\usepackage{empheq}
\usepackage[dvipsnames]{xcolor}
\usepackage[most]{tcolorbox}
\usepackage[british]{babel}
\usepackage{tikz-cd} 
\usepackage{hyperref}
\usepackage{tocloft}
\usepackage{titlesec}
\usepackage{cite} 
\usepackage{lmodern}
\usepackage{setspace}
\usepackage{enumitem}
\usepackage{appendix}
\usepackage{doi}

\titlelabel{\thetitle.\quad}
\titleformat*{\section}{\fontsize{16}{19}\bfseries\selectfont}
\titleformat*{\subsection}{\fontsize{13}{17}\bfseries\selectfont}
\subsubsectionfont{\normalfont\itshape}

\renewcommand{\baselinestretch}{1.2}
\newcommand{\linebr}{\\[0.2em]}

\DeclareMathAlphabet{\mathpzc}{OT1}{pzc}{m}{it}

\definecolor{myblue1}{RGB}{0, 0, 139}

\hypersetup{
    linktocpage,
    colorlinks,
    citecolor=myblue1,
    filecolor=black,
    linkcolor=myblue1,
    urlcolor=myblue1
}

\newtheorem{theorem}{Theorem}[section]
\newtheorem{lemma}[theorem]{Lemma}
\newtheorem{proposition}[theorem]{Proposition}
\newtheorem*{proposition*}{Proposition}
\newtheorem{definition}{Definition}[section]

\newtheorem{corollary}[theorem]{Corollary}

\newtheorem*{assumption(A)}{Condition (A)}

\theoremstyle{remark}
\newtheorem{remark}{{\bf Remark}}[section]

\numberwithin{equation}{section}

\newcommand{\Lie}{\mathcal L}

\newcommand{\HH}{\mathcal H}

\newcommand{\cL}{\mathscr{L}}

\newcommand{\SU}{\mathrm{SU}}
\newcommand{\su}{\mathfrak{su}}

\newcommand{\CC}{\mathbb{C}}

\newcommand{\dd}{{\mathrm d}}
\newcommand{\vol}{\mathrm{vol}}

\newcommand{\divergence}{\mathrm{div}}

\newcommand{\sD}{\slashed{D}}
\newcommand{\cD}{\mathcal{D}}

\newcommand{\beq}{\begin{equation}}
\newcommand{\eeq}{\end{equation}}
\newcommand{\bal}{\begin{align}}
\newcommand{\eal}{\end{align}}
\newcommand{\bmatr}{\begin{bmatrix}}
\newcommand{\ematr}{\end{bmatrix}}

\newcommand{\blangle}{\boldsymbol{\langle}}
\newcommand{\brangle}{{\boldsymbol{\rangle}}}

\newcommand{\ok}{\mathbin{\ast}}
\newcommand{\dok}{\mathbin{\scalebox{0.8}{$\circledast$}}}

 \newtcolorbox{empheqboxed}{ 
 opacityback=0,
 enhanced jigsaw,
 width=\textwidth,
 boxrule=.5pt,
 sharpish corners,
 left=0pt,
 right=2pt,
 top=-9pt, 
 bottom=3pt
}

\newcommand{\LL}{{\text{\scalebox{.9}{$L$}}}}
\newcommand{\RR}{{\text{\scalebox{.9}{$R$}}}}

\DeclareSymbolFont{TOneChars}{T1}{\familydefault}{m}{it}
\SetSymbolFont{TOneChars}{bold}{T1}{\familydefault}{bx}{it} 
\DeclareMathSymbol{\mathdh}{\mathord}{TOneChars}{"F0}

\makeatletter
\def\blfootnote{\xdef\@thefnmark{}\@footnotetext}
\makeatother

\begin{document}

\begin{titlepage}

\title{{\bf Triality rotation of connections \\ and Dirac operators on SU(3)}}

\author{João Baptista}
\date{\vspace{-.3cm}October 2026}

\maketitle

\thispagestyle{empty}
\vskip 10pt
{\centerline{{\large \bf{Abstract}}}}
\noindent
We show how triality, the symmetry permuting the three $8$-dimensional representations of $\mathrm{Spin}(8)$, can also appear as a symmetry of geometric objects such as connections and Dirac operators. Take any Riemannian metric $g$ on the manifold $M = \mathrm{SU}(3)$ and let $\nabla$ be a metric connection on $TM$. Using an explicit model of Clifford multiplication by deformed Okubo products, we combine spinorial lifts with isometries $TM \simeq S^+ \simeq S^-$ to show that $\nabla$ can be rotated into two new metric connections on $TM$. So there is a canonical $\mathbb{Z}_3$-orbit of metric connections determined by $(g, \nabla)$, with no other choices. The three connections coincide when $g$ is bi-invariant and $\nabla$ belongs to the Cartan-Schouten line, but in general they are different. The rotations of $\nabla$ preserve the norm of the curvature, so they define a new $\mathbb{Z}_3$-symmetry of the moduli space of Yang-Mills connections on $TM$. They also preserve the Chern-Simons 3-form and adjoint Yang-Mills-Higgs functionals. After choosing an orientation on $M$, the rotated connections define three self-adjoint, Dirac-type operators on the spinor bundle. We prove that these operators are natural, in the sense that they commute with the Kosmann derivatives of spinors along all invariant Killing vector fields that preserve $\nabla$. We also show that they are not isospectral, in general. This is verified by computing the simplest eigenvalue for the Levi-Civita connection and observing that it splits among the three operators as the bi-invariant metric on $M$ is deformed to one with isometry group $(\mathrm{SU}(3) \times \mathrm{SU}(2) \times \mathrm{U}(1)) / \mathbb{Z}_6$. Finally, we outline how the constructions in the paper may extend to general 8-manifolds with a $\mathrm{PSU}(3)$-structure.

\vspace{-1.0cm}
\let\thefootnote\relax\footnote{ 
\noindent
{\small {\sl \bf  Keywords:}  Connections with torsion; triality; Okubo product; Dirac operators; Yang-Mills-Higgs moduli spaces; $\mathrm{PSU}(3)$-structures. \vspace{0.15cm}
}
}

\addtocounter{footnote}{-1}

\end{titlepage}

\pagenumbering{roman}

\renewcommand{\baselinestretch}{1}\normalsize
\tableofcontents
\renewcommand{\baselinestretch}{1.19}\normalsize

\newpage

\pagenumbering{arabic}

\section{Introduction and overview of results}
\label{Introduction}

In representation theory, triality appears as an exceptional symmetry between the three 8-dimensional representations of $\mathrm{Spin}(8)$ \cite{Sapio, Baez, Harvey}. The motivating question for this paper is whether, on suitable 8-manifolds, triality can also appear as a symmetry between natural geometric objects, such as connections, Higgs fields and Dirac operators.

The  starting point of the study is the observation made by Hitchin in \cite{Hitchin} that, due to triality, the tangent bundle of the manifold $M= \SU(3)$ is isomorphic to both chiral spinor bundles, and that, moreover, Clifford multiplication of vectors and spinors can be modelled by an orthogonal product of vectors $\su(3) \times \su(3) \rightarrow \su(3)$.

The observation is not just that $TM$, $S^+$ and $S^-$ are isomorphic as smooth, rank-8, trivial vector bundles, but also that they have the same structure group $\mathrm{PSU}(3)$. The adjoint action on the Lie algebra $\su(3)$ equipped with the Ad-invariant inner product cuts out a subgroup $\mathrm{PSU}(3) \subset \mathrm{SO}(8)$, which then lifts to a subgroup $\mathrm{PSU}(3) \subset \mathrm{Spin}(8)$, whose two restricted spin representations again coincide with the adjoint one \cite{Hitchin, Witt, Witt2}.

This means that every connection on $TM$ compatible with the Ad-invariant inner product can be transported by the isomorphisms to connections on $S^+$ and $S^-$ compatible with the spinor inner product, and vice-versa. However, it is not obvious whether this transport coincides with the traditional spinorial lifts of connections with torsion. That coincidence is especially uncertain when the metric on $TM$ is deformed away from the bi-invariant one. It should be possible to use the techniques of Wang, Bourguignon and Gauduchon \cite{Wang, BG} to track what happens to the $\su(3) \times \su(3) \rightarrow \su(3)$ model of Clifford multiplication under that deformation. Therefore, the plan of the paper is to:
\begin{itemize}
\item[1)]  Investigate how the bi-invariant picture of \cite{Hitchin, Witt, Witt2} can be extended to an explicit model of Clifford multiplication for general metrics $g$ on $\SU(3)$;
\item[2)] Understand how the isomorphisms $TM \simeq S^+ \simeq S^-$ transport metric connections, and compare it to the standard procedure of spinorial lifts.
\item[3)] Study how gauge transformations interact with the triality transport of connections, and whether that transport descends to a map of moduli spaces.
\item[4)] Investigate the three Dirac operators of the triality-rotated connections, and find out if they are always isospectral. 
\end{itemize}
At the end, we find that the representation-theoretic triality can indeed be implemented as a symmetry of geometric objects. For example, as a cyclic group of automorphisms of the moduli spaces of connections and adjoint Higgs fields. In spin geometry, it relates three distinguished Dirac operators, each with its own spectrum. These spectra move around when the metric and connection are deformed, but collapse to a single spectrum at the bi-invariant, Cartan–Schouten points, where everything is algebraic again and all eigenvalues come from Casimirs \cite{Slebarski, Fegan, LT}. 

The study also has a physical motivation related to the Kaluza-Klein models in \cite{Bap1, Bap3}. There, the manifold $\SU(3)$ is the internal space and the Dirac eigenvalues are fermion masses. Separately, triality and the octonionic algebras have also been proposed as an algebraic origin of the three generations of fermions (e.g. \cite{DVT, FureyHughes, Boyle}).

We now give an overview of the main results and techniques in the paper. After that, we briefly note how these constructions may extend to more general 8-manifolds, such as manifolds with a $\mathrm{PSU}(3)$-structure, or manifolds whose spin bundle reduces to $G_2 \subset \mathrm{Spin}(8)$, the stabiliser of a pair of unit spinors of opposite chirality.

\subsection*{Overview of the main results}

\subsubsection*{Bundle Okubo products and Clifford multiplication}

Section \ref{BasicPropertiesOkubo} establishes notation and recalls the fundamental properties of the Okubo product $\ok$. It is an orthogonal product on $\su(3)$ equipped with the Ad-invariant metric $\langle , \rangle$ which satisfies nice Clifford-like relations \cite{Okubo78, OkuboBook, Elduque}. Appendix \ref{AppendixOkuboIdentities} derives several additional algebraic identities that are needed later in the paper.

Section \ref{FibrewiseOkubo} describes how the tangent bundle $TM$ and the real chiral spinor bundles $S^\pm$ can all be identified with $M \times \su(3)$. 
It explains how the Okubo product can be lifted to a set of bilinear bundle maps $\dok$ between $TM$, $S^+$ and $S^-$, and how those maps can be deformed to model Clifford multiplication for an arbitrary metric $g$ on $TM$, instead of only the bi-invariant metric induced by $\langle , \rangle$. The existence of a Clifford module structure on $M \times (\su(3) \oplus \su(3))$ is what justifies its identification with the spinor bundle $S_g = S^+ \oplus S^-$.

\subsubsection*{Triality-transported connections}

Sections \ref{ThreeNaturalConnections} and \ref{Section3ConnectionsSpinorialLifts} describe how a single metric connection $\nabla$ on $(TM, g)$ determines three different connections, $\nabla^1$, $\nabla^2$ and $\nabla^3$ on each of the bundles $TM$ and $S^\pm$. For example, the new connections on $TM$ are obtained by using natural isometries $TM \simeq S^+ \simeq S^-$ to transport back to $TM$ the spinorial lifts of $\nabla$ on $S^+$ and $S^-$. This procedure returns the original connection when $g$ is bi-invariant and $\nabla$ is in the Cartan-Schouten line, but in general produces three distinct connections on $TM$. The precise conditions under which the triality-rotated connections are distinct from each other are given in theorem \ref{PropertiesTrialityRotationConnections}.

Another way to represent the new connections is through the condensed formulas 
\begin{equation}
\nabla^r_X \, Y  \: = \:  \partial_X Y \; +\;  \rho^{r-1}_{A_X} (Y)    \ ,
\end{equation}
where $A_X : M \rightarrow \mathfrak{so} (8)$ is the map defined by the entries $(A_X)_{jk} \simeq  g(\nabla_{X} V_j , V_k)$ in a $g$-orthonormal trivialization of $TM$. The symbols $\rho^{r-1}_A$ denote the three inequivalent representations of $\mathfrak{so} (8)$, which themselves are related by triality \cite{Sapio, Baez, Harvey}. This works because of the $\mathrm{PSU}(3)$-equivariant identification of the half spinor spaces with $\su(3)$. 

One can also ask whether the spinorial lifts to $S^\pm$ of the new connections on $TM$ coincide with their triality-transported images. The relation between these two procedures to obtain new connections is the content of theorem \ref{SpinorialLiftsVSTrialityTransports}. Another algebraic relation between the $\nabla^r$ is given in lemma \ref{Sum3Reps}.
Section \ref{SectionCurvatureTorsion} gives explicit formulas for the curvature tensor $R^{\nabla^r}$ of the new connections on $TM$ in terms of the curvature of $\nabla$ and Okubo products. The same for the Ricci scalars $\mathrm{Scal}^r$. The rotated connections $\nabla^r$ will in general have torsion. Section \ref{SectionCurvatureTorsion} gives explicit formulas for the torsion traces $\tau_r$.

\subsubsection*{Triality rotations as symmetries of moduli spaces}

Section \ref{SymmetriesModuliSpaces} observes that the triality rotations of $\nabla$ can be regarded as its $\mathbb{Z}_3$-orbit under a cyclic transformation ${\mathcal T}: \mathscr{A}_g \rightarrow \mathscr{A}_g$ of order 3, where $\mathscr{A}_g$ denotes the space of metric connections on $TM$. Theorem \ref{PropertiesTrialityRotationConnections} shows that ${\mathcal T}$ descends to a map of moduli spaces $\mathscr{A}_g / \mathscr{G} \rightarrow \mathscr{A}_g / \mathscr{G}$. Proposition \ref{PropInvarianceNormCurvartureChernSimons} shows that ${\mathcal T}$ preserves the Chern-Simons 3-form and the norm of curvature, $|R^{{\mathcal T} \nabla}|_g^2 =|R^{\nabla}|_g^2$. Hence ${\mathcal T}$ defines a new 3-fold symmetry of the moduli space of Yang-Mills connections on $(TM, g)$.

Section \ref{SectionYangMillsHiggs} studies the space of $g$-skew endomorphims of $TM$, denoted $\HH_g$. This is a space of Higgs fields $\phi$ in the adjoint representation of $\mathrm{SO}(8)$. We define a natural action of triality rotations on $\HH_g$. It preserves the natural inner product on $\HH_g$ and satisfies the compatibility relation ${\mathcal T} (\nabla_X \phi) = ( {\mathcal T}\nabla)_X ({\mathcal T}\phi)$. Theorem \ref{TheorTrialityHiggsModuliSpace} shows that the simultaneous triality rotation of connections and Higgs fields descends to a map of moduli spaces $(\mathscr{A}_g \times \HH_g) / \mathscr{G} \rightarrow (\mathscr{A}_g \times \HH_g )/ \mathscr{G}$. These moduli spaces have natural Yang-Mills-Higgs functionals that are preserved by ${\mathcal T}$. Hence triality-rotations are also symmetries of the space of solutions of the Euler-Lagrange equations of those functionals.

\subsubsection*{Triality-transported Dirac operators}

The spinorial lifts of the connections $\nabla^r$ on $TM$ can be used to define three distinct Dirac operators on $S$. Since the $\nabla^r$ in general have non-zero trace torsion, their classical Dirac operators will not be formally self-adjoint \cite{FS79, Agricola}. 

To define self-adjoint operators, in section \ref{ThreeDiracOperators} we choose a particular chiral bundle, say $S^+$, which is equivalent to choosing an orientation on $M$, and start instead from the connections $\nabla^r$ on $S^+$. We define three elliptic operators $\Gamma (S^+) \rightarrow \Gamma (S^-)$  by the usual expressions $\cD_r \psi = \sum_k V_k \cdot \nabla^r_{V_k} \psi$ and calculate their $L^2$-adjoints $\cD_r^\dag$. The self-adjoint Dirac operator $\sD_r$ on $S = S^+ \oplus S^-$ is then the anti-diagonal operator built from $\cD_r$ and $\cD_r^\dag$. In other words, it is the self-adjoint completion of $\cD_r$.  

Expanding $\sD_r$ as a sum of products of gamma matrices, in section \ref{ThreeDiracOperators} we obtain 
\begin{equation*}
\sD_r \psi \; = \; \sum_j  \gamma_j \big(  \, \partial_{V_j} \, +  \, \tfrac{1}{2}\, (\divergence_g V_j)  \, \big)   \psi\: + \:  \tfrac{1}{4}\, \sum_{j \neq k\neq l}  \Gamma^r_{jkl} \, \gamma_j  \gamma_k  \gamma_l \, \psi   \: - \: \tfrac{1}{2}\, \tau_r \cdot \Gamma_S\,  \psi  \ ,     
\end{equation*}
where $\Gamma_S$ is the chiral operator, the $\tau_r$ are the trace torsion vector fields of $\nabla^r$, and we have abbreviated $\Gamma^r_{jkl} = g(\nabla^r_{V_j} V_k , V_l)$. The vectorial terms with one gamma matrix coincide for all three connections $\nabla^r$. The axial terms with seven gamma matrices do not. In section \ref{SectionExplicitExample} we calculate that the $\tau_r$ can indeed be non-zero, even when $\nabla$ is chosen to be the Levi-Civita connection.

\subsubsection*{Triality rotations preserve symmetries}

The construction of the rotated connections $\nabla^r$ and Dirac operators $\sD_r$ relies on the choice of a pair $(g, \nabla)$ and an orientation on the group $\SU(3)$. One expects that every symmetry that preserves these inputs will also preserve the three connections and Dirac operators. This is confirmed in section \ref{KLDerivatives}. 
Theorems \ref{TrialitiyLieDerivativeConnection} and \ref{TrialitiyKLLieDerivatives} prove that if $X$ is a left- or right-invariant Killing vector field such that the Lie derivative $\Lie_X \nabla$ vanishes, then $\Lie_X \nabla^r = 0$ as well, and the Kosmann-Lichnerowicz derivative $\cL_X$ of spinors will satisfy
\begin{equation}
\label{CommutationKLDerivativeNablar}
\cL_X ( \nabla^r_Y \psi^+ ) \: = \:  \nabla^r_{[X, Y]} \psi^+ \: + \:  \nabla^r_Y ( \cL_X \psi^+) \  
\end{equation}
for $r=1, 2, 3$. Those theorems do more: they also describe what happens when $\Lie_X \nabla$ does not vanish. A consequence of \eqref{CommutationKLDerivativeNablar} is that the derivatives $\cL_X$ will also commute with the operators $\sD_r$ (proposition \ref{PropDiracCommutesWithKosmann}).

\subsubsection*{An explicit non-isospectral example}

Section \ref{SectionExplicitExample} shows that the three operators $\sD_r$ are not in general isospectral when $g$ is not bi-invariant. It considers a simple family of left-invariant metrics on $M$ obtained by rescaling the circle generated by a vector $v_8 \in \su(3)$. These metrics have scalar curvature $R_g =  \frac{3}{ 4\, a^2}  \, ( 9  -  \mu^{2})$, where $a, \mu \in \mathbb{R}^+$ are the parameters of the initial bi-invariant metric and of the circle's rescaling, respectively. The connected component of the isometry group is now $(\SU(3) \times \SU(2) \times \mathrm{U}(1)) / \mathbb{Z}_6$, the gauge group of the Standard Model \cite{Bap1}. Choosing $\nabla$ to be the Levi-Civita connection, section \ref{SectionExplicitExample} then exhibits a spinor $\psi_0$ such that  
\[
\sD_1 \psi_0 \; = \;  -\, \frac{3}{2a} \, \psi_0   \qquad \quad  \ \   \sD_2 \psi_0 \; = \;  -\, \frac{3}{4\, a} \, ( 1 + \mu ) \, \psi_0    \qquad \quad \ \  \sD_3 \psi_0 \; = \;  -\, \frac{3 \, \mu}{2\, a} \, \psi_0    \ .
\]
Observe that the eigenvalue of $\sD_3$ is arbitrarily small as $\mu \rightarrow 0^+$ with fixed $a$, while the scalar curvature tends to a positive constant. The eigenvalue of $\sD_2$ also tends to a finite constant below Friedrich's bound \cite{Friedrich}. So $\sD_2$ and $\sD_3$ cannot be isospectral to the classical Dirac operator $\sD_1$. 
Section \ref{SectionExplicitExample} and appendix \ref{AppendixExplicitExample} also calculate that the trace torsion vector fields of $\nabla^r$ are identified with $\tau_2  =  -\, \tau_3   =  \tfrac{3}{4\, a} \, (1-\mu)  \,  v_8$ in this example. Hence they are non-zero when $\mu \neq 1$ and $g$ is not bi-invariant.

\subsection*{Extensions to other manifolds besides $\SU(3)$}

The constructions in this paper may extend to general $8$-dimensional manifolds carrying a topological $\mathrm{PSU}(3)$-structure in the sense of Hitchin and Witt \cite{Hitchin, Witt, Witt2}, whose intrinsic torsion was studied in \cite{Puhle}. Such a structure is a reduction $P$ of the frame bundle of $M$ to the subgroup $\mathrm{Ad}_{\SU(3)} = \mathrm{PSU}(3) \subset \mathrm{SO}(8)$. It determines a metric $g$, an orientation, and identifies $TM \simeq P \times_{\mathrm{Ad}} \su(3)$. Since $\mathrm{PSU}(3)$ lifts canonically to $\mathrm{Spin}(8)$, the reduction also determines a spin structure and chiral spinor bundles $S^\pm \simeq P \times_{\mathrm{PSU}(3)} \Delta^\pm$. The two half-spin representations restricted to this $\mathrm{PSU}(3)$ are equivalent to the adjoint one \cite{Hitchin, Witt, Witt2}. So the $\mathrm{Ad}$-equivariant isometries $\Delta^\pm \simeq \su(3)$ used for $M=\SU(3)$ determine global bundle isometries ${\mathcal I}^\pm : TM \rightarrow S^\pm$, unique up to sign. The three connections $\nabla^r$ on $TM$ and $S^\pm$ can then be defined  by \eqref{DefinitionThreeConnections} and \eqref{DefinitionThreeConnections3}, from the spinorial lifts of the metric connection $\nabla$ on $(TM, g)$. Since the Okubo product is $\mathrm{Ad}$-invariant, it also descends to fibrewise products $\dok$ between $TM$, $S^+$ and $S^-$ that model Clifford multiplication, as in section \ref{FibrewiseOkubo}.

It would also be interesting to investigate whether the constructions extend to 8-manifolds whose spin bundle reduces to $G_2 \subset \mathrm{Spin}(8)$, the stabiliser of a pair of unit spinors of opposite chirality. Restricted to such a $G_2$, the vector and the two spin representations again coincide \cite[Prop. 1.4]{Witt}. So $TM$, $S^+$ and $S^-$ are associated with a common principal bundle and are isomorphic again. In this case, the para-octonionic product $x\ast y = \bar{x} \bar{y}$ defines a symmetric composition algebra that may effectively replace the Okubo one \cite{Elduque, OkuboBook}.

Finally, triality rotation of connections on $M$ may also extend to higher-dimensional Riemannian submersions with 8-dimensional fibre $M$. One could define the triality rotation of a higher-dimensional connection by rotating its vertical projection while leaving unchanged its horizontal and mixed components.

\section{An explicit model of the spinor bundle}

\subsection{Properties of the Okubo product on $\su(3)$}
\label{BasicPropertiesOkubo}

The Lie algebra $\su (3)$ has a two-dimensional space of ${\mathrm{Ad}}_{\,\SU (3)}$-equivariant, bilinear, algebraic products. All such products are linear combinations of the commutator $[ u, v ]$ and of the traceless part of the anticommutator, 
\begin{align}
i \{ u, v  \}_0 \; &:= \;  i \big( uv \, + \,  vu  \,-  \, \tfrac{2}{3} \,  {\mathrm {Tr}} (u v) \, I_3  \big) \ ,
\end{align}
where $u, v \in \su(3)$. The Okubo product can be defined as the combination
\begin{equation}
\label{DefinitionOkuboProduct}
u \ok v \; := \;  a \, \big( [u, v] +  \sqrt{3}\, i  \{ u, v  \}_0  \big)
\end{equation}
for a fixed constant $a \in \mathbb{R}^+$. Considering the ${\mathrm{Ad}}_{\SU (3)}$-invariant inner product on $\su(3)$
\begin{equation}
\label{AdInvariantProduct}
\langle u ,  v \rangle \; := \; 2 \, a^2 \, {\mathrm {Tr}} (u^\dag v) \ ,
\end{equation}
the Okubo product satisfies the notable identities
\begin{align}
[u, v \ok w]  \; &= \;  [u, v] \ok w \: + \: v \ok [u, w]  & \qquad \qquad & \text{(Ad-equivariance)}   \label{AdInvariantOkubo}  \linebr
\langle u \ok v,  u \ok w \rangle  \; &= \;  \langle u ,  u \rangle \,  \langle v ,  w \rangle & \qquad \qquad & \text{(orthogonality)}   \label{OrthogonalityOkubo}  \linebr
\langle u \ok v, w \rangle \;& = \; \langle u ,  v \ok w \rangle \; = \; \langle v,  w \ok u \rangle   & \qquad \qquad & \text{(adjunction)} \ ,  \label{AdjunctionOkubo}
\end{align}
and the Clifford-like relations
\begin{equation}
\label{CliffordOkubo}
( u \ok v) \ok w  +  ( w \ok v) \ok u   \; = \; 2\, \langle u , w \rangle  \, v   \; = \;  u \ok ( v \ok w ) +   w \ok ( v \ok u )   
\end{equation}
for all $u, v, w \in \su(3)$. See for instance the book by Okubo \cite{OkuboBook}, sections 4.2 and 4.3.

A longer calculation described in appendix \ref{AppendixOkuboIdentities} gives the associator of the $\ok$-product:
\begin{equation}
\label{AssociatorOkubo}
( u \ok v) \ok w  -   u \ok (v \ok w)   \; = \; 4\, a^2  \, [[u,w], v] \, - \, 2   \langle u , v \rangle  \, w  \, + \,  2   \langle v , w \rangle  \, u  \ .
\end{equation}
Finally, the next identities are important for us. They are proved in appendix \ref{AppendixOkuboIdentities}.
\begin{proposition}
\label{PropositionOkuboIdentities}
Let $\{v_j\}$ be a $\langle , \rangle$-orthonormal basis of $\su(3)$. Then for all $u, w \in \su(3)$,
\begin{align}
\sum\nolimits_j \, \big([v_j, u] \ok w \big) \ok v_j \; &= \;  \sum\nolimits_j \,   v_j \ok  \big(  w \ok [v_j, u]  \big) \; = \;  4\, [u, w]    \label{ImportantIdentityOkubo1}  \linebr
\sum\nolimits_j    \big(u \ok [w, v_j] \big) \ok v_j  \; &= \;  \sum\nolimits_j  \,  v_j \ok \big([v_j, u]  \ok w \big) \; = \; 4\, [u, w]   \: - \:  \frac{3}{a} \, u \ok w  \label{ImportantIdentityOkubo2}  \ .
\end{align}
\end{proposition}

\subsection{Bundle Okubo products and Clifford multiplication}
\label{FibrewiseOkubo}

Let $g$ be a fixed Riemannian metric on the tangent bundle $TM$. The derivative of left multiplication on the group can be used to identify $TM$ with the product $M \times \su(3)$. Then a vector field is viewed as a map $M \rightarrow \su(3)$. The metric $g$ is a map from $M$ to the space of inner products on $\su(3)$. Now define two additional trivial vector bundles: 
\begin{equation}
\label{DefinitionSpinorBundles}
S^+ := M \times \su(3)  \qquad \qquad  \quad S^- := M \times \su(3) \ . 
\end{equation} 
Sections of $S^\pm$ can also be viewed as maps $M \rightarrow \su(3)$. The inner product on the real fibres of $S^\pm$ is taken to be the Ad-invariant product $\langle , \rangle$ on $\su(3)$. 

The aim of this section is to lift the Okubo product on $\su(3)$ to a set of bilinear maps between the bundles $TM$, $S^+$ and $S^-$. Those lifted maps need to be deformed in order to model Clifford multiplication for an arbitrary metric $g$. We will show that:
\begin{proposition}
There exists a set of three bilinear bundle maps 
\[
TM \times S^+ \rightarrow S^-  \quad  \qquad  TM \times S^- \rightarrow S^+ \quad  \qquad  S^+ \times S^- \rightarrow TM  \ ,
\] 
all denoted $\dok$, such that for every $v \in TM$ and every $\psi^\pm \in S^\pm $,
\begin{align}
\label{GeneralizedCliffordRelations}
u \dok (v \dok \psi^\pm) \, +  \, v \dok (u \dok \psi^\pm)    \; &= \;  -\, 2\, g(u, v) \, \psi^\pm      \linebr
\psi^+ \dok (v \dok \phi^+) \, + \, \phi^+ \dok (v \dok \psi^+)    \; &= \;  2\, \langle \psi^+, \phi^+ \rangle \, v    \nonumber     \linebr
(\psi^+ \dok \psi^-) \dok \phi^+ \, + \,  (\phi^+ \dok \psi^-) \dok \psi^+      \; &= \;   2\, \langle \psi^+, \phi^+ \rangle \, \psi^-    \nonumber    \linebr
(v \dok \psi^-) \dok \phi^-  \, +  \,  (v \dok \phi^-) \dok \psi^-     \; &= \;  -\, 2\, \langle \psi^-, \phi^- \rangle \, v     \nonumber    \linebr
(\psi^+ \dok \psi^-) \dok \phi^- \, + \,  (\psi^+ \dok \phi^-) \dok \psi^-      \; &= \;  -\, 2\, \langle \psi^-, \phi^- \rangle \, \psi^+    \ . \nonumber  
\end{align}
\end{proposition}
The first identity in \eqref{GeneralizedCliffordRelations} says that the $\dok$-action of $TM$ on $S^+\oplus S^-$ is a model of Clifford multiplication on the manifold $(M, g)$. So the vector bundle $S^+\oplus S^-$ is a 16-dimensional real Clifford module, which justifies the interpretation of $S^\pm$ as chiral spinor bundles. And because $M$ is simply connected, its spin structure is unique, so our constructed spinor bundle is the only one. 

The spinor bundles $S^\pm$ have additional structure in our case $M = \SU(3)$. It is now possible to multiply two spinors to obtain a vector. The last four identities in \eqref{GeneralizedCliffordRelations} involve multiplications of that kind. Note that, on the right-hand side of those identities, the inner product of spinors is always $\langle , \rangle$. The $\dok$-product reflects a deformed metric $g$ on $TM$ but keeps unchanged the bi-invariant inner products on the chiral spinor bundles.

\subsubsection*{Construction and properties of the bundle maps $\dok$}

Consider the fixed Riemannian metric $g$. There exists a unique smooth automorphism $\alpha: TM \rightarrow TM$ that projects to the identity on $M$, is positive-definite, and satisfies 
\begin{equation}
\label{DefinitionAlpha}
g(u, v) \; = \; \langle \alpha^{-1} (u), \, \alpha^{-1} (v) \rangle  \qquad  \qquad \langle \alpha (u),  v \rangle \; = \; \langle u, \alpha (v) \rangle    
\end{equation}
for all vectors $u, v \in TM$. This automorphism is standard when identifying spinors of two different metrics (see \cite{Wang, BG, BGM, AHermann} and the appendices of \cite{Bap4}). Here we use it to deform the Okubo product so that the Clifford-like relations \eqref{CliffordOkubo} hold for the metric $g$, instead of the bi-invariant metric $\langle , \rangle$. 

Identifying $TM$ and $S^\pm$ with $M \times \su (3)$, there are obvious isometries between $(TM, g)$ and $(S^\pm, \langle , \rangle)$. Simply define
\begin{equation}
\label{DefinitionIsometry1}
{\mathcal I}^\pm : TM \longrightarrow S^\pm \ , \qquad (x, v)  \mapsto (x, \alpha_x^{-1} (v)) \ .
\end{equation}
Then we clearly have 
\begin{equation}
\label{Isometry1}
\langle {\mathcal I}^\pm(u) , {\mathcal I}^\pm(v) \rangle  = g( u , v) \ .
\end{equation}
The composition
\begin{equation}
\label{DefinitionIsometry2}
{\mathcal I}^+ \circ ({\mathcal I}^-)^{-1} : S^- \longrightarrow S^+ \ , \qquad (x, v)  \mapsto (x, \alpha_x^{-1} \circ \alpha_x (v)) = (x, v)
\end{equation}
is also a manifest isometry between $(S^-, \langle , \rangle)$ and $(S^+, \langle , \rangle)$.

Now identify $TM$ and $S^\pm$ with $M \times \su (3)$ and consider the Okubo product $\ok$ on $\su(3)$. Define three bilinear multiplications on the fibres of the vector bundles by the formulas
\begin{align}
\label{DeformedCliffordMultiplication}
TM \times S^+ \rightarrow S^- \ &, \qquad v \dok \psi^+ \; := \;  ( \alpha^{-1} v) \ok \psi^+  \linebr
TM \times S^- \rightarrow S^+ \ &,  \qquad  v \dok \psi^- \; := \;  -   \;  \psi^- \ok  ( \alpha^{-1} v)      \nonumber  \linebr
S^+ \times S^- \rightarrow TM \ &,  \qquad \psi^+ \dok \psi^-   \; := \;  \alpha ( \psi^+ \ok \psi^-   )   \ .  \nonumber 
\end{align} 
On the right-hand sides, the elements $v$, $\psi^+$ and $\psi^-$ should all be regarded as having values in $\su(3)$, so they can be multiplied with the Okubo product and transformed with the automorphisms $\alpha$. These multiplication rules were first written down in the bi-invariant case $\alpha = 1$, in which $\dok$ is essentially the same as $\ok$. For the general case, the isometry \eqref{DefinitionIsometry1} suggests that all the vectors in the formulas, namely $v$ and $\psi^+ \dok \psi^-$, should be transformed by $\alpha^{-1}$.

Using the definitions  \eqref{DeformedCliffordMultiplication} of the $\dok$-products, the properties of the Okubo product, and the definition of $\alpha$, one then calculates that
\begin{equation}
\label{PropertiesDokProduct}
g( \psi^+ \dok \psi^- , v)  \; = \;   \langle v \dok \psi^+ , \psi^- \rangle   \; = \;  - \; \langle \psi^+ ,  v \dok \psi^- \rangle  \ .
\end{equation}
Longer but straightforward calculations verify all the identities \eqref{GeneralizedCliffordRelations} of the proposition. These need to be checked only in the simpler case $u=v$ and $\psi^\pm = \phi^\pm$, as the general case follows from polarization. In particular, traditional Clifford multiplication $\cdot$ of vectors and spinors can be written as
\begin{equation}
\label{TraditionalCliffordMultiplication}
v \cdot  \begin{bmatrix} \psi^+ \\ \psi^- \end{bmatrix} \ = \  \begin{bmatrix}  v\dok \psi^- \\  v \dok \psi^+ \end{bmatrix} \ = \  \begin{bmatrix}  - \,  \psi^- \ok  ( \alpha^{-1} v)  \\  ( \alpha^{-1} v) \ok \psi^+ \end{bmatrix} \ .
\end{equation}
From \eqref{GeneralizedCliffordRelations}, it is clear that indeed $v \cdot v \cdot  \psi^\pm = \ -\, g(v, v)   \psi^\pm$.

One can also check that $\dok$-products preserve the orthogonality relation \eqref{OrthogonalityOkubo} of the Okubo product, after the deformation of $\langle , \rangle$ into $g$ on $TM$. Specifically, we have 
\begin{align}
\label{GeneralizedOrthogonalityRelations2}
\langle u \dok \psi^\pm, \, u \dok \psi^\pm \rangle  \; &= \; \, g(u, u) \;  \langle \psi^\pm, \, \psi^\pm \rangle      \linebr
g( \psi^+ \dok \psi^-, \, \psi^+ \dok \psi^- ) \; &= \;  \langle \psi^+, \, \psi^+ \rangle \;  \langle \psi^-, \, \psi^- \rangle   \ .
\end{align} 
So the bilinear bundle maps $\dok$ respect all the inner products in $TM$, $S^+$ and $S^-$.

\section{Three natural connections on $S^+$}
\label{ThreeNaturalConnections}

\subsection{Compatible connections on $TM$, $S^+$ and $S^-$ }

Let $\nabla$ be a fixed metric connection on the Riemannian manifold $(M, g)$. Let $\{ v_j\}$ be an oriented basis of $\su (3)$ that is $\langle , \rangle$-orthonormal. Let $v_j^\LL$ denote the extension of $v_j$ as a left-invariant vector field on $M$. Then \eqref{DefinitionAlpha} implies that the vector fields 
\begin{equation}
\label{DefinitionVj}
V_j \; := \;  \alpha (v_j^\LL)  
\end{equation}
form a $g$-orthonormal trivialization of $TM$. Then the lifts of $\nabla$ to the chiral bundles $S^\pm$ exists and are given, in these trivializations, by standard formulas analogous to the Levi-Civita case \cite{Bourguignon, LM}:
\begin{equation}
\nabla^\pm_X  \psi^\pm \; := \; \partial_X \psi^\pm \; + \; \frac{1}{4}\, \sum\nolimits_{j, k}  \, g (\nabla_X V_j, V_k) \; V_j \cdot [ \, V_k \cdot \psi^\pm  \, ] \ .
\end{equation}
Here we identify $S^\pm = M \times \su(3)$, as before, and so chiral spinors are regarded as simple maps $M \rightarrow \su(3)$ that have directional derivatives like $\partial_X \psi^\pm$. Using the explicit model for Clifford multiplication given in \eqref{DeformedCliffordMultiplication}, we can then write that
\begin{align}
\label{SpinorialConnectionsTrivialization}
\nabla^+_X  \psi^+ \; &= \;   \partial_X \psi^+ \; - \; \frac{1}{4}\, \sum\nolimits_{j, k}\, g (\nabla_X V_j, V_k) \; (v_k \ok \psi^+) \ok v_j    \linebr
\nabla^-_X  \psi^- \;  &= \;  \partial_X \psi^- \; - \; \frac{1}{4}\, \sum\nolimits_{j, k} \, g(\nabla_X V_j, V_k) \; v_j \ok ( \psi^- \ok v_k )    \   \nonumber 
\end{align}
as maps $M \rightarrow \su(3)$. For reference, in the trivialization $\{ V_j\}$ of $TM$, a vector field can be written as $Y = \sum_j Y^j V_j$ and its covariant derivative is the usual
\begin{equation}
\label{VectorialConnectionTrivialization}
\nabla_X Y \; =\;  \sum\nolimits_j \, (\partial_X Y^j) \, V_j  \; + \;  \sum\nolimits_{j \neq k} \, g (\nabla_X V_j, V_k) \; Y^j \, V_k \ .
\end{equation}
Since $\nabla$ is compatible with $g$, the products $g (\nabla_X V_j, V_k)$ are anti-symmetric in $j,k$. So one can use \eqref{SpinorialConnectionsTrivialization} and the adjunction relations  \eqref{AdjunctionOkubo} of the Okubo product to check that
\begin{equation}
\label{CompatibilityConnectionInnerProduct1}
\langle  \nabla^\pm_X \psi_1^\pm , \psi_2^\pm \rangle \; + \; \langle  \psi_1^\pm ,  \nabla^\pm_X \psi_2^\pm \rangle \; = \; \partial_X   \langle  \psi_1^\pm ,  \psi_2^\pm \rangle \ .
\end{equation}
This needed to happen, of course, as we know that the lifts to $S^\pm$ of a metric connection on $TM$ are connections compatible with the inner product of spinors.
\begin{proposition}
\label{PropDokCompatibleConnections}
The lifted connections $\nabla$ are compatible with the three instances of the product $\dok$. Besides the usual two identities of chiral spinors,
\begin{equation}
\label{OldCompatibilitiesConnection}
\nabla^\mp_X  (Y \dok \psi^\pm) \; = \;  (\nabla_X  Y) \dok \psi^\pm    \; +  \; Y \dok  (\nabla^\pm_X \psi^\pm)    \ ,
\end{equation}
we have an additional identity of vector fields on $M$, 
\begin{equation}
\label{NewCompatibilityConnection}
\nabla_X  ( \psi^+ \dok \psi^-) \; = \;  (\nabla^+_X  \psi^+) \dok \psi^-    \; +  \; \psi^+ \dok  (\nabla^-_X \psi^-) \ .
\end{equation}
\end{proposition}
\begin{proof}
Relations \eqref{OldCompatibilitiesConnection} are standard. Using them together with the $g$-compatibility of $\nabla$, the compatibilities \eqref{CompatibilityConnectionInnerProduct1}, and the adjunction formulas \eqref{PropertiesDokProduct} repeatedly, we calculate
\begin{align}
 g\big(\nabla_X(\psi^+\dok\psi^-), Y\big) &= \, \partial_X \, g(\psi^+\dok\psi^-,Y) \, - \, g(\psi^+\dok\psi^-,\nabla_XY)    \nonumber  \linebr
 &=\, \partial_X \, \langle Y\dok\psi^+,\psi^-\rangle \, - \, \langle(\nabla_XY)\dok\psi^+,\psi^-\rangle     \nonumber  \linebr
 &=\, \langle \nabla^-_X(Y\dok\psi^+),\psi^-\rangle \,+\, \langle Y\dok\psi^+,\nabla^-_X\psi^-\rangle\, -\, \langle(\nabla_XY)\dok\psi^+,\psi^-\rangle  \nonumber  \linebr
 &=\, \langle Y\dok(\nabla^+_X\psi^+),\psi^-\rangle \,+ \,\langle Y\dok\psi^+,\nabla^-_X\psi^-\rangle \nonumber  \linebr
 &=\, g\big((\nabla^+_X\psi^+)\dok\psi^-,Y\big) \, + \,g\big(\psi^+\dok(\nabla^-_X\psi^-), Y\big)  \ .  \nonumber   \qedhere
 \end{align}  
\end{proof}

\subsection{Defining the three connections on $S^+$}

In this section we will define three natural connections on the chiral spinor bundle $S^+$ related to the original metric connection $\nabla$ on $TM$. The first one is just $\nabla^+$, the usual spinorial lift of $\nabla$ to $S^+$. The second connection is the image of the spinorial lift $\nabla^-$ by the isometry $S^- \rightarrow S^+$ defined in \eqref{DefinitionIsometry2}. The third connection is the image of $\nabla$ itself by the isometry $TM \rightarrow S^+$ defined in \eqref{DefinitionIsometry1}. Thus, we define
\begin{align}
\label{DefinitionThreeConnections}
\nabla^1_X  \psi^+ \; &:= \; \nabla^+_X  \psi^+    \linebr
\nabla^2_X  \psi^+ \; &:= \;  {\mathcal I}^+ \circ ({\mathcal I}^- )^{-1} \circ  \nabla^-_X \circ {\mathcal I}^- \circ ({\mathcal I}^+ )^{-1}\, (\psi^+)   \nonumber \linebr
\nabla^3_X  \psi^+ \; &:= \;  {\mathcal I}^+ \circ  \nabla_X \circ ({\mathcal I}^+ )^{-1}\, (\psi^+)   \nonumber \ .
\end{align}
From \eqref{CompatibilityConnectionInnerProduct1} we know that $\nabla^1 = \nabla^+$ is compatible with the inner-product $\langle , \rangle$ of spinors. Since $\nabla^-$ is also $\langle , \rangle$-compatible and ${\mathcal I}^+ \circ ({\mathcal I}^- )^{-1}$ is an isometry between $S^-$ and $S^+$, the connection $\nabla^2$ will also be $\langle , \rangle$-compatible in $S^+$. Similarly, because $\nabla$ is $g$-compatible on $TM$ and ${\mathcal I}^+$ is an isometry between $(TM, g)$ and $(S^+, \langle , \rangle)$, the connection $\nabla^3$ will be $\langle , \rangle$-compatible in $S^+$ as well. So for $r=1, 2, 3$ we have
\begin{equation}
\label{CompatibilityConnectionInnerProduct3}
\langle  \nabla^r_X \psi_1^+ , \psi_2^+ \rangle \; + \; \langle  \psi_1^+ ,  \nabla^r_X \psi_2^+ \rangle \; = \; \partial_X   \langle  \psi_1^+ ,  \psi_2^+ \rangle \ .
\end{equation}
Explicitly, one can calculate, for example:
\begin{align}
\partial_X  \langle  \psi^+ , \psi^+  \rangle  \; &= \;  \partial_X \,  \big\langle \,  {\mathcal I}^+ \circ ({\mathcal I}^+ )^{-1} \psi^+ \, , \,   {\mathcal I}^+ \circ ({\mathcal I}^+ )^{-1} \psi^+  \, \big\rangle  \linebr
&= \;  \partial_X \,  g\big( \, ({\mathcal I}^+ )^{-1} \psi^+ \, , \,  ({\mathcal I}^+ )^{-1} \psi^+   \, \big)  \nonumber \linebr
&= \;   g\big( \nabla_X \circ ({\mathcal I}^+ )^{-1} \psi^+ \, , \,  ({\mathcal I}^+ )^{-1} \psi^+   \big)  \; + \:  g\big(  ({\mathcal I}^+ )^{-1} \psi^+ \, , \, \nabla_X \circ   ({\mathcal I}^+ )^{-1} \psi^+   \big)  \nonumber  \linebr
&= \;   \big\langle {\mathcal I}^+ \circ \nabla_X \circ ({\mathcal I}^+ )^{-1} \psi^+ \, , \,  \psi^+  \big\rangle  \; + \:  \big\langle \psi^+ \, , \, {\mathcal I}^+ \circ \nabla_X \circ   ({\mathcal I}^+ )^{-1} \psi^+  \big\rangle  \nonumber  \linebr
&= \; \langle \nabla^3_X  \psi^+ , \psi^+  \rangle  +   \langle \psi^+ ,  \nabla^3_X  \psi^+ \rangle   \nonumber \ ,
\end{align}
where we have made repeated use of \eqref{Isometry1}.

Now we give a more direct expressions for the three connections $\nabla^r$, as written in a trivialization of $S^+ = M \times \su(3)$. 
\begin{proposition} 
\label{PropThreeConnections}
Regarding chiral spinors as maps $M \rightarrow \su(3)$, the three connections on $S^+$ defined above can be written as
\begin{align}
\label{ThreeConnections}
\nabla^1_X  \psi^+ \; &= \; \partial_X \psi^+ \; - \; \frac{1}{4}\, \sum\nolimits_{j, k} \, g (\nabla_X V_j, V_k) \; (v_k \ok \psi^+) \ok v_j      \linebr
\nabla^2_X  \psi^+ \; &= \; \partial_X \psi^+ \; - \; \frac{1}{4}\, \sum\nolimits_{j, k} \, g (\nabla_X V_j, V_k) \; v_j \ok ( \psi^+ \ok v_k )   \nonumber   \linebr
\nabla^3_X  \psi^+    \; &= \; \partial_X \psi^+   \; + \;  \sum\nolimits_{j, k} \, g (\nabla_X V_j, V_k) \; \langle \psi^+  , v_j \rangle \, v_k \   \nonumber
\end{align}
in any given $\langle , \rangle$-orthonormal basis $\{ v_j\}$ of $\su(3)$, where the $V_j = \alpha (v_j^\LL)$ are $g$-orthonormal.
\end{proposition}
\begin{proof}
Since $\nabla^1 = \nabla^+$, the first identity in \eqref{ThreeConnections} follows directly from \eqref{SpinorialConnectionsTrivialization}. For the second identity, we only need to note that the isometry ${\mathcal I}^- \circ ({\mathcal I}^+ )^{-1}$ between $S^+$ and $S^-$, as defined in \eqref{DefinitionIsometry2}, actually acts trivially on $M \times \su(3)$. So the explicit formula for $\nabla^2_X$ coincides with the formula for $\nabla^-_X$ given in \eqref{SpinorialConnectionsTrivialization}.  For the case of $\nabla^3$, first note that the isometry ${\mathcal I}^+$ defined in \eqref{DefinitionIsometry1} transforms vectors in $\su(3)$ as $v \mapsto \alpha^{-1} (v)$. Thus, regarding the spinor $\psi^+$ as a map $M \rightarrow \su(3)$, the covariant derivative $\nabla^3_X = {\mathcal I}^+ \circ  \nabla_X \circ ({\mathcal I}^+ )^{-1}$ has the same formula as the covariant derivative $\nabla_X^\alpha  = \alpha^{-1} \circ \nabla_X \circ \alpha$ acting on such maps. So now we can apply expression \eqref{VectorialConnectionTrivialization} for $\nabla$. After observing that
\[
[\alpha (Y)]^j \; = \; g( \alpha(Y), V_j) \; = \; \langle Y , v_j^\LL\rangle \  , 
\]
expression \eqref{VectorialConnectionTrivialization} and definition \eqref{DefinitionVj} imply that
\begin{align}
\nabla^\alpha_X  Y \; &= \;  \alpha^{-1} \big\{ \, \nabla_X [\,\alpha( Y)\, ] \, \big\}   \linebr
&= \;  \alpha^{-1} \big\{ \, \sum\nolimits_j \big[  \partial_X (\alpha\, Y)^j  \big]\, V_j \, + \, \sum\nolimits_{j, k} \,  g (\nabla_X V_j, V_k) \,  (\alpha \, Y)^j  \, V_k   \big\}    \nonumber   \linebr
&= \;   \sum\nolimits_j \big[  \partial_X \langle Y , v_j^\LL \rangle  \big] \, v_j^\LL \, + \, \sum\nolimits_{j, k} \,  g (\nabla_X V_j, V_k) \,   \langle Y , v_j^\LL \rangle  \, v_k^\LL   \nonumber   \ .
\end{align}
But the inner product $\langle Y , v_j^\LL \rangle$ with the bi-invariant metric on $TM$ is just $ \langle Y , v_j \rangle$ when we regard $Y$ as a map $M \rightarrow \su(3)$. Moreover, the left-invariant vector field $v_j^\LL$ translates as the map on $M$ with constant value $v_j \in \su(3)$. So the formula for $\nabla^\alpha_X$ regarded as an operator on the space of maps $Y: M \rightarrow \su(3)$ is just 
\begin{align}
\label{ExplicitNablaAlpha}
\nabla^\alpha_X  Y \; &= \;  \sum\nolimits_j \big[  \partial_X \langle Y , v_j \rangle  \big] \, v_j \, + \, \sum\nolimits_{j, k} \,  g (\nabla_X V_j, V_k) \,   \langle Y , v_j \rangle  \, v_k   \linebr
&= \;  \partial_X  Y \; + \; \sum\nolimits_{j, k} \,  g (\nabla_X V_j, V_k) \,   \langle Y , v_j \rangle  \, v_k   \nonumber   \ .
\end{align}
Using this general expression for the particular map $\psi^+: M \rightarrow \su(3)$, we obtain the formula for $\nabla^3$ given in the proposition. 
\end{proof}

\subsection{Relations among the connections}

The three connections given in formulas \eqref{ThreeConnections} can be written more succintly.
Given a $\langle , \rangle$-orthogonal endomorphism $A \in \mathfrak{so}( \su(3))$, define linear maps $\rho^r_A: \su(3) \rightarrow \su(3)$ by
\begin{align}
\rho^1_A \, (\psi)  \: &= \:  - \frac{1}{4} \,  \sum\nolimits_j  \, [ ( A\, v_j) \ok \psi ] \ok v_j    \label{rho1} \linebr
\rho^2_A \, (\psi ) \: &= \:  - \frac{1}{4} \, \sum\nolimits_j  \, v_j \ok [\psi \ok ( A\, v_j) ]      \label{rho2} \linebr
\rho^3_A \, (\psi)  \: &= \:   \sum\nolimits_j  \, \langle v_j, \psi \rangle \,  A\, v_j  \: = \: A \, \psi \ , \label{rho3}
\end{align}
where $\{v_j\}$ is a  $\langle , \rangle$-orthonormal basis of $\su(3)$. Then formulas \eqref{ThreeConnections} for the connections $\nabla^r$ on $S^+$ become just
\begin{equation}
\label{CondensedConnections}
\nabla^r_X \, \psi  \: = \:  \partial_X \psi \; +\;  \rho^r_{A_X} (\psi)    \ ,
\end{equation}
where $A_X : M \rightarrow  \mathfrak{so}( \su(3))$ is the map defined by
\begin{equation}
\label{EndomorphismConnection}
 (A_X)_x (v_j)  \: =\:  \sum\nolimits_k \, g_x(\nabla_{X} V_j , V_k) \, v_k \ 
\end{equation}
 for all $x \in M$. Since the connection $\nabla^1$ was defined as the usual spinorial lift of $\nabla$ to $S^+$, the map $A \mapsto \rho^1_A $ is one of the spin representations of $\mathfrak{so}( \su(3)) \simeq \mathfrak{so}(8)$. This can also be checked directly using relation \eqref{TraditionalCliffordMultiplication} between the deformed Okubo product and the action of gamma matrices. The connection $\nabla^2$ was defined as the image of the spinorial lift of $\nabla$ to $S^-$ by the isometry $S^- \rightarrow S^+$ in \eqref{DefinitionIsometry2}. But this isometry acts trivially on $M \times \su(3)$, as the underlying bundle of $S^\pm$. So the map $A \mapsto \rho^2_A $ is the second spin representations of $\mathfrak{so}(8)$. The map $A \mapsto \rho^3_A $ is just the defining vectorial representation of $\mathfrak{so}(8)$. Thus, the three connections $\nabla^r$ are the images by the three inequivalent representations of $\mathfrak{so}(8)$ of the same endomorphism $A_X \in \mathfrak{so}( \su(3))$ determined by the metric connection $\nabla$.
 
The following algebraic identity says that the three representations of $\mathfrak{so}(8)$ have the same restriction to the subalgebra $\su(3) \subset \mathfrak{so}(8)$ determined by $\mathrm{ad}$-transformations. It is a well-known result proved by Hitchin in \cite{Hitchin} through abstract highest weight calculations. Here we can check it directly using the properties of the Okubo product.
\begin{lemma}[\!\!\cite{Hitchin}]
\label{Restriction3Reps}
Given a vector $u \in \su(3)$, the endomorphism $\mathrm{ad}_u$ satisfies
\[
\rho^1_{\mathrm{ad}_u} \: =\: \rho^2_{\mathrm{ad}_u}  \: =\: \rho^3_{\mathrm{ad}_u}  \: =\: \mathrm{ad}_u
\]
 as linear maps $\su(3) \rightarrow \su(3)$.
\end{lemma}
For example, using one of the identities \eqref{ImportantIdentityOkubo1} of the $\ok$-product, we get
\begin{equation}
4 \, \rho^1_{\mathrm{ad}_u} (\psi^+) \: = \: -  \sum\nolimits_{j}  \, ([u, v_j] \ok \psi^+) \ok v_j  \: = \:   \, 4\,   [u, \psi^+]    \nonumber  \ .
\end{equation}

When $g$ is a bi-invariant metric on $M$, there is a famous family of $g$-compatible connections on $TM$ introduced by Cartan and Schouten. They depend on a real parameter $s$ and include the Levi-Civita connection. Denoting $u^\LL$ and $v^\LL$ the left-invariant extensions of $u, v \in \su(3)$, these connections are determined by $\nabla_{u^\LL} v^\LL = s \, [u^\LL,  v^\LL ] =  s \, [u,  v]^\LL$. The previous lemma then has the following consequence:
\begin{corollary}
\label{CoincidenceConnectionsBiinvariant}
When $g$ is bi-invariant and the connection $\nabla$ on $TM$ belongs to the Cartan-Schouten family, the three connections $\nabla^r$ induced on $S^+$ coincide.
\end{corollary}
\begin{proof}
When $g$ is bi-invariant, $g(X,Y) = \lambda^2 \langle X , Y \rangle$ for some fixed $\lambda \in \mathbb{R}^+$. Then the automorphism $\alpha$ in \eqref{DefinitionAlpha} is simply $\alpha (v) = \lambda^{-1} v$. The trivialization vector fields in \eqref{DefinitionVj} are $V_j = \lambda^{-1} v_j^\LL$. So for $\nabla$ in the Cartan-Schouten line we have 
\[
g (\nabla_{u^\LL} V_j, V_k) \:= \:  \langle \nabla_{u^\LL} v_j^\LL , v_k^\LL \rangle \: = \:  s\, \langle [u, v_j] , v_k \rangle  \ .
\] 
Thus, \eqref{EndomorphismConnection} is the constant endomorphism $A_{u^\LL} = s\, \mathrm{ad}_u$. Lemma \ref{Restriction3Reps} and formulas \eqref{CondensedConnections} for the connections then guarantee that $\nabla^1_{u^\LL} = \nabla^2_{u^\LL} = \nabla^3_{u^\LL}$. Since this identity holds for any left-invariant $u^\LL$, it will hold for a general vector field.
\end{proof}
\begin{remark}
Recall that the torsion of a metric connection splits into a vectorial part, a totally skew-symmetric part, and component of mixed symmetry. Only the first two are seen by the associated classical Dirac operator \cite{Agricola}. The Cartan-Schouten connections are the simplest among the much studied connections with purely skew torsion \cite{FI, Agricola2, AgricolaF, Ferreira}. 
\end{remark}

The three spin representations of $\mathfrak{so}(8)$ satisfy an additional algebraic identities:
\begin{lemma}
\label{Sum3Reps}
Given an endomorphism $A \in \mathfrak{so}( \su(3))$, the representations $\rho_A^r$ satisfy
\[
\rho^1_{A}(\psi) \: +\: \rho^2_{A}(\psi) \: +\: \rho^3_{A}(\psi) \: = \:  \mathrm{ad}_{u_A} (\psi)
\]
for all $\psi$, where $u_A \, :=\, a^2 \sum\nolimits_j [v_j, A v_j ]$ is a fixed vector in $\su(3)$.
\end{lemma}
\begin{proof}
Using that $A$ is antisymmetric with respect to $\langle , \rangle$, rewrite \eqref{rho2} as
\[
\rho^2_A \, (\psi ) \: = \:  - \frac{1}{4} \, \sum\nolimits_{j,k}  \,  \langle v_k,  A\, v_j \rangle \,  v_j \ok [\psi \ok v_k ]   \: = \:   \frac{1}{4} \, \sum\nolimits_j  \,   (A\, v_j) \ok [\psi \ok v_j ]  \ .
\]
Then the identity in the lemma follows directly from formula \eqref{AssociatorOkubo} for the associator of the $\ok$-product after substituting $u= A v_j$, $v= \psi$ and $w= v_j$.
\end{proof}
\begin{lemma}
\label{Compatibility3Reps}
Given an endomorphism $A \in \mathfrak{so}( \su(3))$, the maps $\rho_A^r: \su(3) \rightarrow \su(3)$ satisfy
\begin{equation}
\label{CyclicIdentities}
\rho^{\sigma(1)}_{A}(u \ok v)  \: = \: (\rho^{\sigma(2)}_{A} \, u) \ok v \: +\:  u\ok (\rho^{\sigma(3)}_{A}\, v)  
\end{equation}
for all $u, v \in \su(3)$, where $\sigma$ denotes any cyclic permutation of $(1, 2, 3)$.
\end{lemma}
\begin{proof}
Instead of an algebraic proof, we derive these identities from proposition \ref{PropDokCompatibleConnections}. Consider the case where $g$ is bi-invariant, $\alpha$ is the identity, and regard all vectors fields and chiral spinors as simple maps $M \rightarrow \su(3)$. Below \eqref{EndomorphismConnection} we have identified $\nabla^1 = \nabla^+$ as the spinorial lift of $\nabla$ to $S^+$, identified $\nabla^2 = \nabla^-$ as the spinorial lift of $\nabla$ to $S^-$, and identified $\nabla^3 = \nabla$. Thus, applying formula \eqref{CondensedConnections} for the connections $\nabla^r$ in the case of constant maps $M \rightarrow \su(3)$, we get that $\nabla^+_X \psi^+ = \rho^1_{A_X} \psi^+$, that $\nabla^-_X \psi^- = \rho^2_{A_X} \psi^-$, and that $\nabla_X Y = \rho^3_{A_X} Y$. 
The desired identities \eqref{CyclicIdentities} are then a direct consequence of the three compatibility relations in proposition \ref{PropDokCompatibleConnections}, after using definition \eqref{DeformedCliffordMultiplication} of $\dok$ in the case $\alpha =1$, and after using that the range of the transformations $A_X$ in \eqref{EndomorphismConnection} is the whole $\mathfrak{so}( \su(3))$, as $X$ and the metric connection $\nabla$ vary.
\end{proof}
Finally, note that formula \eqref{CondensedConnections} can be extended to construct new connections on $S^+$ through linear combinations with real coefficients,
\begin{equation}
\label{LinearCombinationsConnections}
\nabla^{t_1, t_2, t_3}_X  \, \psi  \: = \:  \partial_X \psi \; +\;  \sum\nolimits_{r=1}^3 \,  t_r\: \rho^r_{A_X} (\psi) \ .
\end{equation}
All these connections are compatible with the spinor inner product.

\section{Rotated connections on $TM$ and their spinorial lifts}
\label{Section3ConnectionsSpinorialLifts}

\subsection{New connections on $TM$ and $S^-$}

As was done in section \ref{ThreeNaturalConnections} in the case of $S^+$, we can use the triality isometries ${\mathcal I}^\pm : TM \rightarrow S^\pm$ of \eqref{DefinitionIsometry1} to transport the connections $\nabla$ on $TM$ and $\nabla^\pm$ on $S^\pm$, and obtain three distinct connections on $S^-$ and $TM$. For any chiral spinor $\psi^-$ and any vector fields $X, Y$, the connections $\nabla^r$ are defined on the respective vector bundles by
\begin{align}
\label{DefinitionThreeConnections2}
\nabla^1_X  \psi^- \ &:= \   \nabla^-_X  \psi^-   \linebr
\nabla^2_X  \psi^- \ &:= \   {\mathcal I}^- \circ  \nabla_X \circ ({\mathcal I}^- )^{-1}\, (\psi^-)  \nonumber \linebr
\nabla^3_X  \psi^- \ &:= \  {\mathcal I}^- \circ ({\mathcal I}^+ )^{-1} \circ  \nabla^+_X \circ {\mathcal I}^+ \circ ({\mathcal I}^- )^{-1}\, (\psi^-)  \nonumber  \displaybreak[0] \\[1.2em]
\label{DefinitionThreeConnections3}
\nabla^1_X  Y \ &:= \  \nabla_X Y    \linebr
\nabla^2_X  Y \ &:= \  ({\mathcal I}^+ )^{-1}  \circ  \nabla^+_X \circ {\mathcal I}^+ \, (Y)  \nonumber \linebr
\nabla^3_X  Y \ &:= \  ({\mathcal I}^- )^{-1}  \circ  \nabla^-_X \circ {\mathcal I}^- \, (Y)   \nonumber \ .
\end{align}
When we regard chiral spinors as maps $M \rightarrow \su(3)$ through the trivialization of $S^-$ induced by the trivialization $\{V_j \}$ of $TM$, the analogs of expressions \eqref{CondensedConnections} are
\begin{align}
\label{CondensedConnections2}
\nabla^1_X  \psi^- \ &= \  \partial_X \psi^- \, + \: \rho^2_{A_X} (\psi^-)     \linebr
\nabla^2_X  \psi^- \ &= \  \partial_X \psi^- \, + \: \rho^3_{A_X} (\psi^-)    \nonumber \linebr
\nabla^3_X  \psi^- \ &= \  \partial_X \psi^- \, + \: \rho^1_{A_X} (\psi^-)    \nonumber \ ,
\end{align}
where the linear maps $\rho^r_A: \su(3) \rightarrow \su(3)$ are still given by \eqref{rho1} and the maps $A_X: M \rightarrow \mathrm{End} (\su(3))$ are still given by \eqref{EndomorphismConnection}. The indices of $\nabla^r$ and $\rho^r$ do not match because we chose to denote by $\nabla^r$ the spinorial lift to $S^-$ of the connection $\nabla^r$ on $TM$ (see theorem \ref{SpinorialLiftsVSTrialityTransports}), not the image of that connection under the triality isometries. For example, $A \mapsto \rho^2_A$ is the usual spin representation on $\Delta^-$, as discussed in the paragraph after \eqref{EndomorphismConnection},  so our connection $\nabla^1$ on $S^-$ is the spinorial lift of $\nabla = \nabla^1$ on $TM$.

In the same way, if we use the correspondence $Y \mapsto \sum_j g(Y, V_j)\, v_j$ to regard vector fields as maps $M \rightarrow \su(3)$, the analogs of \eqref{CondensedConnections} for the connections $\nabla^r$ on $TM$ are
\begin{align}
\label{CondensedConnections3}
\nabla^1_X  Y \ &= \  \partial_X Y \, + \: \rho^3_{A_X} (Y)      \linebr
\nabla^2_X  Y \ &= \  \partial_X Y \, + \: \rho^1_{A_X}  (Y)      \nonumber \linebr
\nabla^3_X  Y \ &= \  \partial_X Y \, + \: \rho^2_{A_X}  (Y)      \nonumber \ .
\end{align}
Again, here we named $\nabla^r$ with an index that does not match that of the corresponding $\rho^r$, but will, instead, allows us to follow the traditional convention of using the same symbol for a connection on $TM$ and its spinorial lifts to $S^\pm$. 

The calculations in $S^-$ and $TM$ that take us from definitions \eqref{DefinitionThreeConnections2} and \eqref{DefinitionThreeConnections3} of $\nabla^r$ to the explicit expressions \eqref{CondensedConnections2} and \eqref{CondensedConnections3} are completely analog to the calculations that, in the case of $S^+$, derived \eqref{ThreeConnections} and \eqref{CondensedConnections} from definitions \eqref{DefinitionThreeConnections}.

\subsection{Triality transports versus spinorial lifts}

At this point we have defined three connections $\nabla^r$ on each of the bundles $TM$ and $S^\pm$. A natural question is what is their precise relation through spinorial lifts of connections and through the triality isometries ${\mathcal{I}^\pm}$. This is answered by the next result.

\begin{theorem}
\label{SpinorialLiftsVSTrialityTransports}
The spinorial lifts of the $g$-compatible connection $\nabla^r$ on $TM$ are the $\langle , \rangle$-compatible connections $\nabla^r$ on $S^\pm$, with the definitions \eqref{DefinitionThreeConnections}, \eqref{DefinitionThreeConnections2}, \eqref{DefinitionThreeConnections3}, for $r=1, 2, 3$. Under the triality isometries, the connection $\nabla^r$ on $TM$ is transported instead to $\nabla^{r \mp 1}$ on $S^\pm$, with indices taken modulo 3.
\end{theorem}
\begin{proof}
The assertion about the transport under the triality isometries is easily proved and almost tautological, due to the definitions of the three connections. For instance, using definition \eqref{DefinitionThreeConnections3}, the image of $\nabla^2$ on $TM$ through the isometry ${\mathcal{I}^-}:TM \rightarrow S^-$ is 
\[ 
{\mathcal I}^- \circ  \nabla^2_X \circ ({\mathcal I}^- )^{-1}\: = \: {\mathcal I}^- \circ  ({\mathcal I}^+ )^{-1}  \circ  \nabla^+_X \circ {\mathcal I}^+  \circ ({\mathcal I}^- )^{-1}  \: = \: \nabla^3_X \ ,
\]
as a connection on $S^-$. The proofs for the other triality transports are similar.

Thus, the main content of the theorem is the assertion about the spinorial lifts. To address it, start by observing that, for every $A \in \mathfrak{so}(\su(3)) \simeq \mathfrak{so}(8)$, the correspondences $A \mapsto \rho^1_A$ and $A \mapsto \rho^2_A$ are, respectively, the spin representations of $\mathfrak{so}(8)$ on the spin spaces $\Delta^+$ and $\Delta^-$. This was justified in the paragraph after \eqref{EndomorphismConnection}. In particular, the spinorial lifts to $S^+$ of the three connections on $TM$, as expressed in \eqref{CondensedConnections3} through the $\rho^r_{A_X}$, are determined by the composed endomorphisms $\rho^1_{\rho^r_{A_X}}$. Similarly, the spinorial lifts to $S^-$ of the three connections on $TM$ are determined by endomorphisms $\rho^2_{\rho^r_{A_X}}$. The assertion of the theorem then follows directly from the algebraic content of proposition \ref{PropositionTrialityCompositions}. 
\end{proof}
\begin{proposition}
\label{PropositionTrialityCompositions}
For every $A \in \mathfrak{so}(\su(3))$, the transformations $\rho^r_A$ are also in $\mathfrak{so}(\su(3))$ and satisfy the composition rules
\begin{align}
\rho^1_{\rho^1_A} \: &= \: \rho^2_A    &   \rho^1_{\rho^2_A} \: &= \: \rho^3_A  &   \rho^1_{\rho^3_A} \: &= \: \rho^1_A     \linebr
\rho^2_{\rho^1_A} \: &= \: \rho^3_A    &   \rho^2_{\rho^2_A} \: &= \: \rho^1_A  &   \rho^2_{\rho^3_A} \: &= \: \rho^2_A   \ .
\end{align}
\end{proposition}
Here, the $\mathfrak{so}(\su(3))$ is defined with respect to the inner product $\langle, \rangle$ on $\su(3)$.  The proof of this proposition is given in appendix \ref{Appendix3ConnectionsSpinorialLifts} after a sequence of auxiliary lemmas.

\subsection{Curvature and torsion of the rotated connections}
\label{SectionCurvatureTorsion}

The curvature of a connection is the tensor $R^\nabla (X, Y) Z =  \nabla_X \nabla_Y Z -  \nabla_Y \nabla_X  Z - \nabla_{[X, Y]}Z$. The curvatures of the triality rotated metric connections $\nabla^2$ and $\nabla^3$ on $TM$ are, in general, different from the curvature of the original $\nabla = \nabla^1$. The purpose of this section is to relate them.
As before, we denote by $\{ v_j \}$ a $\langle , \rangle$-orthonormal basis of $\su(3)$ and by $V_j = \alpha(v_j^\LL)$ the associated $g$-orthonormal frame of $TM$, as in \eqref{DefinitionVj}. 

\begin{proposition} 
\label{CurvatureRotatedConnections}
The curvature tensors of the connections $\nabla^2$ and $\nabla^3$ on $TM$ are related to that of $\nabla = \nabla^1$ by the formulas
\begin{align}
g \big( R^{\nabla^1}(X, Y) V_i \, , V_j \big) \, &= \: \langle \, F_{X,Y} (v_i) , \,  v_j \, \rangle  \: = \: \langle \, \rho^3_{F_{X,Y}} v_i , \,  v_j \, \rangle     \label{ExplicitCurvature}  \linebr
g \big( R^{\nabla^2}(X, Y) V_i \, , V_j \big) \, &= \, - \tfrac{1}{4} \sum\nolimits_{k,l}\ \mathrm{R}^{\nabla}(X, Y, V_k, V_l)\:  \langle \, ( v_l \ok v_i) \ok v_k  ,\,  v_j \, \rangle  \, = \,     \langle \, \rho^1_{F_{X,Y}} v_i , \,  v_j \, \rangle  \nonumber  \linebr
g \big( R^{\nabla^3}(X, Y) V_i  ,\, V_j \big) \, &= \,  - \tfrac{1}{4}  \sum\nolimits_{k,l}  \mathrm{R}^{\nabla}(X, Y, V_k, V_l)\:  \langle \,  v_k \ok (v_i \ok v_l)  ,\,  v_j \, \rangle\, = \,     \langle \, \rho^2_{F_{X,Y}} v_i , \,  v_j \, \rangle  \, , \nonumber
\end{align}
where the representations $\rho^r$ are as in \eqref{rho1}, and  $F_{X, Y}$ is the endomorphism in  $\mathfrak{so} (\su(3))$ defined by $F_{X, Y} (v_k) = \sum_l \mathrm{R}^\nabla(X, Y, V_k, V_l)\, v_l$.
\end{proposition}
\begin{proof}
The first identity is trivial, resulting from the fact that $\rho^3_A = A$ and that $\{ v_j\}$ is a $\langle , \rangle$-orthonormal basis of $\su(3)$.
For the other two identities, observe that definition \eqref{DefinitionThreeConnections3} says that $\nabla^2_X  = (\mathcal{I}^+)^{-1} \circ \nabla_X^+ \circ \mathcal{I}^+$, so the curvatures will also satisfy
\[
R^{\nabla^2}(X, Y)\, Z  \: = \: (\mathcal{I}^+)^{-1} \circ R^{\nabla^+}(X, Y) \circ \mathcal{I}^+ (Z) .
\]
 But the curvatures of the spinorial lifts of $\nabla$ have the standard expression \cite[sec.~II.4]{LM}
\[
R^{\nabla^\pm}(X, Y)\, \psi^\pm \: = \: \tfrac{1}{4} \sum\nolimits_{k, l} \, \mathrm{R}^{\nabla}(X, Y, V_k, V_l)\; V_k \cdot V_l \cdot \psi^\pm \ .
\]
Moreover, regarding chiral spinors as maps $M \rightarrow \su(3)$, the model \eqref{TraditionalCliffordMultiplication} of Clifford multiplication says that $V_k \cdot V_l \cdot \psi^+$ is identified with the map $- ( v_l \ok \psi^+) \ok v_k$. Thus, using definition \eqref{DefinitionIsometry1} of $\mathcal{I}^+$ and the relation \eqref{DefinitionAlpha} between $g$ and $\langle , \rangle$, we obtain that
\begin{align}
g\big( R^{\nabla^2}(X, Y) V_i,\, V_j \big) \: &= \:  \tfrac{1}{4} \sum\nolimits_{k,l} \, \mathrm{R}^{\nabla}(X, Y, V_k, V_l)\;  g\big( \alpha[ V_k \cdot V_l \cdot (\alpha^{-1} V_i) ] ,\,  V_j \big)   \nonumber \linebr
&= \:  - \tfrac{1}{4} \sum\nolimits_{k,l} \, \mathrm{R}^{\nabla}(X, Y, V_k, V_l)\;  \langle \, ( v_l \ok v_i) \ok v_k  ,\,  v_j \, \rangle  \ .
\end{align}
The proof of the formula for $R^{\nabla^3}$ is similar. The last equalities are a direct application of the definitions \eqref{rho1} of the $\rho^r$.
\end{proof}

The formulas of proposition \ref{CurvatureRotatedConnections} relating the curvature tensors $R^{\nabla^r}$ and $R^{\nabla}$ also determine, by straightforward contractions, similar formulas relating the Ricci tensors and Ricci scalars of $\nabla^r$ and $\nabla$. So we get, for example:
\begin{corollary}
The Ricci scalars of the connections $\nabla^r$ satisfy
\begin{align}
\label{ExplicitRicciScalars}
&\mathrm{Scal}^2 \: = \: \tfrac{1}{4} \sum_{i,j,k,l}  \mathrm{R}^\nabla_{ijkl} \, \langle v_k \ok v_j , \, v_l \ok v_i \rangle  \qquad
\mathrm{Scal}^3 \: = \: \tfrac{1}{4} \sum_{i,j,k,l}  \mathrm{R}^\nabla_{klij} \, \langle v_k \ok v_j , \, v_l \ok v_i \rangle \ \nonumber \linebr
&\mathrm{Scal}^1 \, +\, \mathrm{Scal}^2 \, +\, \mathrm{Scal}^3 \:= \: a^2 \,\sum\nolimits_{i,j,k,l} \, \mathrm{R}^\nabla_{ijkl} \, \big\langle [v_i, v_j] , [v_l, v_k] \big\rangle  \ .
\end{align}
Thus $\mathrm{Scal}^2  = \mathrm{Scal}^3$ when $\mathrm{R}^\nabla_{ijkl} = \mathrm{R}^\nabla_{klij}$, which is true for the Levi-Civita connection. 
\end{corollary}
\begin{proof}
The expressions form $\mathrm{Scal}^2$ and $\mathrm{Scal}^3$ follow directly from the standard contractions of the curvatures in \eqref{ExplicitCurvature}, after using the adjunction relations  \eqref{AdjunctionOkubo} and renaming the indices in the case of $\mathrm{Scal}^3$. The identity for the sum of the Ricci scalars is an application of proposition \ref{CurvatureRotatedConnections} and lemma \ref{Sum3Reps}, together with the Ad-invariance of $\langle , \rangle$ :
\begin{align}
\sum_{r=1}^3 \, \mathrm{Scal}^r  \, &=  \,  \sum_{r} \sum_{i,j} \, \langle \, \rho^r_{F_{V_j,V_i}} ( v_i) , \,  v_j \, \rangle \,  =  \,  \sum_{i,j} \,  \langle \, [ u_{F_{V_j,V_i}} ,  v_i ] , \,  v_j \, \rangle \nonumber \linebr
&= \, \sum_{i, j, k} \,  \langle {u_{F_{V_j,V_i}}} , \,  v_k  \rangle  \:  \langle \, [v_k, v_i] ,   v_j \, \rangle \, =\,  a^2 \sum_{i,j ,k, l}  \, \langle  [v_l, F_{V_j,V_i} (v_l) ] , \,  v_k  \rangle  \:  \langle \, [v_k, v_i] , v_j \, \rangle  \nonumber \linebr
&= \, a^2 \sum_{i,j ,k, l, s}  \, \mathrm{R}^\nabla_{j, i, l, s}\,  \langle  [v_l , v_s ] ,  v_k \rangle  \:  \langle \, [v_k, v_i] ,   v_j \, \rangle \, =\,    a^2 \sum_{i,j, l, s}   \,\mathrm{R}^\nabla_{j, i, l, s}\,  \langle  [v_l , v_s ] ,  [ v_i ,  v_j ]\, \rangle  \, .   \nonumber   \qedhere
\end{align}
\end{proof}

The triality rotated connections $\nabla^r$ on $TM$ can have torsion, even when $\nabla = \nabla^1$ does not. Now we offer an expression for the torsion traces, $\tau^{\nabla^r}(X) := \sum_j g(T^{\nabla^r}(X, V_j) , V_j)$, written in terms of the original connection $\nabla$ and Okubo products. The torsion trace is directly related to the axial couplings (degree 7 terms) of the Dirac operators of section \ref{ThreeDiracOperators2}, when expanded as products of gamma matrices.
\begin{proposition}
\label{PropTorsionTrace}
Denote by $\tau_r$ the vector field $g$-dual to the trace torsion 1-form $\tau^{\nabla^r}$. Regarding $\tau_r$ as a map $M\rightarrow \su(3)$, as usual, it is given by 
\begin{align}
\label{TorsionTraceFormulas}
\tau_1 \: &= \:   \sum\nolimits_{j, k}  \Gamma_{jjk}\; v_k  \; + \; \sum\nolimits_j \, (\divergence_g V_j)\, v_j    \linebr
\tau_2 \: &= \: - \tfrac{1}{4}\, \sum\nolimits_{j, k,l} \, \Gamma_{jkl} \; (v_l \ok v_j) \ok v_k  \; + \; \sum\nolimits_j \, (\divergence_g V_j)\, v_j \nonumber \linebr
\tau_3 \: &= \: - \tfrac{1}{4} \, \sum\nolimits_{j, k,l} \,\Gamma_{jkl} \; v_k \ok (v_j \ok v_l)   \; + \; \sum\nolimits_j \, (\divergence_g V_j)\, v_j \ ,  \nonumber
\end{align}
where we have abreviated $\Gamma_{jkl} = g(\nabla_{V_j} V_k , V_l)$.
\end{proposition}
\begin{proof}
In our usual trivialization $\{ V_j\}$ of $TM$, write $T^{\nabla^r}(V_i, V_j) = \nabla^r_{V_i} V_j - \nabla^r_{V_j} V_i - [V_i, V_j]$. Then the trace $\tau^{\nabla^r} (V_i)$ has three terms. The first is $g(\nabla^r_{V_i} V_j , V_j) $, which vanishes because $\nabla^r$ is metric and $V_j$ has unit length. Using the torsionless Levi-Civita connection $\nabla^g$ and a standard formula for the divergence, the third term is 
\[
- \sum\nolimits_j \, g([V_i, V_j], V_j)\, =\, \sum\nolimits_j \, g(\nabla^g_{V_j} V_i -\nabla^g_{V_i} V_j  \, , \,V_j) \, =\, \sum\nolimits_j \, g(\nabla^g_{V_j} V_i , V_j) \,  = \,\divergence_g( V_i ) \; .
\]
For the second term of $\tau^{\nabla^r} (V_i)$, we can use that $\nabla^r$ is metric, use expressions \eqref{CondensedConnections3} for $\nabla^r$ on $TM$, relation \eqref{DefinitionAlpha} between $g$ and $\langle, \rangle$, and the fact that $V_j = \alpha (v_j^\LL)$, to obtain
\begin{equation}
\label{PropTorsionTraceAux1}
- \sum\nolimits_j \, g(\nabla^r_{V_j} V_i, V_j) \, =\, \sum\nolimits_j \, g( V_i, \nabla^r_{V_j} V_j) \, =\,  \sum\nolimits_j \, \langle v_i \, , \,\rho^{r-1}_{A_{V_j}} v_j \rangle    \: ,
\end{equation}
with indices $r$ taken modulo 3. Finally, applying expressions \eqref{rho1} for the $\rho^{r-1}$, formula \eqref{EndomorphismConnection} for $A_{V_j}$, and summing the previously obtained divergence term, the results for the vector fields $\tau_r = \sum_j \tau^{\nabla^r} (V_j)\,  v_j$ are the expressions \eqref{TorsionTraceFormulas}.
\end{proof}
\begin{remark}
\label{RemarkTorsionLeftInvariantMetrics}
When $g$ is a left-invariant metric on the unimodular group $M = \SU(3)$, the Riemannian volume form $\vol_g$ coincides with the fully bi-invariant Haar volume form, up to normalization. Moreover, the $g$-orthonormal trivialization vectors $V_j$ are left-invariant vector fields. So the Lie derivative $\Lie_{V_j} \vol_g $ vanishes, and hence so does the divergence $\divergence_g V_j$ and the last terms on the right-hand side of \eqref{TorsionTraceFormulas}.
\end{remark}
In section \ref{SectionExplicitExample}, we exhibit an example in which the trace torsions $\tau_2$ and $\tau_3$ of the triality rotated connections are non-zero, even though $\nabla = \nabla^1$ is the Levi-Civita connection. Observe that with $\nabla = \nabla^g$, the trace torsions $\tau_r$ are canonical vector fields on $M$ that only depend on the choice of metric. They vanish when $g$ is bi-invariant, but not in general. Similarly, the skew-symmetric parts of the torsions of the rotated $(\nabla^g)^r$ are canonical 3-forms on $M$ that only depend on the choice of metric.

\section{Triality rotations as symmetries of moduli spaces}
\label{SymmetriesModuliSpaces}

\subsection{Gauge transformations and Yang-Mills functionals}

Let $\mathscr{A}_g$ be the space of $g$-compatible connections on $TM$. Let $\mathscr{G}$ be the group of gauge transformations of $(TM, g)$. Its elements are the smooth sections $\varphi$ of the bundle of orientation-preserving isometries of the fibres of $TM$. They act on connections by 
\begin{equation}
\label{GaugeTransformationConnection}
[\varphi (\nabla) ]_X Y \: :=\: \varphi \circ \nabla_X \circ \varphi^{-1} (Y) \ 
\end{equation}
and preserve $\mathscr{A}_g$. The moduli space of metric connections is the quotient $\mathscr{A}_g / \mathscr{G}$.
\begin{definition}
\label{DefinitionTrialityRotationConnections}
The triality rotation is the map $\mathcal{T} : \mathscr{A}_g \rightarrow \mathscr{A}_g$ that sends a metric connection $\nabla$ to the metric connection $\nabla^2$ defined in \eqref{DefinitionThreeConnections3}.
\end{definition}
The next theorem says that $\mathcal{T}$ is a transformation of order three on $\mathscr{A}_g$, has a clearly-defined invariant locus, and descends to an automorphism of the moduli space.
\begin{theorem}
\label{PropertiesTrialityRotationConnections}
Let $\nabla$ be any metric connection on $(TM, g)$. Then
\begin{equation}
\label{CyclicityTrialityRotations}
\mathcal{T} \nabla = \nabla^2  \qquad \quad  \mathcal{T} \mathcal{T} \nabla = \nabla^3 \qquad  \quad \mathcal{T} \mathcal{T}  \mathcal{T}  \nabla = \nabla = \nabla^1 \ ,
\end{equation}
where the $\nabla^r$ are the rotated metric connections of \eqref{DefinitionThreeConnections3}. Moreover, 
\begin{itemize}
\item[(a)] $\mathcal{T} \nabla = \nabla$ if and only if the maps $A^\nabla_X$ of \eqref{EndomorphismConnection} have values in the adjoint transformations $\mathrm{ad} (\su(3))$ inside $\mathfrak{so}(\su(3))$ for all vector fields $X$ and all points in $M$.
\item[(b)] $\mathcal{T}$ descends to a map of moduli spaces $\mathscr{A}_g / \mathscr{G} \rightarrow  \mathscr{A}_g / \mathscr{G}$.
\end{itemize}
\end{theorem}
 So the triple $\{\nabla^1, \nabla^2, \nabla^3\}$ is the orbit of $\nabla$ under the cyclic group $\{ \mathrm{Id}, \mathcal{T}, \mathcal{T}^2 \} \simeq \mathbb{Z}_3$. In particular, this orbit has either one or three distinct elements, never two. Theorem \ref{PropertiesTrialityRotationConnections} is proved in appendix \ref{AppendixSymmetriesModuliSpaces}.
 
When we identify $TM$ with $M \times \su(3)$ and vector fields are regarded as maps $M \rightarrow \su(3)$, as in section \ref{FibrewiseOkubo}, formulas \eqref{CondensedConnections3} and \eqref{rho3} show that $\mathcal{T}$ is implemented by the transformation $A \mapsto \rho^1_A$ of connection representations in $\mathfrak{so}(\su(3))$.

Now let again $\nabla$ be a metric connection on $(TM,g)$ and let $\{V_j\}$ be our usual $g$-orthonormal trivialization of $TM$ associated to a $\langle, \rangle$-orthonormal basis $\{v_j\}$ of $\su(3)$. The norm of the curvature tensor of $\nabla$ is given by
\[
|\,R^{\nabla} \,|_g^2 \: := \: \sum\nolimits_{i, j, k , l} \big[\,  g \big( R^{\nabla}(V_k, V_l) V_i \, , V_j \big) \big]^2   \ .
\]
 The Yang-Mills energy-momentum tensor \cite[ch. 9.3]{Bleecker}, up to normalization, is 
\[
\Theta^{\nabla}(X, Y) \: := \: \sum\nolimits_{i,j} \, g \big( R^{\nabla}(X, V_i)V_j \, ,\, R^{\nabla}(Y, V_i)V_j \big) \, - \, \tfrac{1}{4}\, g(X, Y)  \, |\,R^{\nabla} \,|_g^2 \ .
\]
 The Chern-Simons 3-form of $\nabla$ relative to the flat connection on $TM$ for which the $V_j$ are parallel is the real form \cite{CS, Freed}
\begin{equation}
 \mathrm{CS}(\nabla) \; := \;   \mathrm{tr}\big(\, A\wedge \dd A \,  + \,  \tfrac{2}{3}\, A\wedge A\wedge A \, \big) \ ,
\end{equation}
 where $A= A^\nabla$ is the 1-form on $M$ with values in $\mathfrak{so}(\su(3))$ given by \eqref{EndomorphismConnection}. The next result says that these three objects are invariant under triality rotations of $\nabla$.
\begin{proposition}
\label{PropInvarianceNormCurvartureChernSimons}
For every metric connection $\nabla$ on $(TM, g)$, we have
\[
 |\, R^{ \mathcal{T} \nabla} \, |_g^2 \: =\:  |\, R^{\nabla} \, |_g^2   \qquad \qquad  \Theta^{\mathcal{T} \nabla} \: =\: \Theta^{\nabla} \qquad \qquad  \mathrm{CS} ( \mathcal{T} \nabla)  \: =\: \mathrm{CS}(\nabla) \ .
 \]
So these tensors coincide for the rotated connections $\nabla^1$, $\nabla^2$, and $\nabla^3$.
 \end{proposition}
This proposition is proved in appendix \ref{AppendixSymmetriesModuliSpaces}. The proof relies on the following algebraic identities, which are also proved in appendix \ref{AppendixSymmetriesModuliSpaces}.
 \begin{lemma}
\label{LemmaTracesRho}
For all $A, B, C \in \mathfrak{so}(\su(3))$ and $r = 1, 2, 3$,
\begin{equation}
\label{TraceRhoRho}
\mathrm{tr}( \rho^r_A \, \rho^r_B ) \: = \: \mathrm{tr}(A B) \  \qquad \qquad \mathrm{tr}\big( \rho^r_A \, [\rho^r_B , \rho^r_C] \big) \: = \: \mathrm{tr}( A\, [B, C] ) \ ,
\end{equation}
where $\mathrm{tr}$ is the trace of transformations of the 8-dimensional $\su(3)$.
\end{lemma}
A particular consequence of proposition \ref{PropInvarianceNormCurvartureChernSimons} is that the Yang-Mills functional 
\begin{equation}
\label{YMFunctional}
\mathcal{E}^{YM}_g(\nabla) \: := \:   \frac{1}{2} \,  \int_M  \,   | R^\nabla |_g^2 \,  \, \mathrm{vol}_g \ 
\end{equation}
 is invariant under triality rotations $\nabla \mapsto \mathcal{T} \nabla$. Even its integrand density is. Hence $\mathcal{T}$ takes critical points of $\mathcal{E}^{YM}_g$ to other critical points.
 We get a 3-fold symmetry of the moduli space of solutions of the Yang-Mills equations on $(TM, g)$. 
 
 The transformation $\nabla \mapsto \mathcal{T} \nabla$ also preserves the Einstein-Yang-Mills functional
\begin{equation}
\label{EYMFunctional}
\mathcal{E}^{EYM}(g, \nabla) \: := \:   \frac{1}{2} \,  \int_M  \,( \, R_g \: + \:  \gamma\, | R^\nabla |_g^2\, ) \,  \, \mathrm{vol}_g \ ,
\end{equation}
 where $R_g$ denotes the scalar curvature of the Levi-Civita connection. The invariant tensor $\Theta^{\nabla}$ of proposition \ref{PropInvarianceNormCurvartureChernSimons} is the energy-momentum tensor of precisely this functional. It would be interesting to study $\mathcal{E}^{EYM}(g, \nabla)$ in the setting of Kaluza-Klein models on $M_{4} \times \SU(3)$ (e.g. \cite{Bap1, Bap3}). The existence of 3-point orbits of Yang-Mills connections on the internal space $\SU(3)$ inducing three non-isospectral internal Dirac operators with the same symmetry group is a suggestive fact. It could be useful to model generations of fermions with different masses and the same gauge representations.

 \subsection{Invariance of adjoint Yang-Mills-Higgs functionals}
 \label{SectionYangMillsHiggs}

 Denote by $\HH_g$ the space of sections of the bundle of $g$-skew endomorphisms of $TM$.  Under the identification $TM \simeq M \times \su(3)$ provided by our $g$-orthonormal frames $\{V_j \}$, this $\HH_g$ becomes the space of maps $M \rightarrow \mathfrak{so}(\su(3))$, where the inner product on $\su(3)$ is the Ad-invariant $\langle,  \rangle$. This section exploits the existence of a natural triality action on $\HH_g$.

Let $\{B_c \}$ be a basis of the transformations in $\mathfrak{so}(\su(3))$, so we can expand them as  $\phi = \sum_c \phi^c \, B_c$, where each $\phi^c$ is a real function on $M$. 
\begin{definition}
\label{DefinitionTrialityRotationHiggs}
The triality map $\mathcal{T} : \HH_g \rightarrow \HH_g$ is defined by $\mathcal{T} \phi =  \rho^1_{\phi} =\sum_c \phi^c \, \rho^1_{B_c}$, where the transformations $ \rho^1_{B_c} \in \mathfrak{so}(\su(3))$ are defined through \eqref{rho1}.
\end{definition}
Just as in the case of connections, proposition \ref{PropositionTrialityCompositions} guarantees that $\mathcal{T} \mathcal{T} \mathcal{T} \phi = \phi$. So we obtain an action on $\HH_g$ of the cyclic group $\{ \mathrm{Id}, \mathcal{T}, \mathcal{T}^2 \} \simeq \mathbb{Z}_3$. The space of transformations $\mathfrak{so}(\su(3))$ has a natural inner product, up to normalization, defined by the trace of the transformation. Denoting it by $( , )$, lemma \ref{LemmaTracesRho} shows that
\begin{equation}
\label{InvarianceNormHiggsTriality}
(\mathcal{T}\phi_1 , \mathcal{T}\phi_2)\: =\: \mathrm{tr}( \rho^1_{\phi_1} \, \rho^1_{\phi_2} ) \:= \: \mathrm{tr}( \phi_1 \, \phi_2 ) \:=\: (\phi_1 , \phi_2)
\end{equation}
as functions on $M$. So $\mathcal{T}$ preserves the pointwise inner product of Higgs fields. 

The bundle of $g$-skew endomorphisms of $TM$ has a special sub-bundle. It is defined by the fibre endomorphisms that map to adjoint transformations $\mathrm{ad}_{\su(3)} \subset \mathfrak{so}(\su(3))$ under the identification of $(TM, g)$ with $(M \times \su(3) ,  \langle ,  \rangle)$ provided by our $g$-orthonormal trivializations $\{V_j\}$. So we have a special subspace $\HH_g^{\mathrm{ad}} \subset \HH_g$ of Higgs fields. Then
\begin{equation}
\mathcal{T} \phi \; = \; \phi  \qquad  \iff   \qquad \phi \,\in\, \HH_g^{\mathrm{ad}} \ .
\end{equation}
This is a consequence of the algebraic fact that, for any $A$ in $\mathfrak{so}(\su(3))$, the identity $\rho^1_A = A$ is equivalent to $\rho^1_A = \rho^2_A$ and to $A \in \mathrm{ad}_{\su(3)}$.
This was justified in appendix \ref{AppendixSymmetriesModuliSpaces} in the proof of statement (a) of theorem \ref{PropertiesTrialityRotationConnections}. 

Any metric connection $\nabla$ on $TM$ determines a covariant derivative acting on $\HH_g$ by 
\begin{equation}
\label{CovariantDerivativeHiggs1}
\nabla_X \phi \; := \;  \sum\nolimits_c  \  (\partial_X \phi^c)\, B_c \; + \; \phi^c \, [A_X , B_c ]   \ ,
\end{equation}
where $A_X: M \rightarrow \mathfrak{so}(\su(3))$ is the map \eqref{EndomorphismConnection} determined by $\nabla$ on $TM$.
Then triality rotations transform $\nabla_X \phi$ naturally, in the sense that
\begin{align}
\label{EquivarianceCovDerivHiggsTriality}
\mathcal{T} (\nabla_X \phi) \; &= \;  \sum\nolimits_c  \  (\partial_X \phi^c)\, \rho^1_{B_c} \: + \: \phi^c \, \rho^1_{[A_X , B_c ]} \; = \;  \sum\nolimits_c  \  (\partial_X \phi^c)\, \rho^1_{B_c} \: + \: \phi^c \, [\rho^1_{A_X} ,   \rho^1_{B_c } ] \nonumber \linebr
&= \;  (\mathcal{T} \nabla)_X \big( \sum\nolimits_c  \: \phi^c \: \rho^1_{B_c}  \big) \; = \; (\mathcal{T} \nabla)_X (\mathcal{T} \phi) \ .
\end{align}
The second equality uses that the correspondence $A \mapsto \rho^1_A$ is one of the representations of $\mathfrak{spin}(8)$, hence it is a Lie algebra homomorphism. The third equality uses that the connection $\mathcal{T} \nabla = \nabla^2$ has the explicit formula \eqref{CondensedConnections3}.

The covariant derivative on $\HH_g$ can be regarded in another way. Since $\phi$ is a section of the bundle of $g$-skew endomorphisms of $TM$, for any vector field $Y$ the element $\phi (Y)$ is a vector field too. Then $\nabla_X \phi$ is the endomorphism in $\HH_g$ determined by
\begin{equation}
\label{CovariantDerivativeHiggs2}
(\nabla_X \phi) (Y) \; = \;  \nabla_X [\phi(Y)] \; - \; \phi (\nabla_X Y)   \ ,
\end{equation}
where the right-hand side only has covariant derivatives of vector fields. 

This viewpoint is especially useful when we consider gauge transformations. As before, such a $\varphi \in \mathscr{G}$ is a section of a bundle of endomorphisms of $TM$, so we can compose it with $\phi$ to define the gauge-transformed element $\varphi(\phi) :=  \varphi \circ \phi \circ \varphi^{-1}$ in $\HH_g$.
Combining this expression with \eqref{CovariantDerivativeHiggs2} and rule \eqref{GaugeTransformationConnection} for the gauge transformation of $\nabla$ on $TM$,  we get the equivariance property 
\begin{equation}
\label{GaugeEquivarianceHiggsCovariantDerivative}
\varphi(\nabla_X \phi) =  (\varphi \nabla)_X (\varphi \phi)
\end{equation}
 for the covariant derivative on $\HH_g$.

\begin{theorem}
\label{TheorTrialityHiggsModuliSpace}
The triality map of definition \ref{DefinitionTrialityRotationHiggs} descends to a map of moduli spaces $\HH_g / \mathscr{G} \rightarrow \HH_g / \mathscr{G}$. Moreover, its combination with the triality rotation of connections in \ref{DefinitionTrialityRotationConnections} descends to a map $(\mathscr{A}_g \times \HH_g) / \mathscr{G} \rightarrow (\mathscr{A}_g \times \HH_g )/ \mathscr{G}$.
 \end{theorem}
 This theorem means that, given a gauge transformation $\varphi \in  \mathscr{G}$, there exists a second gauge transformation $\varphi'$ such that $\mathcal{T}(\varphi \phi)  = \varphi' ( \mathcal{T} \phi)$. Moreover, the triality rotation of connections also satisfies $\mathcal{T}(\varphi \nabla)  = \varphi' ( \mathcal{T} \nabla)$ for {\it the same} $\varphi'$. The second transformation is determined by $\varphi$ and is independent of $\phi$ and $\nabla$. It is the transformation $\varphi' = ({\mathcal I}^+)^{-1} \circ \tilde{\varphi} \circ {\mathcal I}^+$ constructed in the proof of theorem \ref{PropertiesTrialityRotationConnections} in appendix \ref{AppendixSymmetriesModuliSpaces}. The proof of theorem \ref{TheorTrialityHiggsModuliSpace} is also given in appendix \ref{AppendixSymmetriesModuliSpaces}.

Finally, consider the Yang-Mills-Higgs functional on pairs $(\nabla, \phi)$ on $(TM, g)$:
\begin{equation}
\label{YMHFunctional}
\mathcal{E}^{YMH}_g(\nabla, \phi) \: := \:   \frac{1}{2} \,  \int_M \Big( \,   | R^\nabla |_g^2 \, + \, \sum\nolimits_j | \nabla_{V_j} \phi |^2 \, + \, f(|\phi|^2) \Big) \, \mathrm{vol}_g \ .
\end{equation}
From the last section we know that $| R^{{\mathcal T} \nabla} |^2 = | R^{\nabla} |^2$. From \eqref{InvarianceNormHiggsTriality} we known that $|{\mathcal T}\phi|^2 = |\phi|^2$. Combining \eqref{EquivarianceCovDerivHiggsTriality} with  \eqref{InvarianceNormHiggsTriality} we know that $|({\mathcal T}\nabla)_{V_j}({\mathcal T} \phi) |^2 =  | \nabla_{V_j} \phi |^2$. Therefore, the whole Lagrangian density of $\mathcal{E}^{YMH}_g$ is invariant under triality rotations $(\nabla, \phi) \mapsto ({\mathcal T}\nabla, {\mathcal T}\phi)$, and hence so is the functional itself. In particular, $(\nabla, \phi)$ is a critical configuration of the functional if and only if $({\mathcal T}\nabla, {\mathcal T}\phi)$ is. So we get a 3-fold symmetry of the moduli space of solutions of the Yang-Mills-Higgs equations on $(TM, g)$.

\section{Three Dirac-type operators on $S$}
\label{ThreeDiracOperators}

\subsection{The operators $\Gamma (S^+) \rightarrow \Gamma (S^-)$ and their $L^2$-adjoints}

We have established that the data $(g,\nabla)$ determines three natural connections $\nabla^r$ on the chiral spinor bundle $S^+$ that are compatible with the inner product $\langle , \rangle$ of spinors. After contraction through Clifford multiplication, these connections then determine three linear, first-order differential operators. As usual, 
\begin{equation}
\label{DefinitioncDr}
\cD_r : \Gamma (S^+) \rightarrow \Gamma (S^-) \ , \qquad  \cD_r \psi^+ \; := \; \sum\nolimits_j \, V_j \cdot \nabla^r_{V_j} \psi^+ \ , \qquad r= 1, 2, 3, 
\end{equation}
where $\{ V_j\}$ denotes any $g$-orthonormal trivialization of $TM$. Since $\nabla^1$ is the standard spinorial lift to $S^+$ of the connection $\nabla$ on $TM$, the operator $\cD_1$ is the standard Dirac operator on $S^+$ determined by $\nabla$. The operators $\cD_2$ and $\cD_3$ are different. However, all three manifestly share the same principal symbol, given by Clifford multiplication. In particular, all three are elliptic.
\begin{proposition}
\label{PropExplicitD}
Regarding chiral spinors as maps $M \rightarrow \su(3)$, the three Dirac-type operators
$\cD_r : \Gamma (S^+) \rightarrow \Gamma(S^-)$ are given by
\begin{align}
\label{ExplicitDr}
\cD_1 \psi^+ \; &= \; \sum_j \, v_j \ok \partial_{V_j}  \psi^+ \; - \; \frac{1}{4} \, \sum_{i, j, k} \, \Gamma_{ijk} \ v_i \ok \big[ \, ( v_k \ok \psi^+ ) \ok v_j \, \big]   \linebr
\cD_2 \psi^+ \; &= \; \sum_j \, v_j \ok \partial_{V_j}  \psi^+  \; - \; \frac{1}{4} \, \sum_{i, j, k} \, \Gamma_{ijk} \ v_i \ok \big[ \, v_j \ok ( \psi^+ \ok v_k ) \,\big]  \nonumber \linebr
\cD_3 \psi^+ \; &= \; \sum_j \, v_j \ok \partial_{V_j}  \psi^+ \; + \;  \sum_{i, j, k} \,\Gamma_{ijk} \; \langle \psi^+ , v_j \rangle \  v_i \ok v_k  \ . \nonumber
\end{align}
Here the $V_j$ form the $g$-orthonormal trivialization of $TM$ determined by the $\langle , \rangle$-orthonormal basis $\{ v_j \}$ of $\su(3)$, as in \eqref{DefinitionVj}, and we have abbreviated $\Gamma_{ijk} = g( \nabla_{V_i} V_j , \, V_k )$.
\end{proposition}
\begin{proof}
From the first equality in \eqref{TraditionalCliffordMultiplication}, definition \eqref{DefinitioncDr} of $\cD_r$ can be written as $\cD_r \psi^+ = \sum_j V_j \dok \nabla^r_{V_j} \psi^+$. But $\alpha^{-1}(V_j) = v_j^\LL$ from \eqref{DefinitionVj}, and the left-invariant vector field $v_j^\LL$ translates as the map from $M$ with constant value $v_j \in \su(3)$. So definitions \eqref{DeformedCliffordMultiplication} of the $\dok$-products give
\begin{equation}
\label{Vjdokpsi}
V_j \dok \psi^+ \; = \; v_j \ok \psi^+  \qquad\qquad V_j \dok \psi^- \; = \; - \, \psi^- \ok v_j 
\end{equation}
for any sections $\psi^\pm$ of $S^\pm$. In particular $\cD_r \psi^+ = \sum_j v_j \ok \nabla^r_{V_j} \psi^+$. Substituting in this expression the formulas \eqref{ThreeConnections} for the three connections $\nabla^r$, we get directly the formulas for $\cD_r \psi^+$ given in the proposition. The only point of note is that, since $v_j$ is a constant vector in $\su (3)$ that does not depend on the point on $M$, 
\begin{equation*}
\label{DerivativeThroughOkubo}
\partial_{V_i} \, ( \psi \ok v_j ) \; = \; ( \partial_{V_i} \psi ) \ok v_j  \qquad \quad \partial_{V_i} \, ( v_j \ok \psi ) \; = \; v_j \ok ( \partial_{V_i} \psi ) \ .  \qedhere
\end{equation*} 
\end{proof}

The metric $g$ on $TM$ and the inner-products $\langle , \rangle$ on $S^\pm$ allow us to define $L^2$-inner products of spinors and vector fields on $M$ by integrating with respect to the Riemannian volume form $\vol_g$.
So we can define the $L^2$-adjoint operators $\cD_r^\dag : \Gamma(S^-) \rightarrow \Gamma(S^+)$ through 
\begin{equation}
\label{DefinitionL2Adjoint}
\blangle \cD_r^\dag \psi^- , \psi^+ \brangle_{L^2}  \; = \;  \blangle \psi^-, \cD_r \psi^+ \brangle_{L^2} \ .
\end{equation}
\begin{proposition}
\label{AlternativeExplicitD1Adjoint}
The $L^2$-adjoints $\cD_r^\dag$ are given by
\begin{align}
\cD_r^\dag \psi^-  \, &= \, \sum\nolimits_j  \big( V_j \cdot \nabla^r_{V_j} \psi^-  \big) \, + \,  \tau_r \cdot \psi^-   \ , 
\end{align}
where $\cdot$ is Clifford multiplication and $\tau_r$ is the vector field $g$-dual to the trace torsion of the connection $\nabla^r$ on $TM$. 
\end{proposition}
This proposition is proved in appendix \ref{AppendixKLDerivatives}.

\subsection{Three self-adjoint operators $\Gamma (S) \rightarrow \Gamma (S)$}
\label{ThreeDiracOperators2}

The triality rotated connections $\nabla^r$ and their spinorial lifts define three classical Dirac operators on $S = S^+ \oplus S^-$ through the usual expression $D _r\psi = \sum\nolimits_j   V_j \cdot \nabla^r_{V_j} \psi$. Since the $\nabla^r$ have non-zero trace torsion, in general, the classical operators $D_r$ will not be self-adjoint (see \cite{FS79, Agricola} or our proposition \ref{AlternativeExplicitD1Adjoint}). However, after choosing a preferred chiral spinor bundle, which is equivalent to choosing a preferred orientation of $M$, it is possible to define three self-adjoint, Dirac-type operators acting on $S$ by the adjoint-completion
\begin{equation}
\sD_r : \Gamma (S) \rightarrow \Gamma (S) \ ,  \qquad  \sD_r \begin{bmatrix} \psi^+ \\ \psi^- \end{bmatrix} :=  \begin{bmatrix}  \cD_r^\dag\, \psi^-  \\ \cD_r\, \psi^+ \end{bmatrix} \ .
\end{equation}
for $r = 1, 2, 3$. Using proposition \ref{AlternativeExplicitD1Adjoint}, it is clear that 
\begin{equation}
\label{ClassicalVsAlternativeDiracOperators}
\sD_r \psi \; = \; D_r \psi \:  + \:  \tau_r \cdot \psi^-  \; = \; D_r \psi  \:  + \: \tfrac{1}{2}\, \tau_r \cdot \psi  \:  - \: \tfrac{1}{2}\, \tau_r \cdot \Gamma_S\,  \psi  \ , 
\end{equation}
where $\Gamma_S$ is the chiral operator acting on the chiral spinor bundles by $\Gamma_S \psi^\pm = \pm\, \psi^\pm$. So $\sD_r$ differs from $D_r$ only when the torsion trace $\tau_r$ is non-zero. The difference are scalar terms involving products of one and seven gamma matrices, since $\Gamma_S$ can be expressed as a product of eight gamma matrices.
\begin{remark}
The difference $\cD_r - D^+_r$ has no derivatives, so the operator $\cD_r$ is homotopic to the classical $D^+_r$ through elliptic operators, and hence has the same index as $D^+_r$. By the Atiyah-Singer theorem this is zero, since $\SU(3)$ is a parallelizable manifold \cite{LM}.
\end{remark}
\begin{lemma}
Expanding $\sD_r$ in sums of products of gamma matrices, 
\begin{equation}
\sD_r \psi \; = \; \sum_j  \gamma_j \, \big( \partial_{V_j}  \, +  \, \tfrac{1}{2}\, (\divergence_g V_j)\,  \big) \psi \: + \:  \tfrac{1}{4}\, \sum_{j \neq k\neq l}  \Gamma^r_{jkl} \, \gamma_j  \gamma_k  \gamma_l \, \psi   \: - \: \tfrac{1}{2}\, \tau_r \cdot \Gamma_S\,  \psi  \ ,     \label{ExpansionsDr}
\end{equation}
where $j \neq k\neq  l$ means distinct indices and we have abreviated $\Gamma^r_{jkl} = g(\nabla^r_{V_j} V_k , V_l)$. In particular, the terms with one gamma matrix coincide for all three connections $\nabla^r$.
\end{lemma}
\begin{proof}
Using the standard expression for a metric connection in terms of gamma matrices \cite{Bourguignon}, the classical Dirac operator of $\nabla^r$ can be expanded as
\begin{align}
D_r \psi \: &= \: \sum\nolimits_j  \gamma_j \, \partial_{V_j} \psi \:+ \: \tfrac{1}{4}\, \sum\nolimits_j \sum\nolimits_{k \neq l}\,  \Gamma^r_{jkl} \, \gamma_j  \gamma_k  \gamma_l \, \psi    \label{ExpansionClassicalDirac} \linebr
&= \: \sum\nolimits_j  \gamma_j \, \partial_{V_j} \psi  \: - \: \tfrac{1}{2}\, \sum\nolimits_{j,l} \Gamma^r_{jjl} \gamma_l \psi \:+ \: \tfrac{1}{4}\, \sum\nolimits_{j \neq k\neq l}  \,  \Gamma^r_{jkl} \, \gamma_j  \gamma_k  \gamma_l \, \psi  \nonumber  \ ,
\end{align}
Using the Levi-Civita connection and the contorsion to write $\nabla^r = \nabla^g + K^{\nabla^r}$, calculations similar to \eqref{cD1Aux4} , \eqref{cD1Aux2} and \eqref{cD1Aux3} yield
\begin{align}
\sum\nolimits_j \, \Gamma^r_{jjl} \: &= \:  \sum\nolimits_j g( \nabla^r_{V_j} V_j ,  V_l)\: =\:  \sum\nolimits_j \, g\big( \nabla^g_{V_j} V_j + K^{\nabla^r} (V_j, V_j) \, ,  V_l \big)  \nonumber \linebr
&= \: - \, \divergence_g (V_l) \: +\: \tau^{\nabla^r}(V_l) \: =\:  - \, \divergence_g (V_l) \: +\: g(\tau^r, V_l) \ .
\end{align}
Substituting this expression and \eqref{ExpansionClassicalDirac} into formula \eqref{ClassicalVsAlternativeDiracOperators},  we see that the two terms with $\tau_r \cdot \psi$ cancel, and hence we obtain the identity of the lemma. 
\end{proof}

Finally, note that the operators $\sD_{t_1, t_2, t_3}$ constructed with the connections $\nabla^{t_1, t_2, t_3}$ on $S^+$ given in  \eqref{LinearCombinationsConnections} will also be elliptic, formally self-adjoint operators on spinors. They have the same principal symbol as the $\sD_r$. A special case is having equal coefficients, due to the identity of the connections $\nabla^r$ in lemma \ref{Sum3Reps}. Moreover, there is a very natural, elliptic, formally self-adjoint operator $\sD_1 \oplus \sD_2 \oplus \sD_3$ on the sum of three copies $S\oplus S \oplus S$ of the spinor bundle. It involves no choices besides orientation and the data $(g, \nabla)$.

\section{Triality rotations preserve symmetries}
 \label{KLDerivatives}

\subsection{Lie derivatives of the rotated connections}

 The purpose of this section is to show that every infinitesimal symmetry that preserves the data $(g, \nabla)$ and the group structure on $M$ will also preserve the three connections $\nabla^r$ on each of the bundles $TM$ and $S^\pm$,  as well as the Dirac operators $\sD_r$. This reinforces the naturalness of the previous constructions in the paper. The main results are theorem \ref{TrialitiyLieDerivativeConnection} for $TM$, proposition \ref{TrialitiyKLLieDerivatives} for $S^\pm$, and proposition \ref{PropDiracCommutesWithKosmann} for the Dirac operators.

Recall the notion of Lie derivative of a connection.
Given a vector field $X$ on $M$ and a metric connection $\nabla$ on $TM$, define the $(2, 1)$-tensor $\Lie_X \nabla$ by
\[
(\Lie_X \nabla)_Y Z \: := \:   [X,  \nabla_Y Z ] \: - \:  \nabla_{[X, Y]} Z \: - \:  \nabla_Y [X, Z] \ .
\]
Observe that, when $X$ is Killing, the $g$-compatibility of $\nabla$ implies that 
\[
g\big((\Lie_X \nabla)_Y Z, W \big) \; = \; - \: g\big(Z, (\Lie_X \nabla)_Y W \big)  \ .
\]
So $(\Lie_X \nabla)_Y$ is a $g$-skew endomorphism of $TM$. In the parlance of section \ref{SectionYangMillsHiggs}, it is an adjoint Higgs field in $\HH_g$. In particular, it has a natural triality rotation ${\mathcal T} [(\Lie_X \nabla)_Y]$. The main result of this section is the following theorem, proved in appendix \ref{AppendixKLDerivatives}.
\begin{theorem}
\label{TrialitiyLieDerivativeConnection}
Let  $X$ be a left or right-invariant Killing vector field on $(M,g)$ and let $\nabla$ be a metric connection on $TM$. Then the triality rotated connection ${\mathcal T} \nabla$ satisfies 
\[
(\Lie_X \,  {\mathcal T} \nabla)_Y  \: = \:  {\mathcal T} [(\Lie_X \nabla)_Y ] \ .
\]
In particular, $(\Lie_X \,  {\mathcal T} \nabla)_Y$ vanishes if and only if $(\Lie_X \nabla)_Y$ does. Since $\nabla$ is arbitrary, analogous relations hold for the second triality-rotated connection ${\mathcal T} {\mathcal T} \nabla$.
\end{theorem}
The proof of this result requires an understanding of how Lie derivatives operate when we regard vector fields as maps $M \rightarrow \su(3)$. This is our usual correspondence $Y = \sum_j g(Y, V_j) \, v_j = \sum_j Y^j \, v_j$, where $\{V_j\}$ is the $g$-orthonormal trivialization of $TM$ associated to a $\langle , \rangle$-orthonormal basis $\{v_j\}$ of $\su(3)$, as in \eqref{DefinitionVj}.  Denote by $u^\LL$ and $u^\RR$ the left- and right-invariant vector fields on $M$ extending $u\in \su(3)$. Then we have:
\begin{lemma}
\label{LieDerivativesInvariantKillingFields}
Regard vector fields as maps $M \rightarrow \su(3)$, with $Y (x) = \sum_j \, Y^j (x)\, v_j$. Suppose that $u^\LL$ and $w^\RR$ are invariant Killing vector fields. Then
\begin{align}
\Lie_{u^\LL} Y \ &= \ \partial_{u^\LL} Y \ + \ [u, Y] \ = \ \sum\nolimits_j   \   (\partial_{u^\LL} Y^j ) \, v_j \: + \:    Y^j  \,   [u, v_j]        \label{ExpressionsLieDerivativeKilling}   \linebr
\Lie_{w^\RR} Y \ &= \ \partial_{w^\RR} Y \ = \ \sum\nolimits_j   \   (\partial_{w^\RR} Y^j ) \, v_j \  \ .  \nonumber
\end{align}
\end{lemma}
This lemma is also proved in appendix \ref{AppendixKLDerivatives}.

\subsection{Kosmann-Lichnerowicz derivatives of spinors}

Let $X$ be a vector field on $M$ and let $\psi$ be a section of $S= S^+ \oplus S^-$. The Kosmann-Lichnerowicz derivative acting on spinors \cite{Kosmann, Lich} is defined by 
\begin{align}
\label{DefinitionKLDerivative}
\cL_X \, \psi \; &:=\;  \nabla^g_X \psi \; -\; \frac{1}{8}\, \sum_{i, j} \big[  g(\nabla^g_{V_i} X, V_j) - g(\nabla^g_{V_j} X, V_i) \big] \, V_i \cdot V_j  \cdot \psi     \linebr
&\, \, = \;  \partial_X \tilde{\psi} \;+ \;\frac{1}{8}\,  \sum_{i, j} \big\{  g([X, V_i], V_j) - g([X, V_j], V_i) \big\} \,   V_i \cdot V_j  \cdot \tilde{\psi} \ . \nonumber
\end{align}
Here $\nabla^g$ denotes both the Levi-Civita connection on $TM$ and its standard lift to the spinor bundle. The dot denotes the Clifford action of vectors on spinors.  The vectors $V_j$ form a $g$-orthonormal, oriented, local trivialization of $TM$. The second line is a local expression using the representative $\tilde{\psi}$ of $\psi$ with respect to the trivialization of $S$ induced by $\{ V_j\}$. It follows from a direct application of the Koszul formula for $\nabla^g$. When $X$ is Killing, the coefficients $g(\nabla^g_{V_i} X, V_j)$ and $g([X, V_i], V_j)$ are both anti-symmetric in $i$ and $j$.

In the case $M=\SU(3)$, we can take our usual $g$-orthonormal trivializations of the form $\{ V_j = \alpha (v_j^\LL)\}$, where $\{v_j\}$ is a $\langle , \rangle$-orthonormal basis of $\su(3)$. Then Clifford multiplication is modelled by the Okubo product on $\su(3)$ and the Kosmann-Lichnerowicz derivatives along invariant vector fields on $M$ take the following, more explicit form:
\begin{proposition}
\label{ExplicitKLDerivativesNonKilling}
Regard Weyl spinors as maps $M \rightarrow \su(3)$, with $\psi (x) = \sum_j \, \psi^j (x)\, v_j$. Denote by $u^\LL$ and $u^\RR$ the extensions of $u \in \su(3)$ as left-invariant and right-invariant vector fields on $M$. Then the Kosmann-Lichnerowicz derivatives act on spinors as
\begin{align}
\cL_{u^\LL} \, \psi^+ \;  &= \;   \sum_j \: \partial_{u^\LL} (\psi^+)^j  \, v_j \:+ \:  (\psi^+)^j \, [u, v_j ] \;- \; \frac{1}{4}\,  \sum_{i \neq j} \,  \big\langle \, \alpha^{-1} (\Lie_{u^\LL} \alpha) ( v_i^\LL) , \,v_j^\LL  \,\big\rangle  \,  ( v_j  \ok \psi^+ ) \ok v_i  \nonumber  \linebr
\cL_{u^\RR} \, \psi^+ \;  &= \;   \sum_j \:  \partial_{u^\RR} (\psi^+)^j  \, v_j \; - \; \frac{1}{4}\,  \sum_{i \neq j} \,  \big\langle \, \alpha^{-1} (\Lie_{u^\RR} \alpha) ( v_i^\LL) , \,v_j^\LL  \,\big\rangle  \,    ( v_j  \ok \psi^+ ) \ok v_i  \ .
\end{align}
where $\alpha$ is the automorphism of $TM$ defined in \eqref{DefinitionAlpha}. For spinors $\psi^-$ in $S^-$ the formula is analogous, except that $( v_j  \ok \psi^+ ) \ok v_i $ should be substituted by $v_i  \ok (\psi^-  \ok v_j)$. 
\end{proposition}
The formulas for $\cL_{u^\LL} $ and $\cL_{u^\RR}$ are asymmetric because the trivialization of $S$ was induced by a left-invariant frame on $TM$, not a right-invariant one. The proposition is proved in appendix \ref{AppendixKLDerivatives}.
\begin{corollary}
\label{ExplicitKLDerivatives}
When $u^\LL$ and $w^\RR$ are Killing with respect to $g$, we get simply
\begin{align}
\label{ExpressionsKLDerivativesKilling}
\cL_{u^\LL} \psi^\pm \ &= \ \partial_{u^\LL} \psi^\pm \ + \ [u, \psi^\pm] \    \linebr
\cL_{w^\RR} \psi^\pm \ &= \ \partial_{w^\RR} \psi^\pm  \ . \nonumber
\end{align}
\end{corollary}
This is because $\Lie_{u^\LL} \alpha = \Lie_{w^\RR} \alpha = 0$ for invariant Killing fields, as in the proof of lemma \ref{LieDerivativesInvariantKillingFields}. Now observe that, applied to the respective representative maps $M \rightarrow \su(3)$, expressions \eqref{ExpressionsKLDerivativesKilling} for the Kosmann-Lichnerowicz derivatives are formally identical to expressions \eqref{ExpressionsLieDerivativeKilling} for the Lie derivatives. Moreover, expressions \eqref{CondensedConnections}, \eqref{CondensedConnections2} and \eqref{CondensedConnections3} for the covariant derivatives $\nabla^r_X$ on $TM$, $S^+$ and $S^-$ are also formally identical, up to permutations of the index $r$. This implies that all the calculations in the proof of theorem \ref{TrialitiyLieDerivativeConnection} in appendix \ref{AppendixKLDerivatives} can be duplicated to show that:
\begin{proposition}
\label{TrialitiyKLLieDerivatives}
Let  $X$ be a left or right-invariant Killing vector field on $(M,g)$ and let $\nabla$ be a metric connection on $TM$. Regard Weyl spinors $\psi^\pm$ as maps $M \rightarrow \su(3)$. Then the spinorial lifts of the triality rotated connections $\nabla^r$ satisfy the identities
\begin{align}
\cL_{X}  ( \nabla^r_Y \psi^+ )  - \nabla^r_Y (\cL_{X} \psi^+ ) -  \nabla^r_{[X, Y ]} \psi^+ \ &= \  \rho^{r}_{( \Lie_X \nabla)_Y} (\psi^+)     \linebr 
\cL_{X}  ( \nabla^r_Y \psi^- )  - \nabla^r_Y (\cL_{X} \psi^- ) -  \nabla^r_{[X, Y ]} \psi^- \ &= \  \rho^{r+1}_{( \Lie_X \nabla)_Y} (\psi^-)    \nonumber   \ ,
\end{align}
where the $\rho^r$ are defined in \eqref{rho1}, as usual, and the indices $r$ are taken modulo 3.
\end{proposition}
Thus, when $( \Lie_X \nabla) = 0$, the derivative $\cL_X$ has the natural commutation relation with the three connections $\nabla^r$ on the chiral spinor bundles.

We have just established how the Kosmann-Lichnerowicz derivative $\cL_X$ commutes with the three covariant derivatives $\nabla^r_Y$ of spinors. Now we will describe how it commutes with the respective Dirac-type operators. The commutation of $\cL_X$ with the classical Dirac operator for Killing $X$ is a standard result \cite{Kosmann}.
\begin{proposition}
\label{PropDiracCommutesWithKosmann}
Let $X$ be a left-invariant or right-invariant vector field on $M$ such that $\Lie_X g = 0$ and $\Lie_X \nabla = 0$. Then, for $r = 1, 2, 3$, the derivative $\cL_X$ satisfies 
\[
\cL_X \, \cD_r \, \psi^+ \; = \; \cD_r \, \cL_X \,  \psi^+ \qquad
\cL_X \, \cD_r^\dag  \, \psi^- \; = \; \cD_r^\dag \, \cL_X  \, \psi^-  \qquad
\cL_X \, \sD_r \, \psi \; = \; \sD_r \, \cL_X \, \psi \ .
\]
\end{proposition}
\begin{proof}
The Kosmann-Lichnerowicz derivative always preserves the chiral spinor bundles $S^\pm$. When $X$ is Killing, it is also compatible with Clifford multiplication \cite{Kosmann}. Combining this fact with proposition \ref{TrialitiyKLLieDerivatives}, we calculate that
\begin{align}
\cL_X \big( \sum\nolimits_j V_j \cdot \nabla^r_{V_j} \psi^+ \big) \: &= \:  \sum\nolimits_j \:  [X, V_j] \cdot \nabla^r_{V_j} \psi^+ \: + \:  V_j \cdot  \cL_X \nabla^r_{V_j} \psi^+   \nonumber  \linebr
&= \:  \sum\nolimits_j \:    [X, V_j] \cdot \nabla^r_{V_j} \psi^+ \: + \:  V_j \cdot  \nabla^r_{[X, V_j]} \psi^+ \: + \:  V_j \cdot  \nabla^r_{V_j} ( \cL_X \psi^+)  \nonumber  \linebr
&= \:   \cD_r \, \cL_X \,  \psi^+ \ .
\end{align}
The last equality uses that $\{V_j \}$ is $g$-orthonormal and $X$ is Killing, and in particular
\begin{align}
\sum\nolimits_j \:   V_j \cdot  \nabla^r_{[X, V_j]} \psi^+ \:  &= \: \sum\nolimits_{j, k} \, g(\, [X, V_j], V_k\,) \:  V_j \cdot  \nabla^r_{V_k} \psi^+  \nonumber  \linebr
 &= \: - \sum\nolimits_{j, k} \, g(\, V_j, [X, V_k] \,) \:  V_j \cdot  \nabla^r_{V_k} \psi^+  \nonumber  \linebr
 &= \: - \sum\nolimits_{k} \: [X, V_k] \cdot \nabla^r_{V_k} \psi^+ \ .
\end{align}
This proves the commutation of $\cL_X$ with $\cD_r$. For the $L^2$-adjoints $\cD_r^\dag$, using twice the fact that $\cL_X$ is compatible with the inner product of spinors, the fact  that $X$ is Killing and hence $\Lie_X \vol_g = 0$, and Stokes' theorem, we have
\begin{align}
\int_M \langle \, \cL_X \, \cD_r^\dag  \, \psi^-, \, \psi^+ \, \rangle \, \vol_g \: &= \: \int_M  \Big[ \partial_X \langle \cD_r^\dag   \psi^-,  \psi^+ \rangle \,  - \, \langle  \cD_r^\dag   \psi^-,  \cL_X  \psi^+ \rangle \Big] \, \vol_g \nonumber  \linebr
&= \: \int_M  \Lie_X  \Big[  \langle \cD_r^\dag   \psi^-,  \psi^+ \rangle \, \vol_g  \Big] \, - \, \langle    \psi^-,  \cD_r \cL_X  \psi^+ \rangle  \, \vol_g  \nonumber  \linebr
&= \: - \int_M \,  \langle    \psi^-,  \cL_X  \cD_r  \psi^+ \rangle \, \vol_g  \nonumber  \linebr
&= \:  \int_M \,  \langle   \cL_X  \psi^-,   \cD_r  \psi^+ \rangle  \, \vol_g  \nonumber  \linebr
&= \:  \int_M \,  \langle  \, \cD_r ^\dag\, \cL_X \, \psi^-,  \, \psi^+\, \rangle  \, \vol_g  \nonumber  \ .
\end{align}
This proves the commutation of $\cL_X$ with $\cD_r^\dag$. The commutation of $\cL_X$ with $\sD_r$ follows from the previous two.
\end{proof}
\begin{remark}
When $g$ is a left-invariant metric, the identity component of the isometry group is contained in the group of left and right multiplications \cite{OT}. Moreover, all smooth conformal vector fields are in fact Killing. This is because such a vector field will generate a 1-parameter group of conformal transformations, as $M$ is compact. But $M$ is homogeneous and not a sphere, so a theorem of Goldberg and Kobayashi \cite{GoKo} guarantees that those transformations are isometries, and the vector field that generates them must be Killing.
\end{remark}

\section{An explicit non-isospectral example}
 \label{SectionExplicitExample}
 
 We start this section by defining a one-parameter deformation $g$ of the bi-invariant metric $\langle , \rangle$ on $TM$. After deformation, the identity component of the isometry group of $g$ becomes $(\SU(3) \times \SU(2) \times \mathrm{U}(1)) / \mathbb{Z}_6$. Then we consider the three self-adjoint Dirac operators $\sD_r$ associated with the Levi-Civita connection $\nabla^g$ and calculate their simplest eigeinvalues. These differ for the three operators and, as a consequence, we observe that $\sD_2$ and $\sD_3$ cannot be isospectral to the classical Dirac operator $\sD = \sD_1$ of $g$. We also calculate the explicit trace torsion vectors $\tau_2$ and $\tau_3$ of the triality-rotated connections ${\mathcal T}\nabla^g$ and ${\mathcal T}{\mathcal T}\nabla^g$. We conclude that they vanish for the bi-invariant metric, but are non-zero for the deformed metrics $g$. Most of the actual computations are relegated to appendix \ref{AppendixExplicitExample}.

Let $\{ \lambda_j \}$ be the usual set of eight Gell-Mann matrices. Define 
\begin{equation}
\label{NormalizedBasissu(3)}
v_j \; := \; \frac{i}{2a}\, \lambda_j \qquad  \quad   \text{for} \ \ j = 1, \ldots, 8 \ .
\end{equation}
These matrices lie in $\su (3)$ and satisfy $ \langle v_j , v_k \rangle   =  \tfrac{1}{2}\, \mathrm{Tr} (\lambda_j \lambda_k) =  \delta_{jk}$, by \eqref{AdInvariantProduct}.
So $\{ v_j \}$ is a $\langle , \rangle$-orthonormal basis of $\su(3)$. Define the subspaces
\begin{equation}
\label{E8Splitting}
 \mathfrak{u} (1)\: := \: \mathrm{span}\{ v_8\} \ , \qquad \su(2) \: := \: \mathrm{span}\{ v_1, v_2, v_3\} \ , \qquad \mathfrak{m} \: := \: \mathrm{span}\{ v_4, v_5, v_6, v_7\} \ .
\end{equation}
There is a well-known $\langle , \rangle$-orthogonal decomposition
\begin{equation}
\su(3) \: =  \: \mathfrak{u} (1) \, \oplus\, \su(2) \,  \oplus \,  \mathfrak{m} \ .
\end{equation} 
All the elements in the subalgebra $\mathfrak{u}(2) := \mathfrak{u} (1) \, \oplus\, \su(2)$ commute with $v_8$. The subgroup $U(2) \subset \SU(3)$ generated by $\mathfrak{u}(2)$ preserves the decomposition under the adjoint action.

Let $v_j^\LL$ denote the left-invariant extension of $v_j$ to a vector field on $M$. Define a left-invariant metric $g$ on $M$ by declaring that the vectors fields
\begin{equation}
V_8 \; = \; \mu^{-1} \, v_8^\LL  \qquad  \qquad  V_j \; = \; v_j^\LL  \ \ \  \text{for} \ \ j=1, \ldots, 7
\end{equation}
 form a $g$-orthonormal trivialization of $TM$. Here $\mu$ is a positive constant. The scalar curvature of this metric is a special case of the one given in \cite{Bap1}, section 3.7. It is
\begin{equation}
R_g \: = \:  \frac{3}{ 4\, a^2}  \, ( 9  -  \mu^{2})  \ .
\end{equation}
The constant $\mu$ measures the relative size of the circle generated by $v_8$ inside $\SU(3)$. When $\mu \rightarrow 0^+$ with fixed $a$, that circle becomes arbitrarily small but the scalar curvature  remains bounded. The constant $a \in \mathbb{R}^+$ determines the normalization of the bi-invariant metric $\langle , \rangle$ of which $g$ is a deformation, as is explicit in \eqref{AdInvariantProduct}.

Regarding chiral spinors as maps $M \rightarrow \su(3)$, define a spinor $\psi_0$ by declaring that its two chiral components are the constant maps $\psi_0^+ = \psi^-_0 = v_8$. Then we have:
\begin{proposition}
\label{PropositionExplicitEigenvalues}
For the Levi-Civita connection of the metric $g$ described above, the Dirac-type operators $\sD_r$ satisfy
\begin{align}
\sD_1 \psi_0 \; &= \;  -\, \frac{3}{2a} \, \psi_0   \qquad \qquad \quad \sD_3 \psi_0 \; = \;  -\, \frac{3 \, \mu}{2\, a} \, \psi_0   \linebr
\sD_2 \psi_0 \; &= \;  -\, \frac{3}{4\, a} \, ( 1 + \mu ) \, \psi_0  \nonumber \ .
\end{align}
\end{proposition}
\begin{remark}
When $\mu =1$ the metric $g$ is bi-invariant and the Levi-Civita connection $\nabla^g$ is in the Cartan-Schouten line. So, from corollary \ref{CoincidenceConnectionsBiinvariant}, we know that $\sD_1 = \sD_2 = \sD_3$. Hence the three eigenvalues computed here must coincide when $\mu = 1$, as they do. Observe also that the $\sD_2$-eigenvalue is the average of the other two eigenvalues.
\end{remark}
\begin{remark}
Friedrich's lower bound \cite{Friedrich} says that the eigenvalues of the classical Dirac operator $\sD_1 = \sD$ must all have a square greater than, or equal to,
$
\tfrac{2}{7} \, R_g =   \tfrac{3}{ 14\, a^2}  \, ( 9  -  \mu^{2})
$. The calculated eigenvalue of $\sD_1$ is entirely independent of $\mu$ and satisfies that bound. In contrast, when $\mu \rightarrow 0^+$ with fixed $a$, the calculated eigenvalue of $\sD_3$ tends to zero. The eigenvalue of $\sD_2$ also tends to a finite constant below Friedrich's bound. This shows that $\sD_2$ and $\sD_3$ do not have the same spectrum as the classical Dirac operator, in general.
\end{remark}
\begin{remark}
Spectral deformations of this kind were calculated long ago for the Berger metrics on odd-dimensional spheres, for example, exhibiting how the Dirac spectrum and the space of harmonic spinors can depend on the metric \cite{Hitchin2, Bar92, Ginoux}. However, sometimes non-isometric metrics do have identical Dirac spectra \cite{Bar96}. Thus, the result of the calculation for our operators $\sD_r$ could hardly be guessed a priori. The behaviour of the classical Dirac spectrum under deformations with collapsing circles, like our $\mu \rightarrow 0^+$ limit, was studied in \cite{AB}, for example.
\end{remark}

Now consider the trace torsion vector fields $\tau_2$ and $\tau_3$ of the triality rotations of the Levi-Civita connection, as given in \eqref{TorsionTraceFormulas}. By \eqref{ExpansionsDr}, these vector fields also determine the degree-7 terms of the operators $\sD_2$ and $\sD_3$ on the spinor bundle. Then we have:
\begin{proposition}
\label{PropositionExplicitTaur}
The trace torsion vector fields $\tau_2$ and $\tau_3$ of the rotated connections ${\mathcal T}\nabla^g$ and ${\mathcal T}{\mathcal T}\nabla^g$ are identified with the constant maps  in $M \rightarrow \su(3)$ having values
\begin{equation}
\tau_2 \: = \: -\, \tau_3  \: = \:  \frac{3}{4\, a} \, (1- \mu)  \,  v_8 \ .
\end{equation}
They only vanish for the bi-invariant metric with $\mu=1$.
\end{proposition}
The proofs of propositions \ref{PropositionExplicitEigenvalues} and \ref{PropositionExplicitTaur} are the content of appendix \ref{AppendixExplicitExample}.

\newpage

\addtocontents{toc}{\cftpagenumbersoff{section}} 

\appendix
\appendixpage
\noappendicestocpagenum
\addappheadtotoc

\addtocontents{toc}{\cftpagenumberson{section}}

\section{Additional identities of the Okubo product}
\label{AppendixOkuboIdentities}

\begin{proposition}
Let $\{v_j\}$ be a $\langle , \rangle$-orthonormal basis of $\su(3)$. Then for all $u$ in $\su(3)$,
\begin{align}
\sum\nolimits_j \, v_j \ok v_j \; &= \; 0   \linebr
\sum\nolimits_j \, (v_j \ok u ) \ok v_j \; &= \; \sum\nolimits_j \, v_j \ok (u  \ok v_j ) \; = \; 8\,u   \label{OkuboIdentity2}   \linebr
\sum\nolimits_j \, (u \ok v_j ) \ok v_j \; &= \;  \sum\nolimits_j \,  v_j \ok ( v_j \ok u)  \; = \;  2\,u   \label{OkuboIdentity3}   \linebr
\sum\nolimits_j \, [u , v_j ] \ok v_j \; &= \;   \sum\nolimits_j \, v_j \ok [ v_j , u ]    \; = \; - \frac{3}{a} \,u  \ .  \label{OkuboIdentity4}  
\end{align}
\end{proposition}
\begin{proof}
Using the definition of the $\ok$-product and the fact that $ \sum_j v_j^2 $ is a Casimir, so proportional to the identity, we get
\[
\sum\nolimits_j \, v_j \ok v_j \: \propto \:  \sum\nolimits_j \, \{ v_j, v_j  \}_0  \: = \:  \sum\nolimits_j \, 2\, v_j^2 \, -\, \frac{2}{3} \, \mathrm{Tr} (v_j^2) \, I_3   \: = \:   0  \ .
\]
Using the adjunction relations \eqref{AdjunctionOkubo} and the orthogonality \eqref{OrthogonalityOkubo}, we have
\[
\sum\nolimits_j \, \langle (v_j \ok u ) \ok v_j   , w \rangle  \: = \:  \sum\nolimits_j \, \langle v_j \ok u  ,  v_j  \ok  w \rangle \: = \:  \sum\nolimits_j \, \langle v_j ,  v_j \rangle \, \langle u , w \rangle \: = \: 8 \, \langle u , w \rangle \ ,
\]
which implies one of the identities in \eqref{OkuboIdentity2}. The other can be obtained similarly. Using the Clifford-like relations \eqref{CliffordOkubo} and the vanishing of $\sum_j \, v_j \ok v_j$, 
\begin{equation}
 2 \, u  \: = \: \sum\nolimits_j \,  2\, \langle u , v_j  \rangle \, v_j   \: = \:  \sum\nolimits_j \, ( u \ok  v_j) \ok  v_j  +  (  v_j \ok  v_j) \ok u   \: = \:  \sum\nolimits_j \, ( u \ok  v_j) \ok  v_j .
 \end{equation}
This establishes one of the identities in \eqref{OkuboIdentity3}. The other can be obtained similarly. 

For the last identity, using the definition of the Okubo product, 
\begin{equation}
 \sum\nolimits_j \, v_j \ok [ v_j , u ]    \; = \; a\,  \sum\nolimits_j \: [ v_j , [ v_j , u ] ] + i \sqrt{3} \, \{ v_j , [ v_j , u ]  \}_0  \ .
\end{equation}
But $\sum\nolimits_j  \{ v_j , [ v_j , u ] \} = \sum\nolimits_j (v_j v_j u - u v_j v_j)$ vanishes, since $\sum\nolimits_j v_j v_j$ is a Casimir proportional to the identity. Moreover, 
 $f(u) = \sum_j [v_j , [v_j , u] ]$ is the Casimir of the adjoint representation, so it is also proportional to the identity. The proportionality constant is 
\begin{equation}
\label{CasimirAdjointConstant}
 \frac{1}{8} \, \mathrm{Tr}( f) \, = \,  \frac{1}{8} \, \sum\nolimits_j \mathrm{Tr} (  \mathrm{ad}_{v_j}   \mathrm{ad}_{v_j }) \, = \,   \frac{6}{8} \, \sum\nolimits_j  \mathrm{Tr}(v_j\, v_j) \, = \,  -\, \frac{3}{8 a^2} \,  \sum\nolimits_j \, \langle v_j , v_j \rangle \, = \,   -\, \frac{3}{a^2}  \ . \qedhere
\end{equation}
\end{proof}

\begin{proposition}
\label{DoubleOkuboIdentities}
The double $\ok$-products can be decomposed as
\begin{align}
(u \ok v) \ok w \: &= \:   2\, a^2\, \big( [[u,w], v]   \, +  \, i \sqrt{3} \, (uvw - wvu )_0   \big) \, +\,  \langle u , w \rangle\, v  \, -\,  \langle u , v \rangle\, w  \, +\,  \langle v , w \rangle\, u  \nonumber  \linebr
u \ok (v \ok w) \: &= \:   2\, a^2\, \big( [[w,u], v]   \, +  \, i \sqrt{3} \, (uvw - wvu )_0   \big) \, +\,  \langle u , w \rangle\, v  \, +\,  \langle u , v \rangle\, w  \, -\,  \langle v , w \rangle\, u    \nonumber  
\end{align}
as matrices in $\su(3)$. The associator of the Okubo product is therefore
\begin{equation}
( u \ok v) \ok w  -   u \ok (v \ok w)   \; = \; 4\, a^2  \, [[u,w], v] \, - \, 2   \langle u , v \rangle  \, w  \, + \,  2   \langle v , w \rangle  \, u  \ .
\end{equation}
\end{proposition}
\begin{proof}
The first step in the proof is to expand the double product $u \ok (v \ok w)$ using the definition \eqref{DefinitionOkuboProduct} of the Okubo product. After some cancellations, one obtains
\begin{multline}
u \ok (v \ok w) =  - 2\, a^2 \, \big[ \, 2\,uwv \,+ \,2\, vwu\, +\, uvw\, +\,wvu\,   - 2\, \mathrm{Tr}(wv) u  \\ - \, \mathrm{Tr}(uwv \,+ \, uvw)\, I_3 \,+ \,i\, \sqrt{3} \,(wvu\, -\, uvw)_0\, \big]  \ . \nonumber
\end{multline}
Now, the polarized Cayley-Hamilton identity guarantees that, for $3\times3$ matrices,  
\begin{multline}
uwv  \, +  \, wuv  \, + \,  vuw \, + \, vwu  \,+  \,uvw \, + \, wvu   \,  =  \, \mathrm{Tr}(uw) \, v \, +  \,  \mathrm{Tr}(uv)  \,w  \\ +  \, \mathrm{Tr}(wv) \, u \, + \, \mathrm{Tr}(uwv \, + \, wuv) \, I_3 \ .  \nonumber
\end{multline}
Substituting this identity into the previous expression and using that $ \langle u , v \rangle = - 2\, a^2 \, \mathrm{Tr}(uv)$, one obtains the expression for $u \ok (v \ok w)$ given in the proposition. 

The proof of the expression for $(u \ok v) \ok w$ is similar. The formula for the associator is just the difference.
\end{proof}

\begin{lemma}
\label{IdentitySchur}
Let $\{v_j\}$ be a $\langle , \rangle$-orthonormal basis of $\su(3)$. Then 
\begin{equation}
\label{IdentitySchurGeneralMatrices}
\sum\nolimits_j\, v_j\, M \, v_j \; = \; \frac{1}{6\, a^2}\, M \: -\: \frac{1}{2\, a^2}\, \mathrm{Tr}(M) \, I_3 
\end{equation}
for any complex $3\times 3$ matrix $M$.
\end{lemma}
\begin{proof}
Assume first that $M = u \in \su(3)$. The map $s(u) := \sum_j v_j\, u \, v_j$ is clearly independent of the choice of $\langle , \rangle$-orthonormal basis $\{v_j\}$. For any $h \in \SU(3)$, the set $\{v'_j = \mathrm{Ad}_h v_j\}$ is another admissible basis. Hence
\[
\mathrm{Ad}_h (s (u)) \: = \: \sum\nolimits_j v'_j\, (\mathrm{Ad}_h  u) \, v'_j \: = \: s(\mathrm{Ad}_h  u) \ .
\]
Since Ad is irreducible on $\su(3)$, Schur's lemma says that $\sum\nolimits_j\, v_j\, u \, v_j  = \lambda u$ for some $\lambda  \in \mathbb{R}$. Now, from the proof of  \eqref{OkuboIdentity4} we know that 
\begin{equation}
\label{IdentitySchurAux1}
 -\, \frac{3}{a^2} \, u \: =\: \sum\nolimits_j \,  [v_j , [v_j , u] ]  \: =\: \sum\nolimits_j \,  v_j v_j u + u v_j v_j - 2 v_j u v_j  \: = \: 2 (c - \lambda) u \ ,
\end{equation}
 where the constant $c$ comes from the Casimir $\sum\nolimits_j \,  v_j v_j  = c I_3$. Taking traces,
 \[
 3\,c \: =\: \sum\nolimits_j \, \mathrm{Tr}(v_j v_j) \: =\:  - \tfrac{1}{2 a^2} \sum\nolimits_j \langle v_j , v_j \rangle \: =\:   - \tfrac{4}{a^2}  \ .
 \]
So $c = -\frac{4}{3\, a^2}$ and, from \eqref{IdentitySchurAux1}, we get $\lambda = \tfrac{1}{6 a^2}$. This proves identity \eqref{IdentitySchurGeneralMatrices} for any $M \in \su(3)$, and hence, by $\CC$-linearity, for any traceless $M \in \mathfrak{sl}(3, \CC)$.

For a general $M$, decompose $M = M_0 + \tfrac{1}{3} \mathrm{Tr}(M) \, I_3$ into traceless and pure trace parts. We have just proved that $M_0$ satisfies \eqref{IdentitySchurGeneralMatrices}. So using again the value of the constant $c$,
\begin{align}
\sum\nolimits_j\, v_j\, M \, v_j \: &= \: \tfrac{1}{6\, a^2}\, M_0 \: +\: \tfrac{1}{3} \mathrm{Tr}(M) \sum\nolimits_j v_j v_j   \nonumber \linebr
& = \: \tfrac{1}{6\, a^2}\, [M - \tfrac{1}{3} \mathrm{Tr}(M) \, I_3 ]\: +\: \tfrac{1}{3} \mathrm{Tr}(M)  (-\tfrac{4}{3\, a^2}\, I_3 ) \nonumber \linebr
& = \: \tfrac{1}{6\, a^2}\, M \: - \:  \tfrac{1}{2 \, a^2} \mathrm{Tr}(M) \, I_3  \ . \nonumber \qedhere
\end{align}
\end{proof}

\begin{lemma}
\label{LemmaAdditionalOkuboIdentities1}
Let $\{v_j\}$ be a $\langle , \rangle$-orthonormal basis of $\su(3)$. Then, for all $u, w \in \su(3)$,
\begin{equation}
\label{OkuboIdentities5}
\sum\nolimits_j\, \big[ v_j \, , \, (u\ok  v_j) \ok w  \big] \; = \; \sum\nolimits_j\, \big[ v_j \, , \,   w  \ok (u\ok  v_j) \big]  \; = \;  4\, [u, w] \ .
\end{equation} 
\end{lemma}
\begin{proof}
Using the Ad-equivariance \eqref{AdInvariantOkubo} and that $[v_j, v_j ] = 0$,  then using \eqref{OkuboIdentity4} and that $\{ v_j\}$ is a basis, and then using the Ad-invariance of $\langle , \rangle$ and identity \eqref{ImportantIdentityOkubo2},
\begin{align}
\sum\nolimits_j\, \big[ v_j \, , \, (u\ok  v_j) \ok w  \big] \: &= \:  \sum\nolimits_j\, ( [ v_j , u ] \ok  v_j) \ok w \: +\:  (u  \ok  v_j) \ok [v_j, w]    \nonumber  \linebr
&= \: \tfrac{3}{a} \,  u  \ok w \: +\: \sum\nolimits_{j,k} \,   \langle \, v_k , [v_j, w] \rangle \: (u  \ok  v_j) \ok v_k \nonumber  \linebr
&= \: \tfrac{3}{a} \,  u  \ok w \: + \:  \sum\nolimits_{k} \, (u  \ok [w,  v_k]) \ok v_k \: = \: 4\, [u, w]  \nonumber \ .
\end{align}
The second identity in \eqref{OkuboIdentities5} is proved similarly.
\end{proof}

\subsection*{Proof of proposition \ref{PropositionOkuboIdentities}}

Take the first identity of proposition \ref{DoubleOkuboIdentities} and substitute $u \to [v_j, u]$, $w \to v_j$ and $v \to w$. Then we get
\begin{equation}
\label{WTerms}
\sum\nolimits_j \, \big([v_j, u] \ok w \big) \ok v_j \; = \; 2\, a^2 (W_1 + W_2) + W_3 \ , 
\end{equation}
where the three terms $W_r$ will be calculated below. Firstly, using \eqref{CasimirAdjointConstant},
\begin{align}
W_1 \: &= \:  \sum\nolimits_j \big[  [\,  [v_j, u]  ,v_j \, ], w \big]  \: = \:  - \,   [ f(u),  w ]  \: = \:    \frac{3}{a^2}  \, [u,  w ]  \ .
\end{align}
Secondly, using the Ad-invariance of $\langle , \rangle$ and that $\su(3)$ is unimodular, 
\begin{align}
W_3 \: &= \:  \sum\nolimits_j \,  \langle [v_j, u] , v_j \rangle\, w  \, -\,  \langle [v_j, u] , w \rangle\, v_j  \, +\,  \langle w , v_j \rangle\, [v_j, u]   \nonumber  \linebr
&= \:   - \mathrm{Tr}( \mathrm{ad}_{u } )\, w \, - \, \sum\nolimits_j  \, \langle v_j, [ u , w] \rangle\, v_j  \, +\,  [w, u]      \nonumber   \linebr
&= \:  0  \, - \, [u , w]  \, +\,  [w, u]    \: = \:   2\, [w, u] \ .
\end{align}
To evaluate $W_2$, note that
\begin{align}
W_2 \: &\propto \:   \sum\nolimits_j  \big( [v_j, u] \, w \, v_j \, - \,   v_j \, w \,  [v_j, u] \big)_0    \nonumber   \linebr
&= \:   \sum\nolimits_j  \big( v_j \, (u \, w \, + \, w \, u )\, v_j \, - \, u \, v_j \, w \, v_j \, - \,  v_j \, w \,  v_j \,  u  \big)_0  \nonumber     \linebr
&= \:   \tfrac{1}{6\, a^2}\, \big( (u \, w \, + \, w \, u )\, - \, u \, w \, - \,  w \,  u  \big)_0  \: = \: 0 \ ,
\end{align}
where the second equality used lemma \ref{IdentitySchur} three times, as well as the facts that $u$ and $w$ are traceless, and that the traceless part of any matrix proportional to the identity vanishes. Finally, substituting in expression \eqref{WTerms} the values of the three terms $W_j$, we obtain the first identity \eqref{ImportantIdentityOkubo1}. 

The second identity in \eqref{ImportantIdentityOkubo1} can be derived in the same way, using the second identity in propositon \ref{DoubleOkuboIdentities} instead of the first one.

The two identities \eqref{ImportantIdentityOkubo2} follow from the Ad-invariance \eqref{AdInvariantOkubo} of the Okubo product together with \eqref{ImportantIdentityOkubo1} and \eqref{OkuboIdentity4}. For example, 
\[
\sum\nolimits_j  (u \ok [v_j, w] ) \ok v_j  \: = \: \sum\nolimits_j \, [v_j, u \ok w]  \ok v_j  \, - \,   ([v_j, u] \ok w ) \ok v_j   \: = \: \frac{3}{a} \, u \ok w  \, - \,  4\, [u, w] \, .  \qedhere
\]

\section{Proofs for section \ref{Section3ConnectionsSpinorialLifts}}
\label{Appendix3ConnectionsSpinorialLifts}

Given a transformation $A \in \mathfrak{so}(\su(3))$, define the vector in $\su(3)$
\begin{equation}
\label{DefinitionUA}
u_A \: :=\: a^2 \, \sum\nolimits_j [v_j, A v_j ] \ ,
\end{equation}
and a linear map $N_A : \su(3) \rightarrow \su(3)$ by
\begin{equation}
\label{DefinitionNA}
N_A (w) \: := \: \sqrt{3}\, a^2\, i \, \sum\nolimits_j \, \big( \, v_j \, w \, (A v_j) \: - \: (A v_j)\, w \, v_j \, \big)_0 \ .
\end{equation}
The following decomposition shows that $N_A$ is the difference between $\rho^1_A$ and $\rho^2_A$.
\begin{lemma}
\label{LemmaDecompositionRho}
For every $A \in \mathfrak{so}(\su(3))$, the map $N_A$ belongs to $\mathfrak{so}(\su(3))$ and
\begin{align}
\label{DecompositionRho}
\rho^1_A \: &= \: \tfrac{1}{2}\, \mathrm{ad}_{u_A} \: - \: \tfrac{1}{2}\, A   \: + \: \tfrac{1}{2}\, N_A   \linebr
\rho^2_A \: &= \: \tfrac{1}{2}\, \mathrm{ad}_{u_A} \: - \: \tfrac{1}{2}\, A   \: - \: \tfrac{1}{2}\, N_A \ . \nonumber
\end{align}
In particular, $\rho^1_A$ and $\rho^2_A$ belong to $\mathfrak{so}(\su(3))$ and $\rho^1_A - \rho^2_A = N_A$.
\end{lemma}
\begin{proof}
To see that $N_A$ is in $\mathfrak{so}(\su(3))$, take an arbitrary vector $z \in \su(3)$. Using the definition \eqref{AdInvariantProduct} of $\langle , \rangle$ and the fact that $z$ is traceless, we have
\[
\langle z,  N_A (w)  \rangle \: = \: -\, 2 \sqrt{3}\, a^4\, i \, \sum\nolimits_j \, \mathrm{Tr} \big[ \, v_j \, w \, (A v_j) \, z \: - \: (A v_j) \, w \, v_j \, z \, \big] \ .
\]
The cyclicity of the trace then implies that $\langle z,  N_A (w)  \rangle$ is antisymmetric in $z \leftrightarrow w$.

To derive the decomposition of $\rho^1_A$, apply the first identity of proposition \ref{DoubleOkuboIdentities} with the choices $(u, v, w) = (A v_j, w, v_j)$:
\begin{multline}
\big( (A v_j) \ok w \big) \ok v_j \: = \: 2 a^2 \, [[A v_j, v_j], w] \: + \: 2 \sqrt{3}\, a^2\, i\, \big( (A v_j) \, w \, v_j - v_j \, w \, (A v_j) \big)_0 \\ + \: \langle A v_j , v_j \rangle \, w \: - \: \langle A v_j , w \rangle \, v_j \: + \: \langle w , v_j \rangle \, A v_j\ .
\end{multline}
The term with $\langle A v_j , v_j \rangle$ vanishes, since $A$ is in $\mathfrak{so}(\su(3))$. Thus, using the definitions of $u_A$ and $N_A$, and the fact that $\{v_j\}$ is a $\langle , \rangle$-ortonormal basis,
\[
\sum\nolimits_j \, \big( (A v_j) \ok w \big) \ok v_j \: = \: -\, 2\, [u_A , w] \: - \: 2\, N_A(w) \: + \: 2\, A w\ .
\]
Applying this formula to definition \eqref{rho1} of $\rho^1_A $gives the first identity in \eqref{DecompositionRho}. The second identity in \eqref{DecompositionRho} is proved similarly, now with the choices $(u, v, w) = (v_j, w, A v_j)$ in the the second identity of proposition \ref{DoubleOkuboIdentities}.
\end{proof}

\begin{lemma}
\label{LemmaURhoA}
For every $A \in \mathfrak{so}(\su(3))$, we have that 
\begin{equation}
u_{\rho^1_A} \: =\: u_{\rho^2_A} \:= \: u_A  \qquad \qquad  u_{N_A} \: = \: 0 \ .
\end{equation}
\end{lemma}
\begin{proof}
Using definition \eqref{DefinitionUA} of $u_A$, definition \eqref{rho1} of $\rho^1_A$, and applying the first identity in \eqref{OkuboIdentities5} with $u = A v_j$ and $w = v_j$,
\[
u_{\rho^1_A} \: = \: a^2 \sum\nolimits_l  [v_l , \rho^1_A (v_l)] \: = \: -\, \frac{a^2}{4} \, \sum\nolimits_{j, l}  \big[ v_l\, , ((A v_j) \ok v_l) \ok v_j \big] \: = \: -\, a^2 \sum\nolimits_j  [A v_j , v_j] \: = \: u_A \ .
\]
The calculation of $u_{\rho^2_A}$ is similar, using now and the second identity in \eqref{OkuboIdentities5}.
Finally, since $A \mapsto u_A$ is linear, lemma \ref{LemmaDecompositionRho} implies that $u_{N_A} = u_{\rho^1_A} - u_{\rho^2_A} = 0$. 
\end{proof}

\begin{lemma}
\label{LemmaNSquared}
For every $w \in \su(3)$ and $A \in \mathfrak{so}(\su(3))$, we have
\begin{align}
N_{\mathrm{ad}_{w}} \: &= \: 0    & N_{N_A} \: &= \: \mathrm{ad}_{u_A} \: - \: 3\, A   \label{LemmaNSquaredIdentities1}  \linebr
N_{\rho^1_A} \: &= \: \rho^2_A \, -\,  A   &  N_{\rho^2_A} \: &= \: A \, - \, \rho^1_A  \ .   \label{LemmaNSquaredIdentities2}
\end{align}
\end{lemma}
\begin{proof}
The first identity follows from lemma \ref{Restriction3Reps}, which says that $\rho^1_{\mathrm{ad}_w} = \rho^2_{\mathrm{ad}_w}$, and the fact that $N_A = \rho^1_A - \rho^2_A$, as stated in lemma \ref{LemmaDecompositionRho}.

The second identity in \eqref{LemmaNSquaredIdentities1} is the longest to prove. Using definition \eqref{DefinitionNA} of $N_A$, the cyclicity of the trace, and definition \eqref{DefinitionUA} of $u_A$,
\begin{align}
N_A (v_l) \: &= \:  \sqrt{3}\, a^2\, i  \: \sum\nolimits_j \, \big( v_j \, v_l \, A v_j - A v_j \, v_l \, v_j \big) \: - \: \frac{1}{3} \, \mathrm{Tr} ( v_j \, v_l \, A v_j - A v_j \, v_l \, v_j) \, I_3 \,    \nonumber \linebr
&= \: \sqrt{3}\, a^2\, i \, \Big[ \,  \frac{1}{3 a^2} \, \mathrm{Tr} (v_l \, u_A) \, I_3  \: + \:  \sum\nolimits_j \, \big( v_j \, v_l \, A v_j - A v_j \, v_l \, v_j \big)   \, \Big] \ .
\end{align}
Inserting this expression again in definition \eqref{DefinitionNA} of $N_A$, we get
\begin{equation}
\label{NASquared}
N_{N_A} (w) \: = \: \sqrt{3}\, a^2\, i \, \sum\nolimits_l \, ( \, v_l \, w \, N_A(v_l) \: - \: N_A(v_l)\, w \, v_l \, )_0 \: = \: -\, 3\, a^4 \, ( E_1 \, + \, E_2 )_0 
\end{equation}
with
\begin{align}
E_1 \: &:= \: \tfrac{1}{3 a^2} \, \sum\nolimits_l \, \mathrm{Tr} (v_l \, u_A) \, ( v_l \, w \: - \: w \, v_l )  \nonumber  \linebr
E_2 \: &:= \: \sum\nolimits_{j, l} \, ( \, v_l \, w \, v_j \, v_l \, A v_j \: - \: v_l \, w \, A v_j \, v_l \, v_j \: - \: v_j \, v_l \, A v_j \, w \, v_l \: + \: A v_j \, v_l \, v_j \, w  \, v_l \, ) \ . \nonumber 
\end{align}
The first term is just
\[
E_1 \: = \:  \tfrac{1}{3 a^2} \, \sum\nolimits_l \, \tfrac{-1}{2 a^2} \,  \langle v_l \, u_A \rangle \, ( v_l \, w \: - \: w\, v_l )  \: = \: -\, \tfrac{1}{6 a^4} \, [u_A , w] \ .
\]
For the second term, we can sum over $l$ using \eqref{IdentitySchurGeneralMatrices} repeatedly. The result is:
\begin{align}
E_2 \: &= \: \sum\nolimits_j \, \Big\{ \, \tfrac{1}{6 a^2} \, \big( w \, v_j \, A v_j - w \, A v_j \, v_j - v_j \, A v_j \, w + A v_j \, v_j \, w \big) \linebr
& \qquad \qquad - \: \tfrac{1}{2 a^2} \, \big( \mathrm{Tr} (w \, v_j) \, A v_j - \mathrm{Tr} (w\, Av_j) \, v_j - \mathrm{Tr} (A v_j \, w) \, v_j + \mathrm{Tr} (v_j \, w) \, Av_j \big) \Big\}   \nonumber \linebr
&= \: \sum\nolimits_j  \, - \, \tfrac{1}{6 a^2} \, \big[ [v_j , A v_j] , w \big] \: - \: \tfrac{1}{a^2} \, \big( \mathrm{Tr} (w \, v_j) \, A v_j - \mathrm{Tr} (w \, A v_j) \, v_j \big)   \nonumber \linebr 
&= \:  - \, \tfrac{1}{6 a^4} \, [ u_A , w ] \: + \: \tfrac{1}{2\, a^4} (A w \,  + \, Aw  ) \ , \nonumber
\end{align}
where the last equality uses the definition \eqref{DefinitionUA} of $u_A$, the relation between the trace and $\langle , \rangle$ given in \eqref{AdInvariantProduct}, and the fact that $A$ is skew symmetric with respect to $\langle , \rangle$.

Therefore $E_1 + E_2 = - \frac{1}{3 a^4} \, [u_A , w] + \frac{1}{a^4} \, A w$. This matrix belongs to $\su(3)$, so the traceless projection in \eqref{NASquared} acts trivially, and we obtain $N_{N_A} (w) = [u_A , w] - 3 A w$. This proves the second identity in \eqref{LemmaNSquaredIdentities1}.

To prove the two identities \eqref{LemmaNSquaredIdentities2}, we can use lemma \ref{LemmaDecompositionRho} and the previous identities of this lemma to calculate that
\begin{equation}
N_{\rho^1_A} \; = \;  \tfrac{1}{2} \, N_{\mathrm{ad}_{u_A}} \: - \:  \tfrac{1}{2} \, N_{A} \: + \:  \tfrac{1}{2}  \, N_{N_A}   
\; =\; \tfrac{1}{2} \,\mathrm{ad}_{u_A}  \: - \:  \tfrac{3}{2}  \, A \: - \:  \tfrac{1}{2}  \, N_A    \; =\;   \rho^2_A \, -\,  A   \ ,  \nonumber 
\end{equation}
and similarly for $N_{\rho^2_A}$.
\end{proof}

\begin{proof}[\bf{Proof of proposition \ref{PropositionTrialityCompositions}}] The six identities in the proposition follow directly from the fact that $\rho^3_A = A$, as per \eqref{rho3}, from the linearity of the correspondences $A \mapsto \rho^r_A$, and from the three previous lemmas. For example, the lemmas imply that:
\begin{align}
\rho^1_{\rho^1_A} \ &= \ \tfrac{1}{2}\, \mathrm{ad}_{u_A} \: - \: \tfrac{1}{2}\, \rho^1_A   \: + \: \tfrac{1}{2}\, N_{\rho^1_A}    \ = \ \tfrac{1}{2}\, \mathrm{ad}_{u_A} \: - \: \tfrac{1}{2}\, \rho^1_A   \: + \: \tfrac{1}{2}\, (\rho^2_A - A)  \nonumber \linebr
&= \ \tfrac{1}{2}\, \mathrm{ad}_{u_A}  \: - \: \tfrac{1}{2}\, A \: - \: \tfrac{1}{2}\, N_A  \ = \ \rho^2_A   \nonumber   \displaybreak[0]  \\[1.3em]
\rho^1_{\rho^2_A} \ &= \ \tfrac{1}{2}\, \mathrm{ad}_{u_A} \: - \: \tfrac{1}{2}\, \rho^2_A   \: + \: \tfrac{1}{2}\, N_{\rho^2_A}    \ = \ \tfrac{1}{2}\, \mathrm{ad}_{u_A} \: - \: \tfrac{1}{2}\, \rho^2_A   \: + \: \tfrac{1}{2}\, (A - \rho^1_A )   \nonumber \linebr
&= \ \ \tfrac{1}{2}\, \mathrm{ad}_{u_A}  \: + \: \tfrac{1}{2}\, A \: - \: \tfrac{1}{2}\, (\mathrm{ad}_{u_A}  - A)  \ = \ A \ = \  \rho^3_A  \ ,   \nonumber
\end{align}
and similar calculations for the $\rho^2_{\rho^r_A}$.
\end{proof}

\section{Proofs for section \ref{SymmetriesModuliSpaces}}
\label{AppendixSymmetriesModuliSpaces}

\subsection*{Proof of theorem \ref{PropertiesTrialityRotationConnections}}

As observed after the definition of ${\mathcal T}$, under the identification of $TM$ with $M \times \su(3)$, formulas \eqref{CondensedConnections3} and \eqref{rho3} say that $\mathcal{T}$ is implemented by the transformation $A \mapsto \rho^1_A$ of connection representations. Identities \eqref{CyclicityTrialityRotations} are then a consequence of \eqref{CondensedConnections3} and the algebraic identities of elements of $\mathfrak{so}(\su(3))$:
\begin{equation}
\label{TripleCompositionIdentity}
\rho^1_{\rho^1_A} \: = \: \rho^2_A \qquad \qquad \qquad \rho^1_{\rho^1_{\rho^1_A}} \: = \: \rho^3_A \: = \: A \ ,
\end{equation}
which themselves come from proposition \ref{PropositionTrialityCompositions}.

We now prove statement $(a)$. Since the orbits of $\mathcal{T}$ have either one or three distinct elements, $\mathcal{T} \nabla = \nabla$ is equivalent to $\mathcal{T} \mathcal{T} \nabla = \mathcal{T} \nabla$, i.e. to the identity $\nabla^3 = \nabla^2$.  Then statement $(a)$ follows from \eqref{CondensedConnections3} and the algebraic fact that $\rho^1_A = \rho^2_A$ if and only if $A \in \mathrm{ad} (\su(3))$. The {\it if} part follows from lemma \ref{Restriction3Reps}. In the other sense, when $\rho^1_A = \rho^2_A$, lemma \ref{LemmaDecompositionRho} implies that $N_A =0$. Then \eqref{LemmaNSquaredIdentities1} implies that $A =  \mathrm{ad}_{u_A / 3} \in \mathrm{ad} (\su(3))$.

We now prove statement $(b)$. Start by observing that, in our global trivializations, a gauge transformation $\varphi$ is represented by a map $\hat{\varphi} : M \rightarrow \mathrm{SO}(\su(3))$. Since $M$ is simply connected, this has a lift $\tilde{\varphi} : M \rightarrow \mathrm{Spin}(\su(3))$ that covers $\hat{\varphi}$ through the double cover $\lambda : \mathrm{Spin}(8) \rightarrow \mathrm{SO}(8)$, unique up to sign. The lift $\tilde{\varphi}$ acts on $S = S^+ \oplus S^-$ through the spin representation, i.e. by Clifford multiplication by $\tilde{\varphi}(h) \in \mathrm{Spin}(8) \subset \mathrm{Cl}(8)$. As is well-known, this action preserves chirality and is compatible with Clifford multiplication \cite{LM, Bourguignon}, so the spinorial gauge transformation satisfies 
\begin{equation}
\label{CompatibilityGaugeTransClifford}
\tilde{\varphi} (X \cdot \psi) \; = \; (\varphi \, X) \cdot ( \tilde{\varphi} \, \psi ) \ .
\end{equation}
The gauge transformation $\varphi$ acts on connections by $\varphi (\nabla) = \varphi \circ \nabla \circ \varphi^{-1}$. Let $\nabla^S$ denote the spinorial lift of $\nabla$ to $S$. We will show that 
\begin{equation}
\label{SpinorialLiftGaugeTransfConnection}
(\varphi (\nabla) )^S \: = \: \tilde{\varphi} \circ \nabla^S \circ \tilde{\varphi}^{-1} \: = \: \tilde{\varphi} (\nabla^S)  \ .
\end{equation}
To recognize this, we use \eqref{CompatibilityGaugeTransClifford} and observe that
\begin{align}
\tilde{\varphi}\, \nabla^S_X \big[ \tilde{\varphi}^{-1}(Y \cdot \psi) \big] \: &= \: \tilde{\varphi} \Big[ \, (\nabla_X \varphi^{-1} Y) \cdot \tilde{\varphi}^{-1} \psi \, + \, (\varphi^{-1} Y) \cdot \nabla^S_X \tilde{\varphi}^{-1} \psi \, \Big]      \nonumber \linebr
&= \: \big( \varphi \nabla_X {\varphi}^{-1} Y \big) \cdot \psi \, + \, Y \cdot \big( \tilde{\varphi}\, \nabla^S_X\, \tilde{\varphi}^{-1} \psi \big) \ . \nonumber
\end{align}
Since the spinorial lift $(\varphi (\nabla) )^S$ is also compatible with Clifford multiplication, we obtain that the difference of connections $ \tilde{\varphi}(\nabla^S)_X -  (\varphi (\nabla) )^S_X$ is a $C^\infty$-linear endomorphism of $S$ that commutes with Clifford multiplication of spinors by vector fields. But the real spinor bundle $S$ is an irreducible Clifford module of $\mathrm{Cl}(8) \simeq M_{16}(\mathbb{R})$ \cite{LM}, so that scalar operator must be a real multiple of the identity. At the same time, both $\tilde{\varphi} (\nabla^S)_X$ and $(\varphi (\nabla) )^S_X$ are compatible with the inner product of spinors, so their difference must be skew-adjoint with respect to that product. This is possible only if the difference is the zero multiple of identity, and since $X$ is arbitrary, we conclude that $(\varphi (\nabla) )^S  = \tilde{\varphi} (\nabla^S)$.

To finish the proof of statement $(b)$, observe that ${\mathcal T}[ \varphi (\nabla)] := [\varphi (\nabla)]^2$, and according to \eqref{DefinitionThreeConnections3} the rotated connections $\nabla^2$ are the transports of the spinorial lifts $\nabla^+$ under the triality isometry ${\mathcal I}^+$. Thus, using the previous result that $(\varphi (\nabla) )^+  = \tilde{\varphi} (\nabla^+)$,
\begin{align}
(\varphi (\nabla))^2 &= ({\mathcal I}^+)^{-1} \circ (\varphi (\nabla))^+ \circ {\mathcal I}^+  =  ({\mathcal I}^+)^{-1} \circ \tilde{\varphi} \circ  \nabla^+  \circ \, \tilde{\varphi}^{-1} \circ {\mathcal I}^+ \nonumber \linebr
& =   \big[ ({\mathcal I}^+)^{-1}  \circ \tilde{\varphi} \circ  {\mathcal I}^+ \big] \circ \nabla^2 \circ \big[ ({\mathcal I}^+)^{-1} \circ \tilde{\varphi} \circ  {\mathcal I}^+ \big]^{-1} \ . 
 \end{align}  
Since ${\mathcal I}^+: TM \rightarrow S^+$ is an isometry and $\tilde{\varphi}$ preserves the spinor inner product, the composition $\varphi' := ({\mathcal I}^+)^{-1}  \circ \tilde{\varphi} \circ  {\mathcal I}^+$ is an automorphism of $TM$ that preserves $g$. It preserves the orientation because both ${\mathcal I}^+$ and $\varphi$ do, and the lift $\tilde{\varphi}$ operates through the half-spin representation which takes values in $\mathrm{SO}(\Delta^+)$. In other words, $\varphi'$ is a gauge transformation of $(TM, g)$.  
So the triality rotation ${\mathcal T}[ \varphi (\nabla)]$ is a gauge transformation of ${\mathcal T}(\nabla)$, and hence ${\mathcal T}$ descends to a map of the moduli spaces.  \qed

\begin{proof}[\bf Proof of proposition \ref{PropInvarianceNormCurvartureChernSimons}]
Define the tensor
\[
Q^\nabla (X,Y,Z,W) \; := \; \sum\nolimits_{i} \, g \big( R^{\nabla}(X, Y)V_i \, ,\, R^{\nabla}(Z, W) V_i \big) \ .
\]
Using formulas \eqref{ExplicitCurvature}, the fact that any $\rho^r_B \in \mathfrak{so}(\su(3))$ is antisymmetric with respect to $\langle, \rangle$, and then the auxiliary lemma \ref{LemmaTracesRho},
\begin{align}
Q^{\nabla^r} (X,Y,Z,W) \: &= \: \sum\nolimits_{i, j}  g \big( R^{\nabla^r}(X, Y) V_i \, , V_j \big) \;  g \big( R^{\nabla^r}(Z, W) V_i \, , V_j \big)     \nonumber \linebr
&= \: \sum\nolimits_{i, j}  \langle \, \rho^{r-1}_{F_{X,Y}} v_i  \, , \,  v_j \, \rangle   \;  \langle \, \rho^{r-1}_{F_{Z,W}} v_i  \, , \,  v_j \, \rangle    \nonumber \linebr
&= \: \sum\nolimits_{i}  \langle \, \rho^{r-1}_{F_{X, Y}} v_i \, , \,  \rho^{r-1}_{F_{Z,W}} v_i \, \rangle  \: = \: - \sum\nolimits_{i}   \langle \, \rho^{r-1}_{F_{Z,W}} \, \rho^{r-1}_{F_{X,Y}} v_i \, , \,  v_i \, \rangle      \nonumber \linebr
&= \: -  \mathrm{Tr}\big(  \rho^{r-1}_{F_{Z,W}} \, \rho^{r-1}_{F_{X,Y}} \big) \: = \: -  \mathrm{Tr}\big( F_{Z,W} \, F_{X,Y} \big)   \ .   \nonumber
 \end{align}
The last term depends only on $(g, \nabla)$, not on the index $r$, so we get that $Q^{\nabla^1} = Q^{\nabla^2} = Q^{\nabla^3}$. But the norm of the curvature of $\nabla^r$ is just $|\, R^{\nabla^r} \, |_g^2 = \sum_{ i,j} Q^{\nabla^r} (V_i,V_j, V_i,V_j) $, so it too is independent of the index $r$. The energy-momentum tensor $\Theta^{\nabla^r}$ can be constructed from $Q^{\nabla^r}$ and  $|\, R^{\nabla^r} \, |_g^2$, so it too is independent of $r$.

 Finally, let $\{B_a\}$ be a basis of the 23-dimensional space $\mathfrak{so}(\su(3))$. The 1-form $A^{\nabla^{r+1}}_X$ of \eqref{EndomorphismConnection} has values in $\mathfrak{so}(\su(3))$ and, using \eqref{CondensedConnections3}, can be written as $\sum_a A^a \, \rho^r_{B_a}$, where the $A^a$ are 1-forms with real values and the indices $r$ are taken modulo 3.
Using the definition of the Chern-Simons form and lemma \ref{LemmaTracesRho} we then have
 \begin{align}
  \mathrm{CS}(\nabla^{r+1}) \; &:= \;  \sum\nolimits_{a, b}  \, A^a \wedge \dd A^b \;  \mathrm{Tr}(\rho^r_{B_a} \,  \rho^r_{B_b} ) \:  + \:   \sum\nolimits_{a, b,c}  \,  \tfrac{2}{3}\, A^a\wedge A^b\wedge A^c \;  \mathrm{Tr}\big(\rho^r_{B_a} \, \rho^r_{B_b} \,  \rho^r_{B_c}   \big)   \nonumber \linebr
&= \: \sum\nolimits_{a, b}  \, A^a \wedge \dd A^b \;  \mathrm{Tr}(\rho^r_{B_a} \,  \rho^r_{B_b} ) \:  + \:   \sum\nolimits_{a, b,c}  \,  \tfrac{1}{3}\, A^a\wedge A^b\wedge A^c \;  \mathrm{Tr}\big(\rho^r_{B_a} \, [ \rho^r_{B_b} , \,  \rho^r_{B_c}] \,   \big)   \nonumber \linebr
&= \: \sum\nolimits_{a, b}  \, A^a \wedge \dd A^b \;  \mathrm{Tr}(B_a \,  B_b ) \:  + \:   \sum\nolimits_{a, b,c}  \,  \tfrac{1}{3}\, A^a\wedge A^b\wedge A^c \;  \mathrm{Tr}\big(B_a \, [ B_b , \,  B_c] \,   \big)   \nonumber  \ .
  \end{align}
The last term is again independent of the index $r$, so $\mathrm{CS}(\nabla^1) = \mathrm{CS}(\nabla^2) = \mathrm{CS}(\nabla^3)$.
\end{proof}

\begin{proof}[\bf Proof of lemma \ref{LemmaTracesRho}]
For $r=3$ the relations are tautological, since $\rho^3_A = A$. For $r =1$, we know that the correspondence $A \mapsto \rho^1_A$ is one of the spin representations of $\mathfrak{spin}(8) \simeq \mathfrak{so}(8)$ on the 8-dimensional space $\su(3) \simeq \Delta^+$, as discussed after \eqref{EndomorphismConnection}. In particular, for every $h \in \mathrm{Spin}(8)$, we have that $\rho^1_{\mathrm{Ad}_h A} =  \sigma_h \circ \rho^1_A  \circ \sigma_h^{-1}$, where $\sigma$ denotes the $\mathrm{Spin}(8)$ group representation on $\Delta^+$. Taking the traces of the endomorphisms of $\Delta^+$, we get that
\[
\mathrm{tr}\big( \rho^1_{\mathrm{Ad}_h A} \ \rho^1_{\mathrm{Ad}_h B} \big) \ = \ \mathrm{Tr}( \rho^1_A \, \rho^1_B ) \ .
\]
Thus $\beta(A, B) := \mathrm{tr}(\rho^1_A \rho^1_B)$ is a $\mathrm{Ad}$-invariant bilinear form on the simple Lie algebra $\mathfrak{spin}(8)$, and hence must be proportional to the Killing form, i.e. $\beta(A, B) = b\, \mathrm{tr}(AB)$ for some real constant $b$. But then formula \eqref{TripleCompositionIdentity} implies that
\[
\mathrm{tr}(A B) \: = \: \mathrm{tr}( \rho^1_{\rho^1_{\rho^1_A}} \, \rho^1_{\rho^1_{\rho^1_B}} ) \: = \: b \,  \mathrm{tr}( \rho^1_{\rho^1_A} \, \rho^1_{\rho^1_B} )  \: = \: b^2 \,  \mathrm{tr}( \rho^1_A \, \rho^1_B )  \: = \: b^3 \,  \mathrm{tr}(A B) \ ,
\]
and hence $b=1$. The proof for $\rho^2$ is similar.
The second identity in \eqref{TraceRhoRho} follows from the first identity and the fact that $[\rho^r_B , \rho^r_C] = \rho^r_{[B, C]}$, which is true because the correspondence $A \mapsto \rho^r_A$ is a spin representation.
 \end{proof}

  \subsection*{Proof of theorem \ref{TheorTrialityHiggsModuliSpace}}
  
  Let $\{V_j = \alpha (v_j^\LL)\}$ be our usual $g$-orthonormal trivialization of $TM$ associated to a $\langle , \rangle$-orthonormal basis $\{ v_j\}$ of $\su(3)$. It can by used to identify $TM$ with $M \times \su(3)$. Then a Higgs field $\phi$ of $g$-skew endomorphism of $TM$ is represented by the map $\hat{\phi}: M \rightarrow \mathfrak{so}(\su(3))$ determined by 
\begin{equation}
\label{RepVectorialHiggs}
\hat{\phi}(v_j) = \sum\nolimits_k \,  g(\phi\, V_j, V_k)\, v_k \ .
\end{equation}
At the same time, from the standard isomorphism between $ \mathfrak{so}(T_xM)$ and $ \mathfrak{spin}(T_xM)$, our $\phi$ induces a $\langle , \rangle$-skew endomorphism of $S^+$ given by 
\begin{equation}
\label{SpinorialHiggs}
\tilde{\phi}(\psi^+) \: = \: \tfrac{1}{4} \sum\nolimits_{j} \, V_j \cdot (\phi V_j) \cdot \psi^+ \: = \: \tfrac{1}{4} \sum\nolimits_{j,k} \, g(\phi\, V_j , V_k)\; V_j \cdot V_k \cdot \psi^+ \ .
\end{equation}
Then triality map ${\mathcal T}\phi$ has an alternative description, equivalent to definition \ref{DefinitionTrialityRotationHiggs}.
\begin{proposition}
\label{DefinitionTrialityRotationHiggs2}
Let $\phi$ be a Higgs field in $\HH_g$ and let $\tilde{\phi}$ denote the induced $\langle , \rangle$-skew endomorphism of $S^+$. Then
\[
 {\mathcal T} \phi \; = \; ({\mathcal I}^+)^{-1} \circ \tilde{\phi} \circ {\mathcal I}^+ \ ,
\]
where ${\mathcal I}^+: TM \rightarrow S^+$ is the isometry \eqref{DefinitionIsometry1}.
\end{proposition}
To prove this proposition we should be careful about the bundle endormorphisms and their representative maps $M \rightarrow \mathfrak{so}(\su(3))$. The $\langle , \rangle$-orthonormal basis $\{ v_j\}$ of $\su(3)$ determines both the $g$-orthonormal trivialization $\{V_j\}$ of $TM$ and a trivialization $\{\psi_j \}$ of $S^+$ which is orthonormal with respect to the spinor inner product $\langle , \rangle$. In fact, in the model $S^+$ of section \eqref{DefinitionSpinorBundles}, we have $\psi_j\, |_x = (x, v_j)$ for all $x\in M$. If $\xi$ is a $\langle , \rangle$-skew endomorphism of $S^+$, its representative map $M \rightarrow \mathfrak{so}(\su(3))$ is then
\begin{equation}
\label{RepSpinorialHiggs}
\hat{\xi}(v_j) = \sum\nolimits_k \,  \langle \xi\, \psi_j, \psi_k \rangle \, v_k  \ .
\end{equation}
\begin{lemma}
The induced spinor endomorphism $\tilde{\phi}$ has the same representative map $M \rightarrow \mathfrak{so}(\su(3))$ as the triality rotated field ${\mathcal T}\phi$.
\end{lemma}
\begin{proof}
In the model \eqref{TraditionalCliffordMultiplication} of Clifford multiplication, the spinor $V_j \cdot V_k \cdot \psi^+$ corresponds to the map $-(v_k \ok \psi^+) \ok v_j$ with values in $\su(3)$. Hence, using \eqref{RepVectorialHiggs} and definition \eqref{rho1}, the transformed spinor $\tilde{\phi}(\psi^+)$ of \eqref{SpinorialHiggs} is represented by
\[
\tilde{\phi}(\psi^+) \: = \: - \tfrac{1}{4} \sum\nolimits_{j} \, \big( (\hat{\phi}\, v_j) \ok \psi^+ \big) \ok v_j \: = \: \rho^1_{\hat{\phi}} (\psi^+) \ .
\]
This can be compared to the definition \ref{DefinitionTrialityRotationHiggs} of ${\mathcal T} \phi$.
\end{proof}
Now the problem is that, even though ${\mathcal T} \phi$ and the spinorial $\tilde{\phi}$ and have the same representative maps, they are sections of different bundles. This can be solved by transporting $\tilde{\phi}$ to a vectorial endomorphism through the  isometry ${\mathcal I}^+: TM \rightarrow S^+$ of \eqref{DefinitionIsometry1}. This is the content of the next lemma, which finishes the proof of proposition \ref{DefinitionTrialityRotationHiggs2}.
\begin{lemma}
Let $\xi$ be a section of the bundle of $\langle , \rangle$-skew endomorphisms of $S^+$. Then $({\mathcal I}^+)^{-1} \circ \xi \circ {\mathcal I}^+$ is in $\HH_g$ and is represented by the same map $M \rightarrow \mathfrak{so}(\su(3))$ as $\xi$. 
\end{lemma}
\begin{proof}
Using \eqref{RepVectorialHiggs} and the fact that ${\mathcal I}^+(V_j) = \psi_j$, the endomorphism $({\mathcal I}^+)^{-1} \circ \xi \circ {\mathcal I}^+$ of $TM$ is represented by the map
\begin{align}
\big[ \widehat{({\mathcal I}^+)^{-1} \circ \xi \circ {\mathcal I}^+ } \big] (v_j) \: &=  \: \sum\nolimits_k \,  g\big(({\mathcal I}^+)^{-1} \, \xi \, \psi_j \,,\, V_k \big)\, v_k  \nonumber \linebr
&= \: \sum\nolimits_{k,l} \,  \langle \psi_l , \xi \, \psi_j \rangle\ g\big(({\mathcal I}^+)^{-1} \psi_l \, , \, V_k \big)\; v_k     \nonumber \linebr
&= \:   \sum\nolimits_{k,l} \,  \langle \psi_l , \xi\, \psi_j \rangle\;  g(V_l , V_k) \; v_k   \nonumber  \linebr
&=  \: \sum\nolimits_{k} \,  \langle \psi_k ,\xi \, \psi_j \rangle\;  v_k   \: =  \: \hat{\xi}(v_j)   \nonumber \ ,
\end{align}
where the last equality is \eqref{RepSpinorialHiggs}.
\end{proof}
We will now use proposition \ref{DefinitionTrialityRotationHiggs2} to finish the proof of theorem \ref{TheorTrialityHiggsModuliSpace}. Recall from the proof of theorem \ref{PropertiesTrialityRotationConnections}-(b) that the gauge transformation $\varphi$ on $TM$ can be lifted to a gauge transformation $\tilde{\varphi}$ on $S$ that preserves chirality and satisfies the compatibility \eqref{CompatibilityGaugeTransClifford} with Clifford multiplication. Thus, using definition \eqref{SpinorialHiggs} of $\tilde{\phi}$ and the fact that the gauge transform of $\phi$ is the endomorphism $\varphi(\phi) =  \varphi \circ \phi \circ \varphi^{-1}$ of $TM$,
\begin{align}
\tilde{\varphi} [ \, \tilde{\phi}(\psi^+) \, ] \, &= \, \tfrac{1}{4} \sum\nolimits_{j} \, (\varphi\, V_j) \cdot [\varphi (\phi V_j)] \cdot  \tilde{\varphi} (\psi^+) \, = \, \tfrac{1}{4} \sum\nolimits_{j} \, (\varphi\, V_j) \cdot [ (\varphi \, \phi) (\varphi \, V_j)] \cdot  \tilde{\varphi} (\psi^+)   \nonumber  \linebr
&= \, \tfrac{1}{4} \sum\nolimits_{j} \, (V'_j) \cdot [ (\varphi \, \phi) ( V'_j)] \cdot  \tilde{\varphi} (\psi^+)  \, =\, \widetilde{(\varphi \, \phi)} (\tilde{\varphi} \, \psi^+) \ .
\end{align}
Therefore, using this relation and using proposition \ref{DefinitionTrialityRotationHiggs2} two times,
\begin{align}
{\mathcal T} (\varphi \, \phi) \, &= \, ({\mathcal I}^+)^{-1} \circ \widetilde{(\varphi \, \phi)} \circ {\mathcal I}^+  \, = \, ({\mathcal I}^+)^{-1} \circ \tilde{\varphi} \circ \tilde{\phi} \circ \tilde{\varphi}^{-1} \circ {\mathcal I}^+  \nonumber  \linebr
 &= \,  \big[ ({\mathcal I}^+)^{-1} \circ \tilde{\varphi} \circ {\mathcal I}^+ \big] \circ ({\mathcal I}^+)^{-1} \circ \tilde{\phi}\circ {\mathcal I}^+  \circ \big[ ({\mathcal I}^+)^{-1} \circ \tilde{\varphi} \circ {\mathcal I}^+ \big]^{-1}     \nonumber  \linebr
 &= \,  \big[ ({\mathcal I}^+)^{-1} \circ \tilde{\varphi} \circ {\mathcal I}^+ \big] \circ ({\mathcal T} \phi) \circ \big[ ({\mathcal I}^+)^{-1} \circ \tilde{\varphi} \circ {\mathcal I}^+ \big]^{-1}   \nonumber  \linebr
 &= \, \varphi' \circ ({\mathcal T} \phi) \circ (\varphi')^{-1}  \,  =  \, \varphi' ({\mathcal T} \phi )   \ , \nonumber
\end{align}
where $\varphi' := ({\mathcal I}^+)^{-1} \circ \tilde{\varphi} \circ {\mathcal I}^+ $ is the same gauge transformation on $TM$ that appeared in the proof of statement (b) of theorem \ref{PropertiesTrialityRotationConnections}. This finishes the proof of theorem \ref{TheorTrialityHiggsModuliSpace}.

\section{Proofs for sections \ref{ThreeDiracOperators} and \ref{KLDerivatives}}
\label{AppendixKLDerivatives}

\subsection*{Proof of proposition \ref{AlternativeExplicitD1Adjoint}}

This is an adaptation of the proof that the classical Dirac operator is formally self-adjoint \cite{Bourguignon}. Using the definition of $\cD_r$ and standard properties of Clifford multiplication, write
\begin{align}
\label{cD1Aux1}
\langle \cD_r \psi^+ , \psi^- \rangle &=  \sum\nolimits_j \,  \langle V_j \cdot \nabla^r_{V_j} \psi^+ , \psi^- \rangle  =  -  \sum\nolimits_j \langle  \nabla^r_{V_j} \psi^+ , \, V_j \cdot \psi^- \rangle \nonumber \linebr
&= -  \sum\nolimits_j  \,  \partial_{V_j}  \langle  \psi^+ , \, V_j \cdot \psi^- \rangle  -   \langle  \psi^+ , \, \nabla^r_{V_j} (V_j \cdot \psi^-) \rangle \ .
\end{align}
The last step is just the $\langle , \rangle$-compatibility of the connection $\nabla^r$ on $S^+$. Now define a vector field $Y$ on $M$ by
\[
Y :=   \sum\nolimits_j  \langle  \psi^+ , \, V_j \cdot \psi^- \rangle  \,  V_j \ .
\]
Using the Levi-Civita connection $\nabla^g$ and a standard way to write the divergence, 
\begin{align}
\label{cD1Aux4}
\sum\nolimits_ j \partial_{V_j}  \langle  \psi^+ , \, V_j \cdot \psi^- \rangle \:& =\:  \sum\nolimits_ j \:   \partial_{V_j}  \, g ( Y , V_j)  \: =\:   \sum\nolimits_ j \:    g ( \nabla^g_{V_j} Y , V_j) \, +\,  g ( Y , \nabla^g_{V_j}  V_j)  \nonumber \linebr
&= \: \divergence_g Y \: + \:  \sum\nolimits_j \: \langle  \psi^+ , \, (\nabla^g_{V_j} V_j) \cdot \psi^- \rangle \ .
\end{align}
Moreover, theorem \ref{SpinorialLiftsVSTrialityTransports} says that the connections $\nabla^r$ on $S^\pm$ coincide with the spinorial lifts of $\nabla^r$ on $TM$. Thus, standard properties of Clifford multiplication allow us to expand
\begin{align}
 \label{cD1Aux2}
\nabla^r_{V_j} (V_j \cdot \psi^-) &=   (\nabla^r_{V_j} V_j) \cdot \psi^- +  V_j \cdot \nabla^r_{V_j} \psi^-  \nonumber \linebr
&=   \big[\nabla^g_{V_j} V_j + K^{\nabla^r}(V_j, V_j)  \big] \cdot \psi^- +  V_j \cdot \nabla^r_{V_j} \psi^-  \ ,
\end{align}
where $K^{\nabla^r}$ denotes the contorsion tensor of $\nabla^r$ on $TM$. 
The trace of $K^{\nabla^r}$ is related to the trace $\tau^{\nabla^r}$ of the usual torsion tensor $T^{\nabla^r}$ by
\begin{align}
\label{cD1Aux3}
\sum\nolimits_j g( K^{\nabla^r}(V_j, V_j) , X) =  \sum\nolimits_j g( T^{\nabla^r}(X, V_j) , V_j) =:  \tau^{\nabla^r}(X) \ ,
\end{align}
where $\tau$ is a 1-form on $M$. Thus, bringing \eqref{cD1Aux1}, \eqref{cD1Aux2}, \eqref{cD1Aux3} and \eqref{cD1Aux4} together, the two terms with $\nabla^g_{V_j} V_j$ cancel and we get 
\begin{equation}
\langle \cD_r \psi^+ , \psi^- \rangle =  -\, \divergence_g Y \; + \; \sum\nolimits_j \,  \big\langle  \psi^+ , \,  V_j \cdot \nabla^r_{V_j} \psi^-  \: + \:   \tau^{\nabla^r}(V_j)  \: V_j \cdot  \psi^- \big\rangle \ .
\end{equation}
Therefore, using Stokes' theorem, the $L^2$-adjoint of $\cD_r$ can be written as 
\begin{equation}
\cD_r^\dag  \psi^- =  \sum\nolimits_j \:  V_j \cdot \nabla^r_{V_j} \psi^- +   \tau^{\nabla^r}(V_j) \: V_j  \cdot \psi^- \ . \nonumber  
\end{equation}

\subsection*{Proofs for section \ref{KLDerivatives}}

\begin{proof}[\bf Proof of theorem \ref{TrialitiyLieDerivativeConnection}]
Regard vector fields as maps $M \rightarrow \su(3)$ through our usual correspondence $Y = \sum_j g(Y, V_j) \, v_j$. Then the triality rotated connections are given by 
\begin{equation}
\label{RepresentativeNablar}
\nabla^r_Y  Z \ = \  \partial_Y Z \, + \: \rho^{r-1}_{A_Y} (Z)           \ ,
\end{equation}
as in \eqref{CondensedConnections3}, with index $r$ taken modulo 3. Denote by $u^\LL$ and $u^\RR$ the left- and right-invariant vector fields on $M$ extending $u\in \su(3)$. Using the expression for $\Lie_{u^\LL}$ given in lemma \ref{LieDerivativesInvariantKillingFields} and the fact that the correspondence $A \mapsto \rho^{r-1}_A$ defined by \eqref{rho1} does not depend on the point in $M$, we get
\begin{align}
\Lie_{u^\LL}  ( \nabla^r_Y Z ) \: &= \: \partial_{u^\LL} \partial_{Y} Z \: + \: [u ,  \partial_{Y} Z ]\: +\: \partial_{u^\LL} ( \rho^{r-1}_{A_Y} (Z) ) \:+ \: [u ,  \rho^{r-1}_{A_Y} (Z) ]   \nonumber \linebr
 &=  \: \partial_{u^\LL} \partial_{Y} Z \:+ \: \partial_{Y}  [u ,   Z ] \:+ \: \rho^{r-1}_{(\partial_{u^\LL} A_Y)} (Z)  \:+ \: \rho^{r-1}_{A_Y} (\partial_{u^\LL}  Z)   \:+ \: \mathrm{ad}_u \circ \rho^{r-1}_{A_Y} (Z)   \nonumber \ .
\end{align}
At the same time, 
\[
\nabla^r_Y (\Lie_{u^\LL} Z ) \: = \: \partial_{Y} \partial_{u^\LL} Z \: + \:  \partial_{Y} [u , Z ] \:+ \: \rho^{r-1}_{A_Y} (\partial_{u^\LL}  Z) \: +\: \rho^{r-1}_{A_Y} ([u ,  Z])  \ .
\]
Thus, using the general property $ \partial_{X} \partial_{Y}  -  \partial_{Y} \partial_{X} =  \partial_{[X, Y]}$, the difference is
\begin{align}
\Lie_{u^\LL}  ( \nabla^r_Y Z ) \: - \:\nabla^r_Y (\Lie_{u^\LL} Z ) \: &= \: \partial_{[u^\LL, Y ]} Z  \:+\:  \rho^{r-1}_{(\partial_{u^\LL} A_Y)} (Z) \:+\:  \big[ \mathrm{ad}_u \, ,  \, \rho^{r-1}_{A_Y} \big] (Z)  \ .   \nonumber 
\end{align}
But lemma \ref{Restriction3Reps} says that $\rho^s_{\mathrm{ad}_u} = \mathrm{ad}_u$. Moreover, the correspondence $A \mapsto \rho^{r-1}_A$ is a vectorial or spin representation of $\mathfrak{spin}(8)$, as discussed after \eqref{EndomorphismConnection}, and in particular is a Lie algebra homomorphism. Hence the last term is
\[
\big[ \mathrm{ad}_u \, ,  \, \rho^{r-1}_{A_Y} \big]  = \big[ \rho^{r-1}_{\mathrm{ad}_u} \, ,  \, \rho^{r-1}_{A_Y} \big]  = \rho^{r-1}_{ [ \mathrm{ad}_u \, ,  \, A_Y ] } \ .
\]
Using the formula for $\nabla^r_{[u^\LL, Y ]} Z$ coming from \eqref{RepresentativeNablar}, we obtain
\begin{align}
\Lie_{u^\LL}  ( \nabla^r_Y Z )  - \nabla^r_Y (\Lie_{u^\LL} Z ) -  \nabla^r_{[u^\LL, Y ]} Z  &=  \rho^{r-1}_{(\partial_{u^\LL} A_Y)} (Z) +  \rho^{r-1}_{ [ \mathrm{ad}_u \, ,  \, A_Y ] }  (Z)   -   \rho^{r-1}_{  A_{[u^\LL , Y ]} }  (Z)      \nonumber \linebr
&=   \rho^{r-1}_{ \{\,  (\partial_{u^\LL} A_Y)  +  [ \mathrm{ad}_u \, ,  \, A_Y ]    -     A_{[u^\LL , Y ]} \, \} } (Z)  \nonumber  \ .
\end{align}
Since $\nabla^1 = \nabla$ and $\rho^3_A = A$, this identity with $r=1$ implies that $(\Lie_{u^\LL} \nabla)_Y =   (\partial_{u^\LL} A_Y)  +  [ \mathrm{ad}_u \, ,  \, A_Y ]    -     A_{[u^\LL , Y ]} $. Thus, now taking $r=2$, recalling definition \eqref{DefinitionTrialityRotationConnections} of the triality rotation of $\nabla$ and definition \eqref{DefinitionTrialityRotationHiggs} of triality rotation in $\HH_g$, we finally get that
\[
(\Lie_{u^\LL} \,  {\mathcal T} \nabla)_Y Z \: = \: (\Lie_{u^\LL} \,  \nabla^2)_Y Z  \: = \:  \rho^{1}_{(\Lie_{u^\LL} \nabla)_Y} (Z) \: = \:  {\mathcal T} [(\Lie_{u^\LL} \nabla)_Y ] \, (Z) \ .
\]
The proof of the formula for $(\Lie_{w^\RR} \,  {\mathcal T} \nabla)_Y$ is similar but simpler.
\end{proof}

\vspace{.3cm}

\begin{proof}[\bf Proof of lemma \ref{LieDerivativesInvariantKillingFields}]
Let $X$ be $u^\LL$ or $w^\RR$. The derivative  $\Lie_X Y$ is identified with the map $\sum_j g(\Lie_X  Y, V_j)  v_j$. Since $X$ is Killing both for $g$ and for the bi-invariant metric $\langle , \rangle$, we know that the automorphism $\alpha$ of \eqref{DefinitionAlpha} satisfies
\begin{equation}
\Lie_X \alpha \: =\: \Lie_X (\alpha^{-1})\: =\: 0 \ .
\end{equation}
This follows from the proof of proposition B.1 in \cite{Bap4}. Thus, using that $\Lie_X g = 0$, using relation \eqref{DefinitionAlpha} between $g$ and  $\langle , \rangle$, and using the definition $V_j = \alpha (v_j^\LL)$,
\begin{align}
 g(\Lie_X  Y, V_j)\: &= \: \partial_X \, g(Y, V_j)   - g(Y, \Lie_X V_j)\: =\: \partial_X  Y^j  \:-\:  \sum\nolimits_k g(Y, V_k)  \: g(V_k, \Lie_X V_j)   \nonumber \linebr
 &=\: \partial_X  Y^j  \: +  \: \sum\nolimits_k Y^k \: g( \Lie_X V_k,  V_j)\: =\: \partial_X  Y^j  \:+ \: \sum\nolimits_k Y^k \,  \big\langle \alpha^{-1} (\Lie_X V_k) ,  v^\LL_j  \big\rangle  \nonumber \linebr
 &=\: \partial_X  Y^j  \: + \:  \sum\nolimits_k Y^k \,  \langle \Lie_X v^\LL_k ,  v^\LL_j \rangle  =\: \partial_X  Y^j  \: + \:  \sum\nolimits_k Y^k \,  \langle [X, v^\LL_k] ,  v^\LL_j \rangle  \ .   \nonumber
\end{align}
But the Lie bracket of invariant vector fields always satisfies $[u^\LL ,  v_k^\LL] = [u ,  v_k]^\LL$ and  $[w^\RR ,  v_k^\LL] = 0$. Hence
\begin{align}
\sum\nolimits_j g(\Lie_{u^\LL}  Y, V_j)  \, v_j \: &= \: \sum\nolimits_j (\partial_{u^\LL}  Y^j) \, v_j  \: + \:  \sum\nolimits_{j, k} Y^k \,  \langle [u, v_k]^\LL ,  v^\LL_j \rangle \, v_j    \nonumber \linebr
 &=\:  \sum\nolimits_j (\partial_{u^\LL}  Y^j )\, v_j  \: + \:  \sum\nolimits_{j, k} Y^k \,  \langle [u, v_k] ,  v_j \rangle \, v_j      \nonumber \linebr
 &=\:  \sum\nolimits_j (\partial_{u^\LL}  Y^j )\, v_j  \: + \:  \sum\nolimits_{k} Y^k \,  [u, v_k]    \ ,   \nonumber  
\end{align}
and similarly $\sum\nolimits_j g(\Lie_{w^\RR}  Y, V_j)  \, v_j = \sum\nolimits_j (\partial_{w^\RR}  Y^j) \, v_j$.
\end{proof}

\vspace{.3cm}

\begin{proof}[\bf Proof of proposition \ref{ExplicitKLDerivativesNonKilling}]
Recall expression \eqref{DefinitionKLDerivative} for the Kosmann-Lichnerowicz derivative. Since $V_i \cdot V_j = - V_j \cdot V_i$ for $i \neq j$, we have
\begin{equation}
\label{KLDerivativesSimplification1}
\frac{1}{8}\,  \sum_{i, j} \big\{  g([X, V_i], V_j) - g([X, V_j], V_i) \big\} \,   V_i \cdot V_j  \cdot \psi \ = \  \frac{1}{4}\,  \sum_{i \neq j} \, g([X, V_i], V_j) \:   V_i \cdot V_j  \cdot \psi \ .
\end{equation}
Using that $V_j = \alpha (v_j^\LL)$ and definition \eqref{DefinitionAlpha} of $\alpha$,
\[
g(\,[u^\LL, V_i], V_j \, ) \: =\: \langle  \alpha^{-1}[u^\LL, \alpha \, v^\LL_i],\, v^\LL_j \rangle \: =\:  \langle\,  [u^\LL ,  v^\LL_i ] \, + \,  \alpha^{-1} (\Lie_{u^\LL} \alpha ) (v^\LL_i)  \, ,\, v^\LL_j \, \rangle \ .
\]
Thus, using that $[u^\LL ,  v^\LL_i ]  = [u ,  v_i ]^\LL$ and formula \eqref{TraditionalCliffordMultiplication} for Clifford multiplication, 
\[
  \sum_{i \neq j} \, g([X, V_i], V_j) \:   V_i \cdot V_j  \cdot \psi^+ \:=\:  -\,  \sum_{i \neq j} \,  \langle\,  [u ,  v_i ]^\LL \, + \,  \alpha^{-1} (\Lie_{u^\LL} \alpha ) (v^\LL_i)  \, ,\, v^\LL_j \, \rangle \:   ( v_j  \ok \psi^+ ) \ok v_i  \ .
\]
But \eqref{ImportantIdentityOkubo1} now says that
\begin{align}
\sum_{i \neq j} \,  \langle\,  [u ,  v_i ]^\LL \, ,\, v^\LL_j \, \rangle \:   ( v_j  \ok \psi^+ ) \ok v_i  \: &= \: \sum_{i , j} \,  \langle\,  [u ,  v_i ] \, ,\, v_j \, \rangle \:   ( v_j  \ok \psi^+ ) \ok v_i  \nonumber \linebr
&= \: \sum_{i} \,   ( [u ,  v_i ]  \ok \psi^+ ) \ok v_i  \: = \: -\, 4\, [u, \psi^+] \ .
\end{align}
Hence
\[
 \sum_{i \neq j} \, g([X, V_i], V_j) \:   V_i \cdot V_j  \cdot \psi^+ \: = \:  4\, [u, \psi^+] \, -\,   \sum_{i \neq j} \,  \langle\,  \alpha^{-1} (\Lie_{u^\LL} \alpha ) (v^\LL_i)  \, ,\, v^\LL_j \, \rangle \:   ( v_j  \ok \psi^+ ) \ok v_i  \ .
\]
Combining this expression with \eqref{KLDerivativesSimplification1} and the definition \eqref{DefinitionKLDerivative} of the Kosmann-Lichnerowicz derivative, we obtain the expression for $\cL_{u^\LL} \, \psi^+$ given in the proposition. The expression for $\cL_{u^\RR} \, \psi^+$ can be obtained similarly, after using that $[u^\RR ,  v^\LL_i ]  = 0$.

The expressions for $\cL_{u^\LL} \, \psi^-$ and $\cL_{u^\RR} \, \psi^-$ can also be obtained similarly,  using the appropriate entries in \eqref{TraditionalCliffordMultiplication} for $\psi^-$.
\end{proof}

\section{Explicit calculations of Dirac eigenvalues}
\label{AppendixExplicitExample}

\refstepcounter{subsection}
\subsection*{\thesubsection\quad  Computing $\sD_r$-eigenvalues}


Let $\{v_j \}$ be our usual $\langle , \rangle$-orthonormal basis of $\su (3)$ and write $f_{ijk} := \langle [v_i, v_j] , v_k \rangle$. These coefficients are totally antisymmetric. Define a new metric $g$ on $M$ by declaring that the rescaled vector fields $V_j = \mu_j \, v_j^\LL$, with $\mu_j \in \mathbb{R}^+$, are $g$-orthonormal. Then we have:
\begin{lemma}
\label{LemmaChristoffel}
The Levi-Civita connection $\nabla^g$ of the left-invariant metric $g$ satisfies
\begin{equation}
\label{ChristoffelSymbols}
g( \nabla^g_{V_i} V_j , V_k) \; = \; \Gamma_{ijk} \; = \; \frac{f_{ijk}}{2}\, \Big( \frac{\mu_i \mu_j}{\mu_k} \, - \, \frac{\mu_j \mu_k}{\mu_i} \, + \, \frac{\mu_k \mu_i}{\mu_j} \Big) \ .
\end{equation}
In particular, when $\mu_1 = \cdots = \mu_7 = 1$ and $\mu_8 = \mu^{-1}$, 
\begin{equation}
\label{ChristoffelSymbols2}
\Gamma_{ijk} \; = \;
\begin{cases}
\ \frac{1}{2}\, f_{ijk} & \text{if } \  i,j,k \neq 8 \ , \\
\ \frac{\mu}{2}\, f_{ijk} & \text{if } \ j = 8 \text{ or } k = 8 \ ,\\
\ \big( \frac{1}{\mu}  - \frac{\mu}{2}\big)\, f_{ijk} & \text{if } \ i = 8 \ .
\end{cases}
\end{equation}
\end{lemma}
\begin{proof}
The $V_j$ are left-invariant, so the derivatives in the Koszul formula drop out and
\[
2\, g( \nabla^g_{V_i} V_j , V_k) \; = \; g([V_i, V_j], V_k) \, - \, g([V_j, V_k], V_i) \, + \, g([V_k, V_i], V_j) \ .
\]
By \eqref{DefinitionAlpha} we have $g(u^\LL , V_k) = \langle \alpha^{-1} u , v_k\rangle = \langle u , \alpha^{-1} v_k\rangle = \mu_k^{-1} \langle u , v_k \rangle$, while  $[V_i, V_j] = \mu_i \mu_j [v_i,v_j]^\LL$. Hence $g([V_i,V_j],V_k) = \mu_i \mu_j \mu_k^{-1} f_{ijk}$, and \eqref{ChristoffelSymbols} follows using $f_{ijk} = f_{jki} = f_{kij}$. 
Substituting the values of $\mu_j$ into \eqref{ChristoffelSymbols} gives \eqref{ChristoffelSymbols2}, since at most one index can be $8$. 
\end{proof}

Now recall the linear maps $A_Y : M \rightarrow  \mathrm{End}( \su(3))$ defined in \eqref{EndomorphismConnection}. When $Y = V_i$ is a left-invariant vector field, 
$A_{V_i}$ is a constant endomorphism of $\su(3)$ determined by $A_{V_i} (v_j) =  \sum\nolimits_k \, \Gamma_{ijk} \, v_k $. It is convenient to separate its bi-invariant part, as in
\begin{equation}
\label{SeparationBiinvariantPartAVi}
A_{V_i}  \: = \:    \frac{1}{2} \,  \mathrm{ad}_{v_i} \, +  \, C_{V_i}   \ .
\end{equation}
Going through the three cases of \eqref{ChristoffelSymbols2}, the endomorphisms $C_{V_i}$ can be written as
\begin{equation}
\label{DefinitionCVi}
C_{V_i} (w) \; = \;
\begin{cases}
\ c_0 \, \big( \,\langle [v_i, w] , v_8 \rangle \,  v_8   +  \langle w, v_8  \rangle [v_i, v_8]   \, \big)  & \ \text{if } i \neq 8 \ , \\
\ c_8 \, [v_8, w] & \ \text{if } i = 8 \ ,
\end{cases}
\end{equation}
with the constants $c_0 := \tfrac{1}{2}(\mu - 1)$ and $c_8 := \tfrac{1}{2}( \tfrac{2}{\mu} - \mu - 1)$.

Define $\cD^{g_0}: S^+ \rightarrow S^-$ to be the classical chiral Dirac operator induced by the Levi-Civita connection of the bi-invariant metric $\langle , \rangle$ on $TM$. Let $\psi^+$ be a left-invariant spinor on $S^+$, i.e. a spinor represented by a constant map $M \rightarrow \su(3)$ in our model of $S^+$.
Then the definition of  $\cD_r$ and the separation \eqref{SeparationBiinvariantPartAVi} of $A_{V_i}$ into the bi-invariant part plus $C_{V_i}$ imply that
\begin{equation}
\label{DecompositioncD_r}
\cD_r \psi^+ \:  = \:  \cD^{g_0}  \psi ^+ \: + \:  \sum\nolimits_i \,  v_i \ok [ \rho^r_{C_{V_i}} (\psi^+) ] \ .
\end{equation}
The plan is to calculate the two terms on the right-hand side separately when the spinor $\psi^+ = \psi_0^+$ is the constant map $M \rightarrow \su(3)$ with value $v_8$.

\begin{lemma} 
\label{ExplicitEigenvaluesLemma1} The classical Dirac operator of the bi-invariant metric satisfies
\begin{equation}
\cD^{g_0}  \psi_0^+ \: = \:  -\, \frac{3}{2\, a} \, \psi_0^- \ .
\end{equation}
\end{lemma}
\begin{proof}
Since the metric is bi-invariant, the three operators $\cD^{g_0}_r$ coincide. Therefore, we will use $\cD^{g_0}_3$ to perform the calculation, because it involves only one $\ok$-product. For the bi-invariant metric  $\langle , \rangle$ we have $V_j = v_j^\LL$ and the Levi-Civita connection $\nabla^{g_0}$ satisfies
\[
\Gamma_{ijk}^{g_0} \: = \:  g_0 (\, \nabla^{g_0}_{v_i^\LL} \,  v_j^\LL,  \, v_k^\LL \, ) \: = \:  \frac{1}{2} \, \langle   [v_i, v_j] , v_k \rangle \ .
\]
Thus, applying formula \eqref{ExplicitDr} for $\cD_3$ and using that $\{ v_j\}$ is a $\langle , \rangle$-orthogonal basis of $\su(3)$, 
\begin{align}
\cD^{g_0}_3 \,  v_8 \: &= \:   \frac{1}{2} \,   \sum_{i, j, k} \, \langle   [v_i, v_j] , v_k \rangle \; \langle v_8 , v_j \rangle \  v_i \ok v_k    \nonumber  \linebr 
&= \:   \frac{1}{2} \,   \sum_{i} \,  v_i \ok  [v_i, v_8]   \: = \: -\, \frac{3}{2\, a} \: v_8  \nonumber  \ ,
\end{align}
where the last equality used identity \eqref{OkuboIdentity4}. The formula of the lemma follows because the chiral spinors $\psi_0^+$ and $\psi_0^-$ were both defined to correspond to the map $M\rightarrow \su(3)$ with constant value $v_8$.
\end{proof}

\begin{lemma} 
\label{ExplicitEigenvaluesLemma2}
The endomorphisms $\rho^r_{C_{V_i}}$ satisfy
\begin{align}
\rho^1_{C_{V_i}} (w) \: &= \:  \frac{c_0}{2} \,  ( v_8  \ok w )  \ok  [v_i, v_8]      \linebr
\rho^2_{C_{V_i}} (w) \: &= \:  \frac{c_0}{2} \, [v_i, v_8] \ok (w \ok v_8)     \nonumber  \linebr
\rho^3_{C_{V_i}} (w) \: &= \:  c_0\, \big(  \, \langle w, v_8 \rangle \, [v_i, v_8] \, -  \,  \langle w, [v_i, v_8] \rangle \, v_8 \, \big)   \nonumber
\end{align}
for all $w \in \su(3)$ and $i \neq 8$, while for $i= 8$ we have $\rho^r_{C_{V_8}} (w) = c_8 \, [v_8 , w]$.
\end{lemma}
\begin{proof}
The expressions for $i \neq 8$ result from straightforward calculations using definitions \eqref{rho1} of the $\rho^r_A$; formula \eqref{DefinitionCVi} for the endomorphisms $C_{V_i}$; the fact that $\{ v_j\}$ is a $\langle , \rangle$-orthonormal basis of $\su(3)$; the Ad-invariance of  $\langle , \rangle$; and the Clifford-like relations \eqref{CliffordOkubo} for the Okubo product. The expressions given for $i= 8$ also use the simplifications provided by identities \eqref{ImportantIdentityOkubo1}. 
\end{proof}

\begin{lemma} 
\label{ExplicitEigenvaluesLemma3}
The endomorphisms $\rho^r_{C_{V_i}}$ satisfy
\begin{align}
\sum\nolimits_i \,  v_i \ok [ \rho^1_{C_{V_i}} (v_8) ] \: &= \:  0      \linebr
\sum\nolimits_i \,  v_i \ok [ \rho^2_{C_{V_i}} (v_8) ] \: &= \:  - \, \frac{3\, c_0}{2\, a} \: v_8       \nonumber   \linebr
\sum\nolimits_i \,  v_i \ok [ \rho^3_{C_{V_i}} (v_8) ] \: &= \:  - \, \frac{3\, c_0}{a} \: v_8   \ .   \nonumber
\end{align}
\end{lemma}
\begin{proof}
The starting point is lemma \ref{ExplicitEigenvaluesLemma2} and the algebraic fact that $v_8 \ok v_8 = v_8$, which can be directly checked using the explicit expression for the Okubo product, definition \eqref{NormalizedBasissu(3)} of $v_8$, and the explicit expression for the Gell-Mann matrix $\lambda_8$. The calculation for $\rho^1$ also uses the simplifications provided by identities \eqref{ImportantIdentityOkubo1}. The calculation for $\rho^3$ also uses identity \eqref{OkuboIdentity4}. The calculation for $\rho^2$ uses the associator formula \eqref{AssociatorOkubo} and then identity \eqref{OkuboIdentity4}. We do not write down the detailed calculations here.
\end{proof}
Combining the decomposition \eqref{DecompositioncD_r} of $\cD_r$ with lemmas \ref{ExplicitEigenvaluesLemma1} and \ref{ExplicitEigenvaluesLemma3}, we then get
\begin{align}
\cD_1 \, \psi^+_0  \: &= \:  -\, \frac{3}{2\, a} \, \psi_0^-      \qquad \qquad \qquad \cD_2 \, \psi^+_0  \: = \:  - \, \frac{3}{2\, a} \, (c_0 +1) \: \psi_0^-       \nonumber   \linebr
\cD_3 \,\psi^+_0  \: &= \:  - \, \frac{3}{2\, a} \, (2\, c_0 +1) \: \psi_0^-   \ ,   \nonumber
\end{align}
where $c_0 = \tfrac{1}{2}(\mu - 1)$, recall.

To compute $\cD_r^\dag \psi^-_0$ we can now use a shortcut. Note that the adjoint operators $\cD_r^\dag$ commute with the Kosmann-Lichnerowicz derivatives along all Killing vectors fields on $(M, g)$, by proposition \ref{PropDiracCommutesWithKosmann}. Thus, restricted to the spinors identified with constant maps $M \rightarrow \su(3)$, the operator $\cD_r^\dag$ must preserve the $\mathrm{Ad}_{\mathrm{U}(2)}$-invariant subspace of $\su(3)$, which is the line spanned by $v_8$. In other words, $\cD_r^\dag \psi^-_0$ must be proportional to $\psi^+_0$. The constant of proportionality is determined by $\langle \cD_r^\dag \psi^-_0 ,  \psi^+_0 \rangle = \langle \psi^-_0, \cD_r \psi^+_0\rangle$, so it coincides with the constant of proportionality computed above for $\cD_r \psi^+_0$. This constant of proportionality is therefore the eigenvalue of $\sD_r$ acting on the spinor $\psi_0$ with $\psi_0^+ = \psi_0^- = v_8$.

\refstepcounter{subsection}
\subsection*{\thesubsection\quad  Computing the trace torsions $\tau_2$ and $\tau_3$}


In this section we compute the trace torsion vector fields $\tau_2$ and $\tau_3$, as given in \eqref{TorsionTraceFormulas}, for the Levi-Civita connection of $g$. We first prove that $\tau_3 = -\tau_2$ and then compute $\tau_2$.
\begin{proposition}
For the Levi-Civita connection of the metric $g$, we have $\tau_3 = - \tau_2$.
\end{proposition}
\begin{proof}
The Levi-Civita connection $\nabla = \nabla^1$ has no torsion, so $\tau_1$ vanishes. Moreover, by remark \ref{RemarkTorsionLeftInvariantMetrics}, the divergences $\divergence_g V_j$ also vanish. Thus, from \eqref{PropTorsionTraceAux1} and the proof of proposition \ref{PropTorsionTrace}, we get that 
\begin{equation}
\label{FormulaTaur}
\tau_r \: = \:   \sum\nolimits_{j} \, \rho^{r-1}_{A_{V_j}} (v_j) \ .
\end{equation}
Combining these facts with the identity of lemma \ref{Sum3Reps} and formulas \eqref{TorsionTraceFormulas}, we have
\begin{align}
\tau_2 + \tau_3  \: &= \: \sum\nolimits_{j}  \, (\rho^1_{A_{V_j}} + \rho^2_{A_{V_j}} + \rho^3_{A_{V_j}}) (v_j)     \nonumber \linebr
&= \: \sum\nolimits_{j} \,  \mathrm{ad}_{u_j} (v_j)  \ , \nonumber 
\end{align}
where we have denoted $ u_j = u_{A_{V_j}}$. Now, using the definition of $u_A$ in lemma \ref{Sum3Reps} and the separation \eqref{SeparationBiinvariantPartAVi} of $A_{V_j}$, we get
\begin{align}
u_{A_{V_j}}  \: &= \: a^2 \sum\nolimits_k [v_k , A_{V_j} v_k] \: = \:  a^2 \sum\nolimits_k \tfrac{1}{2}  \, [v_k , [v_j,  v_k ]] \: +\:  [v_k , C_{V_j} v_k]   \nonumber \linebr
&= \:  \frac{3}{2}\,  v_j  \:+\:  a^2 \sum\nolimits_k  [v_k , C_{V_j} v_k] \ ,  \nonumber 
\end{align}
where the last equality comes from the Casimir \eqref{CasimirAdjointConstant} of the adjoint representation. Thus,
\begin{equation}
\label{OppositeWrAux1}
\tau_2 \, + \, \tau_3 \: = \: a^2\, \sum\nolimits_{j,k}  \big[\, [v_k , C_{V_j} v_k], v_j  \, \big] \ .
\end{equation}
The final step is to use the explicit expressions for the endomorphisms $C_{V_j}$ coming from \eqref{DefinitionCVi}. For $j \neq 8$, that expression and the $\langle , \rangle$-orthogonality of the  basis $\{ v_j\}$ imply that
\[
\sum\nolimits_{k}   [v_k , C_{V_j} v_k] \: = \:  -\, 2\, c_0 \, [\, v_8 , [v_8, v_j] \, ] \ .
\]
So from the Jacobi identity and the Casimir \eqref{CasimirAdjointConstant} of the adjoint representation, 
\begin{align}
\sum\nolimits_{j\neq 8}  \sum\nolimits_k \big[\, v_j ,  [v_k , C_{V_j} v_k]  \, \big]  \: &= \: -\, 2\, c_0 \sum\nolimits_{j}  \big[\, v_j ,  [\, v_8 , [v_8, v_j] \, ]  \, \big]   \nonumber \linebr
& = \: 2\, c_0 \sum\nolimits_{j}  \big[\, [v_8, v_j] ,  [v_j , v_8  ]  \, \big]  +  \big[\, v_8 ,  [\, [v_8, v_j] , v_j \, ]  \, \big] \: =\: 0 \ . \nonumber
\end{align}
For $j = 8$, the explicit expression of $C_{V_8}$ coming from \eqref{DefinitionCVi} says that
\[
\sum\nolimits_k \big[\, v_8 ,  [v_k , C_{V_8} v_k]  \, \big]  \: = \: c_8 \, \sum\nolimits_k \big[\, v_8 ,  [v_k , [v_8,  v_k] ]  \, \big]   \: = \: 0 \ ,
\]
where the last equality uses again the Casimir of the adjoint representation. Thus, the conclusion from \eqref{OppositeWrAux1} is that $\tau_3 = - \tau_2$.
\end{proof}

\begin{proposition}
For the Levi-Civita connection of the metric $g$, we have
 \[
 \tau_2 \:= \:   \frac{3}{4\, a} \, (1- \mu)  \,  v_8 \ .
 \]
\end{proposition}
\begin{proof}
From formula \eqref{FormulaTaur}, the definition \eqref{rho1} of $\rho^1$, and the splitting \eqref{SeparationBiinvariantPartAVi} of $A_{V_i}$,
\begin{align}
4 \, \tau_2 \: &= \: 4\, \sum\nolimits_{j} \, \rho^1_{A_{V_j}} (v_j)    \: = \: - \sum\nolimits_{j,k} \, (( A_{V_j} v_k) \ok v_j) \ok v_k   \nonumber \linebr
&= \: -  \tfrac{1}{2} \sum\nolimits_{j,k} \, ([v_j, v_k] \ok v_j) \ok v_k  \: -\:  \sum\nolimits_{j,k} \, (( C_{V_j} v_k) \ok v_j) \ok v_k   \nonumber \linebr
&= \:  2 \sum\nolimits_{j} \,  [ v_j , v_j  ]    \: -\:  \sum\nolimits_{j,k} \, (( C_{V_j} v_k) \ok v_j) \ok v_k    \: = \:  - \sum\nolimits_{j,k} \, (( C_{V_j} v_k) \ok v_j) \ok v_k  \ . \label{CalculateW3Aux1} 
\end{align}
The penultimate equality used identity \eqref{ImportantIdentityOkubo1}. Now, for $j = 8$, the explicit expression of $C_{V_8}$ coming from \eqref{DefinitionCVi} implies that
\[
\sum\nolimits_{k} \, (( C_{V_8} v_k) \ok v_8) \ok v_k  \: = \:  c_8 \sum\nolimits_{k} \, ( [v_8 , v_k] \ok v_8) \ok v_k \: = \: 0 \ ,
\]
where the last equality used identity \eqref{ImportantIdentityOkubo1}. Thus, to calculate \eqref{CalculateW3Aux1}, we only need to consider the expression \eqref{DefinitionCVi} of $C_{V_j}$ for the case $j\neq 8$. But that particular expression also vanishes for $j=8$, so we can extend the resulting sum to the whole range of $j$: 
\begin{align}
\sum\nolimits_{j,k} \, [( C_{V_j} v_k) \ok v_j] \ok v_k  &=  c_0 \sum\nolimits_{j,k} \,  \big[ \,  \big( \,\langle [v_j, v_k] , v_8 \rangle \,  v_8   +  \langle v_k, v_8  \rangle [v_j, v_8]   \, \big)  \ok v_j   \,  \big]  \ok v_k  \  \nonumber \linebr 
&=  - c_0 \, \sum\nolimits_{j} \:  ( v_8 \ok v_j ) \ok  [v_j, v_8] \: +\:  c_0 \, \sum\nolimits_{j} \: ( [v_j, v_8]  \ok v_j   \, )  \ok v_8       \nonumber \linebr 
&=   - c_0 \, \sum\nolimits_{k} \:  ( v_8 \ok [v_8,  v_k ]) \ok v_k  \: +\:  \frac{3\, c_0}{a}\:  v_8 \ok v_8   \ . \nonumber
\end{align}
Here the second equality uses the $\langle , \rangle$-orthogonality of the basis $\{ v_j\}$, and the third equality uses the Ad-invariance of $\langle, \rangle$ and identity \eqref{OkuboIdentity4}. But $ v_8 \ok v_8 = v_8$, as described in the proof of lemma \ref{ExplicitEigenvaluesLemma3}. So applying identity \eqref{ImportantIdentityOkubo2}  and recalling \eqref{CalculateW3Aux1}, we get 
\[
-\, 4 \, \tau_2 \: = \: \sum\nolimits_{j,k} \, [( C_{V_j} v_k) \ok v_j] \ok v_k  \: =\:    \frac{3\, c_0}{a}\:  v_8 \ok v_8  \: +\:  \frac{3\, c_0}{a}\:  v_8 \ok v_8 \: =\:  \frac{6\, c_0}{a}\:  v_8  \ .
\]
The desired formula for $\tau_2$ then follows from the definition of $c_0$ in \eqref{DefinitionCVi}.
\end{proof}

\newpage


\renewcommand{\baselinestretch}{1.2}\normalsize

\phantomsection       

\addcontentsline{toc}{section}{References}

\renewcommand\doitext{doi:}

\vspace{1cm}


\begin{thebibliography}{99}    

\bibitem{Sapio} R. Sapio: On Spin(8) and triality: a topological approach, \textit{Expositiones Math.} \textbf{19} (2001), 143--161.

\bibitem{Baez} J. Baez: The octonions, \textit{Bull. Amer. Math. Soc.} \textbf{39} (2002), 145--205.

\bibitem{Harvey} F. Harvey: \textit{Spinors and calibrations}, Academic Press, 1990.

\bibitem{Hitchin} N. Hitchin: Stable forms and special metrics, in:  \textit{Contemp. Math.}, \textbf{288} (2001), 70--89.

\bibitem{Witt} F. Witt: Special metric structures and closed forms, PhD thesis, Univ. Oxford, 2005.

\bibitem{Witt2} F. Witt: Special metrics and triality, \textit{Advances Math.}, \textbf{219} (2008), 1972--2005.

\bibitem{Wang} M. Wang: Preserving parallel spinors under metric deformations, \textit{Indiana Univ. Math. J.} \textbf{40} (1991), 815--844, \doi{10.1512/iumj.1991.40.40037}.

\bibitem{BG} J. Bourguignon and P. Gauduchon: Spineurs, op\'erateurs de Dirac et variations de m\'etriques, \textit{Comm. Math. Phys.} \textbf{144} (1992), 581--599, \doi{10.1007/BF02099184}.

\bibitem{Slebarski} S. Slebarski: Dirac operators on a compact Lie group, \textit{Bull. London Math. Soc.} \textbf{17} (1985), 579--583.

\bibitem{Fegan} H. Fegan: The spectrum of the Dirac operator on a simply connected compact Lie group, \textit{Simon Stevin} \textbf{61} (1987), 97--108.

\bibitem{LT} A. Lai and K. Teh: Spectral action for a one-parameter family of Dirac operators on SU(2) and SU(3), \textit{J. Math. Phys.} \textbf{54} 022302 (2013).

\bibitem{Bap1} J. Baptista: Internal symmetries in Kaluza-Klein models,  \textit{J. High Energy Phys.} \textbf{2024} (2024), 178, \doi{10.1007/JHEP05(2024)178}.


\bibitem{Bap3} J. Baptista: Chiral interactions of fermions and massive gauge fields in Kaluza-Klein models, \textit{J. High Energy Phys.} \textbf{2026} (2026), 008, \doi{10.1007/JHEP05(2026)008}.

\bibitem{DVT} M. Dubois-Violette and I. Todorov: Exceptional quantum geometry and particle physics II, \textit{Nucl. Phys. B} \textbf{938} (2019), 751--761.

\bibitem{FureyHughes} C. Furey and M. Hughes: Three generations and a trio of trialities, \textit{Phys. Lett. B} \textbf{865} (2025), 139473.

\bibitem{Boyle} L. Boyle: The standard model, the exceptional Jordan algebra, and triality, \textit{J. Math. Phys.} \textbf{67} (2026), 071701.

\bibitem{Okubo78} S. Okubo: Pseudo-quaternion and pseudo-octonion algebras, \textit{Hadronic Journal} \textbf{1} (1978), 1250--1278.

\bibitem{OkuboBook} S. Okubo: \textit{Introduction to octonion and other non-associative algebras in physics}, Cambridge Univ. Press, 1995.

\bibitem{Elduque} A. Elduque: Composition Algebras, in:  \textit{Algebra and Applications 1} (2021), \doi{10.1002/9781119818175.ch2}.

\bibitem{FS79} T. Friedrich and S. Sulanke: Ein Kriterium f\"ur die formale Selbstadjungiertheit des Dirac-Operators, \textit{Coll. Mathematicae} \textbf{40} (1979), 239--247.

\bibitem{Agricola} I. Agricola, The Srní lectures on non-integrable geometries with torsion, \textit{Arch. Math.} \textbf{42} (2006).

\bibitem{Friedrich} T. Friedrich: Der erste Eigenwert des Dirac-Operators einer kompakten, Riemannschen Mannigfaltigkeit nichtnegativer Skalarkr\"ummung, \textit{Math. Nachrichten } \textbf{97} (1980), 117--146.

\bibitem{Puhle} C. Puhle: Riemannian manifolds with structure group PSU(3), \textit{J. London Math. Soc.} (2) \textbf{85} (2012), 79--100.

\bibitem{BGM} C. B\"ar, P. Gauduchon and A. Moroianu: Generalized cylinders in semi-Riemannian and spin geometry, \textit{Math. Zeitschrift} \textbf{249} (2005), 545--580, \doi{10.1007/s00209-004-0718-0}.

\bibitem{AHermann} A. Hermann: \textit{Dirac eigenspinors for generic metrics}, PhD thesis, Univ. Regensburg, 2012, \doi{10.5283/EPUB.25024}.

\bibitem{Bap4} J. Baptista: The geometry of CP symmetry in Kaluza-Klein models, \href{https://arxiv.org/abs/2601.08902}{\tt arXiv:2601.08902 [hep-th]}

\bibitem{Bourguignon} J. Bourguignon, O. Hijazi, J. Milhorat, A. Moroianu and S. Moroianu: \textit{A spinorial approach to Riemannian and conformal geometry}, European Math. Society, 2015, \doi{10.4171/136}.

\bibitem{LM} H. Lawson and M. Michelsohn: \textit{Spin geometry}, Princeton Univ. Press, 1989.

\bibitem{FI} T. Friedrich and S. Ivanov: Parallel spinors and connections with skew-symmetric torsion in string theory, \textit{Asian J. Math.} \textbf{6} (2002), 303--336.

\bibitem{Agricola2} I. Agricola, Connections on naturally reductive spaces, their Dirac operator and homogeneous models in string theory, \textit{Commun. Math. Phys.} \textbf{232} (2003), 535--563.

\bibitem{AgricolaF} I. Agricola and T. Friedrich, On the holonomy of connections with skew-symmetric torsion, \textit{Math. Ann.} \textbf{328} (2004), 711--748.

\bibitem{Ferreira} A. Ferreira: Riemannian geometry with skew torsion, PhD thesis, Univ. Oxford, 2010.

\bibitem{Bleecker} D. Bleecker: \textit{Gauge theory and variational principles}, Addison-Wesley, 1981.

\bibitem{CS} S. Chern and J. Simons: Characteristic forms and geometric invariants, \textit{Ann. of Math.} \textbf{99} (1974), 48--69.

\bibitem{Freed} D. Freed: Classical Chern-Simons theory, Part 2, \textit{Houston J. Math.} \textbf{28} (2002), 293--310.

\bibitem{Kosmann} Y. Kosmann: D\'eriv\'ees de Lie des spineurs, \textit{Ann. di Matematica Pura ed Applicata} \textbf{91} (1971), 317--395, \doi{10.1007/BF02428822}.

\bibitem{Lich} A. Lichnerowicz: Spineurs harmoniques, \textit{Comptes rendus Acad. Sc. Paris, groupe 1} \textbf{257} (1963), 7--9.

\bibitem{OT} T. Ochiai and T. Takahashi: The group of isometries of a left invariant Riemannian metric on a Lie group, \textit{Math. Ann.} \textbf{223} (1976), 91--96, \doi{10.1007/BF01360280}.

\bibitem{GoKo} S. Goldberg and S. Kobayashi: The conformal transformation group of a compact homogeneous Riemannian manifold,  \textit{Bull. Amer. Math. Soc.} \textbf{68} (1962), 378--381.

\bibitem{Hitchin2} N. Hitchin: Harmonic spinors, \textit{Advances Math.} \textbf{14} (1974), 1--55.

\bibitem{Bar92} C. B\"ar: The Dirac operator on homogeneous spaces and its spectrum on 3-dimensional lens spaces, \textit{Arch. Math.} \textbf{59} (1992), 65--79.

\bibitem{Ginoux} N. Ginoux: \textit{The Dirac spectrum}, Springer, 2009, \doi{10.1007/978-3-642-01570-0}.

\bibitem{Bar96} C. B\"ar: The Dirac operator on space forms of positive curvature, \textit{J. Math. Soc. Japan} \textbf{48} (1996), 69--83.

\bibitem{AB} B. Ammann and C. B\"ar: The Dirac operator on nilmanifolds and collapsing circle bundles, \textit{Ann. Global Anal. Geom.} \textbf{16} (1998), 221--253.



\end{thebibliography}
\end{document}